\documentclass[11pt,reqno]{amsart}
\usepackage{amsmath,amssymb,amsthm, comment,graphicx}
\usepackage{tikz}
\newtheorem{theorem}{Theorem}[section]
\newtheorem{definition}{Definition}[section]
\newtheorem{lemma}{Lemma}[section]
\newtheorem{remark}{Remark}[section]
\newtheorem{proposition}{Proposition}[section]
\newtheorem{corollary}{Corollary}[section]

\numberwithin{equation}{section}

\newcommand{\R}{{\mathbb R}}

\newcommand{\rr}{\mathbf{r}}

\begin{document}
\title[Steady Supersonic Exothermically Reacting Euler flow past Bending Walls]
{Two-Dimensional Steady Supersonic Exothermically Reacting Euler Flow past Lipschitz Bending Walls}
\author{Gui-Qiang G. Chen}
\address{Mathematical Institute,\
 University of Oxford, Oxford, OX2 6GG, UK}
\email{\tt Gui-Qiang.Chen@maths.ox.ac.uk}
\author{Jie Kuang}
\address{School of Mathematical Sciences, \
 Fudan University, Shanghai 200433;
Institute of Applied Mathematics,
Academy of Mathematics and Systems Science,
Chinese Academy of Sciences, Beijing 100190, P. R. China}
\email{\tt jkuang12@fudan.edu.cn; jkuang@amss.ac.cn}
\author{Yongqian Zhang}
\address{School of Mathematical Sciences, \
 Fudan University, Shanghai 200433, P. R. China}
\email{\tt yongqianz@fudan.edu.cn}

\date{\today}

\begin{abstract}
We are concerned with the two-dimensional steady supersonic reacting Euler flow past
Lipschitz bending walls that are small perturbations of a convex one,
and establish the existence of global entropy solutions when the total variation of
both the initial data and the slope of the boundary is sufficiently small.
The flow is governed by an ideal polytropic gas and undergoes a one-step exothermic
chemical reaction under the reaction rate function that is Lipschtiz and
has a positive lower bound.
The heat released by the reaction may cause the total variation
of the solution to increase along the flow direction.
We employ the modified wave-front tracking scheme to construct approximate solutions
and develop a Glimm-type functional by incorporating the approximate
strong rarefaction waves and
Lipschitz bending walls to obtain the uniform bound on the total variation
of the approximate solutions.
Then we employ this bound to prove the convergence of the approximate solutions to a global entropy solution
that contains a strong rarefaction wave generated by the Lipschitz bending wall.
In addition, the asymptotic behavior of the entropy solution in the flow direction
is also analyzed.
\end{abstract}

\maketitle
\section{Introduction }\setcounter{equation}{0}
We are concerned with the problem of the two-dimensional steady  supersonic
exothermically reacting Euler flow past Lipschitz bending walls that are small perturbations
of a convex one (see Fig. 1.1).
The governing system for steady exothermically reacting flow consists of the Euler equations
with the following form:
\begin{eqnarray}\label{eq:1.1}
\left\{
\begin{array}{llll}
     \partial_x(\rho u)+\partial_y(\rho v)=0, \\[5pt]
     \partial_{x}(\rho u^{2}+p)+\partial_{y}(\rho uv)=0, \\[5pt]
   \partial_{x}(\rho uv)+\partial_{y}(\rho v^{2}+p)=0, \\[5pt]
     \partial_{x}\big((\rho E+p)u\big)+\partial_{y}\big((\rho E+p)v\big)=0,\\[5pt]
     \partial_{x}(\rho u Z)+\partial_{y}(\rho v Z)=-\rho Z\phi(T),
     \end{array}
     \right.
\end{eqnarray}
where $(u,v), p, \rho, Z$, and $\phi(T)$  stand for the velocity, pressure, density,
fraction of unburned gas in the mixture, and reaction rate, respectively, and
\begin{equation*}
E=\frac{1}{2}(u^2+v^2)+e+\tilde{q}Z
\end{equation*}
denotes the specific total energy with the specific internal energy $e$, and
$\tilde{q}$ is the specific binding energy of unburned gas.
The other two thermodynamic variables are the temperature $T$
and entropy $S$ which are defined through thermodynamical relations:
\begin{equation}
TdS=de-\frac{p}{\rho^{2}}d\rho.\label{eq:1.2}
\end{equation}
Then the pressure $p$ and internal energy $e$ can be regarded as functions of $(\rho, S)$:
$$
p=p(\rho, S), \qquad e=e(\rho, S).
$$
In particular, $\partial_{\rho}p(\rho, S)>0$ and $\partial_{\rho}e(\rho, S)>0$ for $\rho>0$,
and $c=\sqrt{\partial_{\rho}p(\rho, S)}$ is called the local sound speed.

For an ideal polytropic gas, the constitutive relations are given by
\begin{equation}
p=R\rho T,\quad e=c_vT,\quad \gamma=1+\frac {R}{c_v}>1,\label{eq:1.3}
\end{equation}
and
\begin{equation}
p=p(\rho, S)=\kappa\rho^{\gamma}e^{S/c_v}, \quad
e(\rho,S)=\frac{\kappa}{\gamma-1}\rho^{\gamma-1}e^{S/c_v}=\frac{RT}{\gamma-1},\label{eq:1.4}
\end{equation}
where $R,\ \kappa$, $c_v$, and $\gamma>1$ are all positive constants.
Then the sonic speed is given by
$c=\sqrt{\frac{\gamma p}{\rho}}$.
\vspace{10pt}
\begin{center}
\begin{tikzpicture}[scale=1.8]
\draw [thick][<->] (0,1.9) --(0,0)-- (2.2,0);
\draw [thick][->] (-1.6,1.2)to [out=30, in=-150](-0.3,1.2);
\draw [thick][->] (-1.6,0.8)to [out=30, in=-160](-0.3,0.8);
\draw [thick][->] (-1.6,0.4)to [out=30, in=-140](-0.3,0.4);
\draw [thick][<-] (-0.1,-0.9) --(1,-0.4);
\draw [thick][->] (2.3,0.1) to[out=-30, in=130](3.4,-0.35);
\draw [line width=0.1cm](-1.5,0)to [out=0, in=0](0, 0)to [out=0,in=145](2.6,-1.1);
\draw [thick](0,0)to [out=45, in=200](1.5, 1.8);
\draw  [ thick](0.6,-0.06)to [out=45, in=200](2, 1.5);
\draw  [ thick](1.1,-0.18)to [out=45, in=200](2.7, 1.3);
\draw  [ thick](1.6,-0.45)to [out=45, in=200](3.3, 1);
\node at (2.2,-0.15) {$x$};
\node at (-0.15, 1.85) {$y$};
\node at (-0.1, -0.15) {$O$};
\node at (-1, 1.5) {$\bar{U}(y)$};
\node at (3.0, 0.1) {$U(x,y)$};
\node at (0.1, -1) {$y=g(x)$};
\node [below] at (0.3, -1.4)
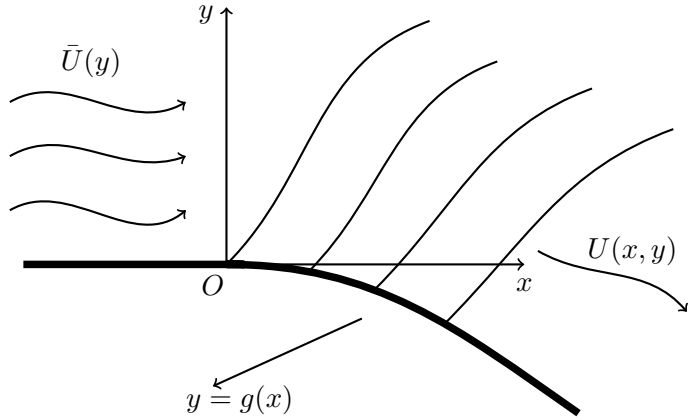
{Fig. 1.1. Supersonic reacting flow past a Lipschitz bending wall};
\end{tikzpicture}
\end{center}

In addition, inadmissible discontinuous solutions are eliminated by
requiring the following entropy condition for the solutions:
\begin{equation}\label{eq:entropy}
\partial_x(\rho uS)+ \partial_y(\rho vS)\ge \frac{\tilde{q}\rho Z\phi(T)}{T}
\end{equation}
in the distributional sense.

System \eqref{eq:1.1} for the exothermically reacting steady Euler flow
can be written in the following general form:
\begin{eqnarray}\label{eq:2.1}
W(U)_{x}+H(U)_{y}=G(U)
\end{eqnarray}
with $U=(u,v,p,\rho,Z)^{\top}$, where
\begin{eqnarray*}
\begin{split}
 W(U)&=\Big(\rho u, \rho u^2+p, \rho uv, \rho u(\tilde{h}+\frac{u^2+v^2}{2}), \rho uZ\Big)^{\top},\\
H(U)&=\Big(\rho v, \rho uv,\rho v^2+p, \rho v(\tilde{h}+\frac{u^2+v^2}{2}), \rho vZ\Big)^{\top},\\
G(U)&=\Big(0, 0, 0, \tilde{q}\rho \phi(T)Z, -\rho \phi(T)Z\Big)^{\top},
\end{split}
\end{eqnarray*}
with $\tilde{h}=\frac {\gamma p}{(\gamma -1)\rho}$.
When $\rho>0$ and $u>c$, $U$ can also be represented by $W$, that is, $U=U(W)$,
by the implicit function theorem since the Jacobian does not vanish:
$$
{\rm det}\big(\nabla_U W(U)\big)=-\frac{\rho u^2}{\gamma-1}(u^2-c^2)\ne 0.
$$

Throughout this paper, we assume the following:
\begin{enumerate}
\item[$\mathbf{(H_{1})}$]  The Lipschitz continuous bending wall $y=g(x)$  is a small perturbation of the
convex wall $y=g_{*}(x)$ for $x\ge 0$ with
\begin{equation*}
g(x)=0 \,\,\,\,\,\mbox{for $x\leq 0$}, \qquad   g'_{+}(x)\in \textit{BV}(\mathbb{R}_{+};\mathbb{R}),
\end{equation*}
and, for some small $\varepsilon>0$,
\begin{eqnarray*}
  \|g'_+(x)-g'_*(x)\|_{BV(\R_+)}\le \varepsilon,
\end{eqnarray*}
where
\begin{equation*}
g'_{+}(x)=g'(x+)=\lim_{\hat{x}\rightarrow x+}\frac{g(\hat{x})-g(x)}{\hat{x}-x};
\end{equation*}

\item[$\mathbf{(H_{2})}$]
The incoming flow $\bar{U}(y)=(\bar{u}, \bar{v}, \bar{p}, \bar{\rho}, \bar{Z})(y)$ at $x=0$
with bounded total variation
is a small perturbation of the constant state
$U_{\infty}=(u_{\infty}, 0, p_{\infty}, \rho_{\infty}, 0)$  and satisfies
\begin{equation}\label{eq:1.5}
\bar{u}^{2}+\bar{v}^{2}>\bar{c}^{2}, \qquad
0\leq\bar{Z}\leq 1,\qquad \bar{Z}(\infty)=0,
\end{equation}
where $u_{\infty}>c_{\infty}=\sqrt{\frac{\gamma p_{\infty}}{\rho_{\infty}}}$.
In addition, there exists a positive constant $T_{*}$ such that
\begin{equation}
 \bar{T}(y)>T_{*}>0. \label{eq:1.7}
\end{equation}
\end{enumerate}

Assumption \eqref{eq:1.7} in $\mathbf{(H_{2})}$  is to make sure that the reaction rate function $\phi(T)$
has a positive minimum value
$L_{*}=\phi(T_{*})$
which never vanishes.
In a sense, this is a very realistic condition. Typically, $\phi (T)$ has the Arrhenius form as in \cite{cd}:
\begin{equation}
\phi(T)=T^{\alpha}e^{-\frac{\mathcal{E}}{RT}},\label{eq:1.8a}
\end{equation}
which vanishes only at absolutely zero temperature, where
$\mathcal{E}$ is the action energy and $\alpha$ is a positive constant.

For the given bending wall $y=g(x)$, the domain and its boundary are defined as
$$
\Omega=\{(x,y):x>0,\ y>g(x)\},\qquad \Gamma=\{(x,y): x>0,\ y=g(x)\},
$$
and
$$
\textbf{n}=\textbf{n}(x, g(x))=\frac{(g'(x),-1)}{\sqrt{1+(g'(x))^2}}
$$
is the outer normal vector to $\Gamma$ at point $x$ except the non-differential point.
Regarding $x$ as a time-like variable, the planar flow problem can be formulated as the
initial-boundary problem for system \eqref{eq:1.1} with

\medskip
$\mathbf{Cauchy}$ $\mathbf{condition}$:
\begin{equation}
U\big|_{x=0}=\bar{U}(y);   \label{eq:1.8}
\end{equation}
\par$\mathbf{Boundary}$ $\mathbf{condition}$:
\begin{equation}
(u,v)\cdot\textbf{n}\big|_{\Gamma}=0. \label{eq:1.9}
\end{equation}

\begin{definition}[Entropy solutions]\label{def:1.1}
A function $U\in BV_{loc}(\Omega)$ is called an entropy solution to
the initial-boundary value problem \eqref{eq:1.1} and \eqref{eq:1.8}--\eqref{eq:1.9} in
$\Omega\subset \R^2_+$  provided that,
for any convex entropy pair $(\eta,q)$ with respect to $W=W(U)$ of \eqref{eq:1.1}, that is,
$\nabla^2\eta(W)\ge 0$ and $\nabla q(W)=\nabla\eta(W)\nabla H(U(W))$,
the following entropy inequality holds{\rm :} For any $\psi \in C_0^{\infty}(\R^2)$ with $\psi\ge 0$,
\begin{eqnarray}\label{eq:1.11}
&&\iint_{\Omega}\big(\eta(W(U))\psi_x+q(W(U))\psi_y+\nabla_W\eta(W(U))G(U)\psi\big)dxdy\nonumber\\
&& +\int_{0}^{\infty}\eta(W(\bar{U}(y)))\psi(0,y)\, dy\ge 0.
\end{eqnarray}
\end{definition}

The entropy inequality \eqref{eq:1.11} directly implies that $U=U(x,y)$ is a weak solution of system \eqref{eq:1.1} with \eqref{eq:1.8}--\eqref{eq:1.9}:
\begin{eqnarray}
\iint_{\Omega}
\big(W(U)\phi_x+H(U)\phi_y+G(U)\phi\big) dxdy+\int_{0}^{\infty}W(\bar{U}(y))\phi(0,y) dy= 0 
\label{eq:1.10}
\end{eqnarray}
for any $\phi \in C_0^{\infty}(\R^2)$.  This can be seen by choosing $\eta(W)=\pm W$.

Moreover, $\eta(W)= -\rho uS$ is an entropy which is convex with respect to
$W$, while $q(W) = -\rho vS$ is the corresponding entropy flux, when $\rho>0$ and $u>c$,
so that the entropy inequality \eqref{eq:1.11} also implies
the physical entropy condition \eqref{eq:entropy}.

\medskip
For the non-reacting steady Euler system with supersonic state and certain physical boundaries,
some analysis on the shocks and rarefaction waves has been made.
For example, as described in \cite{courant},
when a supersonic flow hits a sharp body or moves around a sharp corner,
a supersonic shock is formed and attached to the body,
or a rarefaction wave is generated by the corner.
Such a physical phenomenon has been extensively studied
for the Lipschitz boundaries or smooth boundaries that are
small perturbations of a straight one.
For instance, Zhang \cite{zh1}--\cite{zh4}
considered the two-dimensional steady supersonic potential flow past Lipschitz wedges
or over bending walls
and obtained the global existence and asymptotic behavior of entropy solutions in $BV$
which contain a strong shock  with large vertex angle or a strong rarefaction wave.
Later, similar results have been obtained for the full Euler equations
for entropy solutions in $BV$ which contain a strong shock in Chen-Zhang-Zhu \cite{czz}.
Moreover, based on \cite{czz}, the stability and uniqueness of entropy
solutions containing a strong shock by the wave-front tracking algorithm
were established in Chen-Li \cite{cl}.
For the space dimension higher than two, the global existence of
weak solutions for steady supersonic conical flow
has been analyzed first in Lien-Liu \cite{lien} for isentropic Euler flow
under the assumptions that the symmetrical cone has a small opening angle and
the initial strength of the relatively strong shock is sufficiently
weak, and has been studied then in Wang-Zhang  \cite{wz} for steady potential flow past
the cone with an arbitrary opening angle that is less than a critical value.
On the other hand, the local/global existence
for steady supersonic flow past cones or sharp corners
with smooth boundaries has also been studied extensively
in \cite{cy,chen,chenli,cxy,wy} and the references cited therein.

For the exothermically reacting Euler equations,
the global existence of entropy solutions for the Cauchy problem
was first established for
the one-dimensional case that the reaction rate function is Lipschtiz
and has a lower bound
in Chen-Wagner \cite{cd} when the total variation of
the initial data is bounded by
a constant proportional to a parameter $\frac{1}{\gamma-1}$,
by developing the Glimm-type fractional-step scheme.
In Chen-Xiao-Zhang \cite{cxz},
the Cauchy problem for the two-dimensional steady case has first been
studied, and the initial-boundary value problem for supersonic reacting flow
past a Lipschitz wedge
with large or small angle has then been analyzed:
When the total variation of both the initial data and the slope of
the wedge boundary is suitably small,
the global existence and asymptotic behavior of the entropy
solutions have been established.
For the multidimensional case,
we also refer the reader to \cite{cjlw} for the details.

When the reaction rate function is discontinuous,
we refer to \cite{lya,sz,tcl,Wagner,yt,zz} and the references therein
for the Riemann problem for the one-dimensional case.
For further information on this topic and related combustion theories,
we refer the reader to \cite{Fickett, Neu, rm, Williams}.

In this paper, we establish the global existence and asymptotic behavior of
entropy solutions for two-dimensional steady supersonic inviscid
reacting flow over a Lipschitz bending wall.
Our problem is different from \cite{zh1} that has been done for the potential flow.
The flow here is described by five equations with the reaction source terms, which
may be viewed as the full Euler equations coupled with a
nonhomogeneous transport equation
owing to the reacting process.
Thus the problem we consider can be formulated as an initial-boundary value problem
for a hyperbolic system of balance laws, which involves a strong rarefaction wave.
One of our main motivations is the mathematical difficulty that the heat released
by the reaction may cause the total variation of the solution to increase along
the flow direction, even in the one-dimensional case ({\it cf.} \cite{BMR,cd,Er,FW,LSS,OS}).
One of our main results in this paper is to prove that
the increase in the total variation of the solution is eventually bounded
as a result of the uniform and exponential decay of the reactant along the flow direction.

To treat this nonlinear problem with the strong rarefaction wave and reaction source terms,
one natural way is to combine the wave-front tracking algorithm with
the fractional-step technique, since it is more convenient
to determine the position and
control the strength of the strong rarefaction wave in the construction
of approximate solutions.
More precisely, to achieve this, we proceed in the following two steps:

\smallskip
We first study the homogeneous system by approximating the Lipschitz boundary with polylines
and employing the ideas developed in Amadori \cite{amadori} and Bressan \cite{bressan}
to construct approximate solutions $U^{\nu, h}$ of the initial-boundary value problem
in each interval $((k-1)h,kh),\ k\in\mathbb{N}_{+}$.
In this construction, our key observation is that the discontinuities of the 5th component $Z(x,y)$
propagate only along the $4$--contact discontinuities and is unchanged when they cross the
other families of waves, especially for the non-physical waves.
In addition, in dealing with the change in the strength of weak waves
after they interact with strong rarefaction fronts,
we impose the weights $W(\alpha_{i}, x, -), 1\leq i\leq4$,
and $W(\epsilon, x, -)$ for weak waves and introduce
a functional $F_{1}(U; x)$ (see \S 5 below).

The second step is to consider the reaction process that is 
related by \eqref{eq:4.27-1}--\eqref{eq:4.27-5}
with the reaction step $h$.  To do this, we employ the accurate Riemann solver
to solve the Riemann problem with the Riemann data
on the physical waves for the reaction step,
and let the non-physical waves across line $x=kh$ directly when the reaction
step occurs on it.
We also remark that the orders of strong rarefaction fronts
and the non-physical fronts are unchanged.
However, in the proof of the convergence and consistency of the solutions,
we first fix $h$ to take the limits for $\nu\rightarrow \infty$
and then let $h\rightarrow 0$.
We do this in order to avoid
the wave-fronts that may increase owing to the reaction process.
Similar ideas have also been used in Amadori-Gosse-Guerra \cite{agg}.
Finally, with the estimates on the wave interactions,
we can first obtain the {\it a priori} bound
on the total variation of the approximate solutions
and then, by carrying out the steps as in \cite{agg,glimm},
we obtain the global existence of entropy solutions for problem
\eqref{eq:1.1} and \eqref{eq:1.8}--\eqref{eq:1.9}.

\smallskip
For the asymptotic behavior of entropy solutions,
we need further estimates on the approximate solutions $U^{\nu,h}$.
The key step for the estimates required here is when
the $1$--generalized characteristics intersect with the boundary
after they cross the rarefaction area.
This is different from the wedge case that has been handled in \cite{cxz}.
Then, by applying the Glimm-Lax theory \cite{gl},
one can derive the asymptotic behavior of entropy solutions.

Before concluding this section, we remark in passing that
the global existence for the Cauchy problem of
the one-dimensional hyperbolic systems of conservation laws
with the initial data containing strong rarefaction waves
has been considered by Lewicka in \cite{Le4, Le5}.
The problem we are considering here
is different from \cite{Le4,Le5},
since our problem is of initial-boundary value type for
the hyperbolic systems of balance laws including
the reaction source terms.
For the existence and $L^{1}$--stability for the Cauchy problem
containing weak elementary waves ({\it i.e.},
shocks, rarefaction waves, and contact discontinuities)
for one-dimensional hyperbolic systems, we refer the reader
to \cite{Le3,sch} and the references therein.
Also see Lewicka-Trivisa \cite{LT} for the $L^1$ well-posedness
of the one-dimensional hyperbolic systems of conservation laws
near solutions containing two large shocks.

The rest of this paper is organized as follows:
In \S 2, we present the nonlinear waves and the Riemann solutions
for the homogeneous equation \eqref{eq:2.1}
in the supersonic region.
In \S 3, we consider the background solution for the homogeneous system
as supersonic Euler flow ({\it i.e.}, $Z=0$)
past the convex Lipschitz bending wall $y=g_{*}(x)$.
In \S 4, the approximate solutions
for the initial-boundary value problem
\eqref{eq:1.1} and \eqref{eq:1.8}--\eqref{eq:1.9}
are constructed by developing a fractional-step wave-front tracking algorithm.
In \S 5, some weights are introduced to construct the Glimm-type functional.
In \S 6, we consider various types of wave interactions between weak waves
and strong rarefaction wave-fronts.
Then we establish the uniform bound on the total variation
of the approximate solutions for the homogeneous system ({\it i.e.},  \eqref{eq:2.2}).
In \S 7, we study the $BV$--stability of the approximate solutions
for the reacting step and establish
the uniform bound on the total variation of the approximate solutions.
In \S 8, the convergence and consistency of the approximate solutions are established
and then the existence theorem is stated.
In \S 9, we study the asymptotic behavior of entropy solutions.
Finally, we give a detail proof of Lemma \ref{lem:2.1} in \S 10.

We also remark that, throughout this paper, $O(1)$ stands for the bounded
quantities that depend only on the system.

\section{Riemann Problem}
\setcounter{equation}{0}

In this section, we present some basic properties
of the homogeneous system of \eqref{eq:1.1}.

\subsection{Homogeneous system}

In the case when $G(U)$ is identically zero, system \eqref{eq:2.1} becomes
\begin{eqnarray}\label{eq:2.2}
W(U)_x+H(U)_y=0.
\end{eqnarray}
With $x$ as the time-like variable, system \eqref{eq:2.2} is hyperbolic for $u>c$.
This system has five eigenvalues:
\begin{eqnarray} \label{eq:2.3}
&&\lambda_j=\frac{uv+(-1)^{\frac{j+3}{4}} c\sqrt{u^2+v^2-c^2}}{u^2-c^2}, \qquad j=1,5,\\[4pt]
&&\lambda_i=\frac{v}{u}, \qquad i=2, 3, 4,
\end{eqnarray}
and corresponding five linearly independent eigenvectors:
\begin{eqnarray}\label{eq:2.4}
&&\tilde{\rr}_j=(-\lambda_j, 1, \rho(\lambda_j u-v), \frac{\rho (\lambda_ju-v)}{c^2},  0)^\top, \qquad j=1,5,\\[4pt]
&&\tilde{\rr}_2=(u,v,0,0,0)^\top, \quad \tilde{\rr}_3=(0,0,0,1,0)^\top, \quad \tilde{\rr}_4=(0,0,0,0,1)^\top.
\end{eqnarray}

Let
\begin{equation*}
q=\sqrt{u^{2}+v^{2}},\quad  \theta=\arctan(\frac{v}{u}),\quad
\theta_{\rm ma}=\arctan(\frac{c}{\sqrt{q^{2}-c^{2}}}),
\end{equation*}
where $\theta_{\rm ma}$ is the Mach angle.
Then
\begin{eqnarray*}
\lambda_{j}=\tan(\theta+(-1)^{\frac{j+3}{4}}\theta_{\rm ma}), \,\,\,\, j=1, 5; \qquad
\lambda_{i}=\tan\theta, \,\,\,\, i=2, 3, 4.
\end{eqnarray*}
With the state variables $(q, \theta, p, \rho, Z)^\top$, we can rewrite the eigenvectors as
\begin{eqnarray*}
&&\tilde{\rr}_{1}=\big(-\tan(\theta-\theta_{\rm ma}), 1, -\rho q\sec(\theta-\theta_{\rm ma})\sin(\theta_{\rm ma}),
 -\frac{\rho q}{c^{2}}\sec(\theta-\theta_{\rm ma})\sin(\theta_{\rm ma}), 0\big)^{\top},\\[4pt]
&&\tilde{\rr}_{5}=\big(-\tan(\theta+\theta_{\rm ma}), 1, \rho q\sec(\theta+\theta_{\rm ma})\sin(\theta_{\rm ma}),
 \frac{\rho q}{c^{2}}\sec(\theta+\theta_{\rm ma})\sin(\theta_{\rm ma}), 0\big)^{\top},\\[4pt]
&&\tilde{\rr}_{2}=\big(\cos\theta, \sin\theta, 0, 0, 0 )^{\top},\qquad
\tilde{\rr}_{3}=\big(0, 0, 0, 1, 0\big)^{\top},\qquad
\tilde{\rr}_{4}=\big(0, 0, 0, 0,1\big)^{\top}.
\end{eqnarray*}
Moreover, we have the following properties whose proofs are given in \S 10.
\begin{lemma}\label{lem:2.1}
For $u>c$,
\begin{equation}\label{eq:2.5}
\nabla_{U}\lambda_{j}\cdot\tilde{\rr}_j>0, \,\,\,\, j=1, 5; \qquad
\nabla_{U}\lambda_{i}\cdot\tilde{\rr}_i=0, \,\,\,\, i=2, 3, 4.
\end{equation}
\end{lemma}

Lemma \ref{lem:2.1} implies that the $j$-th characteristic fields are genuinely nonlinear for $j=1, 5$,
while the $i$-th characteristic fields, $i=2,3,4$, are linearly degenerate.
Therefore, we can choose $\rr_j=\kappa_j\tilde{r}_j$ with $\kappa_j=\frac{1}{\nabla_{U}\lambda_{j}\cdot\tilde{\rr}_j}$, $j=1,5$,
and $\rr_i=\tilde{\rr}_i, i=2,3,4$, such that
\begin{eqnarray*}
\nabla_{U}\lambda_j\cdot \rr_j=1,\qquad j=1, 5.
\end{eqnarray*}

\subsection{Nonlinear waves for system \eqref{eq:2.2}}
In this subsection, we discuss nonlinear wave curves in the phase-space
for system \eqref{eq:2.2}.
As indicated in \S 2.1, system \eqref{eq:2.2}
admits five elementary waves that belong to the corresponding characteristic families for $u>c$.

The Rankine-Hugoniot conditions for system \eqref{eq:2.2} with the discontinuity speed $s$ are
\begin{eqnarray}\label{eq:2.8}
s[W(U)]=[H(U)],
\end{eqnarray}
where $[W(U)]=W(U_{+})-W(U_{-})$, and $U_{+}$ and $U_{-}$ are the {\it above} and {\it below}
states of the discontinuity, respectively.

The $i$-contact wave curves $C_{i}(U_{0})$, $i=2,3,4$, through $U_{0}$
are
\begin{eqnarray}\label{eq:2.7}
\begin{split}
C_{i}(U_{0}):\,\, p=p_{0},\  s_i=\frac{v}{u}=\frac{v_{0}}{u_{0}}.
\end{split}
\end{eqnarray}
That is,
we can parameterize $C_{i}(U_{0})$ as
\begin{eqnarray}\label{eq:2.7a}
&&C_{2}(U_{0}): \,\, u=u_0e^{\sigma_2},\ v=v_0e^{\sigma_2},\ p=p_{0},\ \rho=\rho_0,\ Z=Z_{0},\\[4pt]
&&C_{3}(U_{0}): \,\, u=u_0,\ v=v_0,\ p=p_{0},\ \rho=\rho_0+\sigma_3,\ Z=Z_{0},\\[4pt]
&&C_{4}(U_{0}): \,\, u=u_0,\ v=v_0,\ p=p_{0},\ \rho=\rho_0,\ Z=Z_{0}+\sigma_4,
\end{eqnarray}
where $C_2(U_0)$ corresponds the compressible vortex sheet, $C_3(U_0)$ the entropy wave,
and $C_4(U_0)$ the reactant wave.

\smallskip
The $j$-th shock wave curves $S_{j}(U_0),  j=1,5$, through $U_{0}$ are
\begin{eqnarray}\label{eq:2.9}
S^{-}_j(U_0):\, \, [p]=\frac{c_0^2}{\tilde{b}}[\rho],\ [u]=-s_j[v], \ [p]=\rho_0(s_ju_0-v_0)[v],\, [Z]=0,\quad  \rho>\rho_0,
\end{eqnarray}
where $\tilde{b}=\frac{\gamma+1}{2}-\frac{\gamma-1}{2}\frac{\rho}{\rho_0}$,\
$\tilde{c}^2=\frac{c_0^2}{\tilde{b}}\frac{\rho}{\rho_0}$, and
\begin{eqnarray}\label{eq:2.10}
\begin{split}
s_j=\frac{u_0v_0+(-1)^{\frac{j+3}{4}}\tilde{c}\sqrt{u_0^2+v_0^2-\tilde{c}^2}}{u_0^2-\tilde{c}^2},\qquad j=1,5.
\end{split}
\end{eqnarray}

The rarefaction wave curves $R_{j}(U_{0}), j=1, 5$, in the phase-space through $U_{0}$ are given by
\begin{eqnarray}\label{eq:2.11}
\begin{split}
 R^{+}_j(U_0):\,\, dp=c^2d\rho,\,\, du=-\lambda_jdv,\,\, \rho(\lambda_ju-v)dv=dp,\,\, dZ=0,\qquad  \rho<\rho_0.
\end{split}
\end{eqnarray}
In the physical states $(q, \theta, p, \rho, Z)^\top$, $R_{j}(U_{0}), j=1,5$, can also be expressed as
\begin{eqnarray}
&&R^{+}_1(U_0)\,:\,J(q,\mathcal{B})-\theta=J(q_{0},\mathcal{B}_0 )-\theta_{0},\, \frac{q^{2}}{2}+\frac{c^{2}}{\gamma-1}
=\frac{q^{2}_{0}}{2}+\frac{c^{2}_{0}}{\gamma-1},\, \frac{p}{\rho^{\gamma}}=\frac{p_{0}}{\rho^{\gamma}_{0}},\, Z=Z_{0},\nonumber\\
 &&        \label{eq:2.12}\\
&&R^{+}_5(U_0)\,:\,J(q,\mathcal{B})+\theta=J(q_{0},\mathcal{B}_0  )+\theta_{0}, \,\frac{q^{2}}{2}+\frac{c^{2}}{\gamma-1}
=\frac{q^{2}_{0}}{2}+\frac{c^{2}_{0}}{\gamma-1},\, \frac{p}{\rho^{\gamma}}=\frac{p_{0}}{\rho^{\gamma}_{0}},\, Z=Z_{0}, \nonumber\\
&& \label{eq:2.13}
\end{eqnarray}
where $U=(u, v, p, \rho, Z)^\top$,\ $U_{0}=(u_{0}, v_{0}, p_{0}, \rho_{0}, Z_{0})^\top$,\
$q_{0}=\sqrt{u^{2}_{0}+v^{2}_{0}},\  \theta_{0}=\arctan(\frac{v_{0}}{u_{0}})$, and
\begin{eqnarray*}
J(q,\mathcal{B}):=\int^{q}\frac{\sqrt{(\gamma+1)\mu^{2}-2(\gamma-1)\mathcal{B}}}{(\gamma-1)\mu\sqrt{2\mathcal{B}-\mu^{2}}}\ d\mu,
\end{eqnarray*}
with
\begin{equation*}
\mathcal{B}:= \frac{q^{2}}{2}+\frac{c^{2}}{\gamma-1}, \qquad \mathcal{B}_0:=\frac{q^{2}_{0}}{2}+\frac{c^{2}_{0}}{\gamma-1}.
\end{equation*}
Note that the shock wave curve $S^{-}_{j}(U_0)$ contacts with $R^{+}_{j}(U_0)$ at $U_0$ up to second order for each $j=1,5$.

\subsection{Riemann problem for the homogeneous system \eqref{eq:2.2}}
For $\delta_{0}>0$, define
\begin{equation}\label{eq:2.14}
D_{\delta_{0}}(U_{\infty})=\left\{U:
\begin{array}{l}
\big|J(q,\frac{q^{2}}{2}+\frac{c^{2}}{\gamma-1})+\theta
-J(q_{\infty},\frac{q^{2}_{\infty}}{2}+\frac{c^{2}_{\infty}}{\gamma-1})\big|<\delta_{0}\\[6pt]
\big|\frac{q^{2}}{2}+\frac{c^{2}}{\gamma-1}-\frac{q^{2}_{\infty}}{2}-\frac{c^{2}_{\infty}}{\gamma-1}\big|<\delta_{0},
\,\,\, \Big|\frac{p}{\rho^{\gamma}}-\frac{p_{\infty}}{\rho^{\gamma}_{\infty}}\Big|<\delta_{0}\\ [6pt]
0\leq Z<\delta_{0}, \,\,\, \delta_{0}+\theta_{\rm crit}<\theta<\delta_{0}
\end{array}
\right \},
\end{equation}
where
\begin{eqnarray}\label{eq:2.16}
\theta_{\rm crit}:=\inf_{\theta}\big\{\theta : U\in R^{+}_{5}(U_{\infty}),\ u>c_{*},\ q<q_{*}\big\}
\end{eqnarray}
with the critical sonic speed $c_{*}$ and critical speed $q_{*}$ given, respectively, by
\begin{equation}\label{eq:2.17}
c_{*}=\sqrt{\frac{(\gamma-1)u^{2}_{\infty}}{\gamma+1}+\frac{2c^{2}_{\infty}}{\gamma+1}},\qquad
q_{*}=\sqrt{u^{2}_{\infty}+\frac{2c^{2}_{\infty}}{\gamma-1}},
\end{equation}
which depend only on the unperturbed initial state.

In what follows, $\delta_{0}$ is always chosen small enough so that
$\overline{D_{\delta_{0}}(U_{0})}\subset \{U: u>c\}$.

\smallskip
We now consider the Riemann problem for \eqref{eq:2.2} in $D_{\delta_{0}}(U_{\infty})$ with
\begin{equation}\label{eq:2.19}
\left.U\right|_{x=\hat{x}}=\left\{
\begin{array}{llll}
      U_a, \quad  &y>\hat{y},\\[1mm]
      U_b, \quad  &y<\hat{y},
      \end{array}
 \right.
\end{equation}
where the constant states $U_a$ and $U_b$ denote the $above$ state and $below$
state with respect to line $y=\hat{y}$, respectively.

When $|U_a-U_b|$ is sufficiently small,  following Lax \cite{lax},
we can parameterize any physically admissible wave curve in a neighborhood of
$D_{\delta_{0}}(U_{\infty})$ by
$$
\alpha_{j}\mapsto\Phi_{j}(\alpha_{j}, U_b)
$$
with $\Phi_{j}\in C^{2}\big((-\delta_{0}, \delta_{0})\times D_{\delta_{0}}(U_{\infty})\big)$ and
$$
\left. \Phi_{j}\right|_{\alpha_{j}=0}= U_b,\qquad
\left. \frac{\partial\Phi_{j}}{\partial\alpha_{j}}\right|_{\alpha_{j}=0}=\rr_{j}(U_b).
$$
Then $\alpha_{j}>0$ along $\emph{R}^{+}_{j}(U_b)$,
while $\alpha_{j}<0$ for $\emph{S}^{-}_{j}(U_b)$ with $1\leq j\leq 5$.

For simplicity, let
\begin{equation}\label{eq:2.20}
\Phi(\alpha_{5},\alpha_{4},\alpha_{3},\alpha_{2}, \alpha_{1}; U_b)
:=\Phi_{5}(\alpha_{5},\Phi_{4}(\alpha_{4},\Phi_{3}(\alpha_{3}, \Phi_{2}(\alpha_{2}, \Phi_{1}(\alpha_{1}, U_b))))).
\end{equation}
Then we have

\begin{lemma}\label{lem:2.2}
There exists $0<\delta_{1}\ll 1$ such that, for any states
$U_a, U_b\in D_{\delta_{0}}(U_{\infty})$ with $|U_a-U_b|<\delta_{1}$,
problem \eqref{eq:2.2} and \eqref{eq:2.19} admits a unique admissible solution consisting of five elementary waves.
In addition, state $U_a$ can be solved by
\begin{eqnarray*}
U_a=\Phi(\alpha_{5},\alpha_{4}, \alpha_{3}, \alpha_{2}, \alpha_{1};U_b),
\end{eqnarray*}
with $\left. \Phi\right|_{\alpha_{1}=\alpha_{2}=\alpha_{3}=\alpha_{4}=\alpha_{5}=0}= U_b$ and
$\left.\frac{\partial\Phi}{\partial\alpha_{j}} \right|_{\alpha_{1}=\alpha_{2}=\alpha_{3}=\alpha_{4}=\alpha_{5}=0}=\rr_{j}(U_b),\ 1\le j\le 5$.
\end{lemma}

\begin{remark}\label{rem:2.1}
For any state
$ U_b\in D_{\delta_{0}}(U_{\infty})$,
we can also parameterize the whole curve $R_{5}(U_b)$
by solving the equation{\rm :}
\begin{eqnarray*}
\frac{dU_{\rm Ra}(\sigma, U_b)}{d\sigma}=\rr_{5}(U_{\rm Ra}(\sigma, U_b))
\end{eqnarray*}
with $\left. \Phi\right|_{\sigma=0}= U_b$. Then $\Phi_{5}$ can be extended by $\Phi_{5}(\sigma, U_b)=U_{\rm Ra}(\sigma, U_b)$
for $\sigma\geq 0$.
\end{remark}

\section{Strong Rarefaction Waves for Euler Flow}
In this section, we investigate the background solution
for the Euler equations for steady supersonic Euler flow with $Z=0$ for \eqref{eq:2.2}
over a convex Lipschitz wall $g_*(x)$ and the constant incoming flow
$U_{\infty}=(u_{\infty}, 0, p_{\infty}, \rho_{\infty}, 0)^{\top}$.
We call system \eqref{eq:2.2} with $Z=0$ simply as the Euler equations.
We first consider a special case $g_{*}(x)=g_{*\Delta}(x)$ with
\begin{equation*}
 g_{*\Delta}(x)=\left\{
\begin{array}{ll}
  0, & x\leq 0,\\ [1mm]
  g_{*}(x^{*}_{k})+ (x-x^{*}_{k})\tan\theta^{*}_{k}, & x\in(x^{*}_{k}, x^{*}_{k+1}),\ k\geq0,
 \end{array} \right.
 \end{equation*}
where $\{x^{*}_{k}\}^{\infty}_{k=0}$ is a sequence of points with
$x^{*}_{0}=0,\ 0<x^{*}_{k}<x^{*}_{k+1}, k\geq1,\
\Delta x^{*}=x^{*}_{k+1}-x^{*}_{k}$, and
\begin{eqnarray*}
\theta^{*}_{k}=\arctan(\frac{g_{*}(x^{*}_{k+1})-g_{*}(x^{*}_{k})}{\Delta x^{*}})
\end{eqnarray*}
satisfies $\theta_{\rm crit}<\theta^{*}_{k+1}<\theta^{*}_{k}<0$ for $k\geq 0$.

When the Euler flow passes the convex Lipschitz wall $g_{*\Delta}(x)$, some rarefaction waves are generated by the
corner points of the wall.
\vspace{20pt}
\begin{center}
\begin{tikzpicture}[scale=2.0]
\draw [ultra thick] (-3.2,0)--(-1.4,0)-- (0.5,-0.4);
\draw [thin](-1.4,0) --(-0.6,1.2);\draw [thin] (-1.4,0) --(-0.9,1.2);\draw [thin] (-1.4,0) --(-1.2,1.2);
\draw [thick][->](-2.6, 0.5) --(-1.6, 0.5);
\draw [thick][->](-0.6, 0.5) --(0.2, 0.4);
\draw [thin][<-](-0.9,-0.8) --(0, -0.5);
\node at (-2, 0.7) {$U^{*}_{k}$};
\node at (-0.1, 0.7) {$U^{*}_{k+1}$};
\node at (-1.5, -0.4) {$(x^{*}_{k},g_{*\Delta}(x^{*}_{k}))$};
\node at (-1, -1.0) {$y=g_{*\Delta}(x)$};
\node [below] at (-0.5, -1.2)
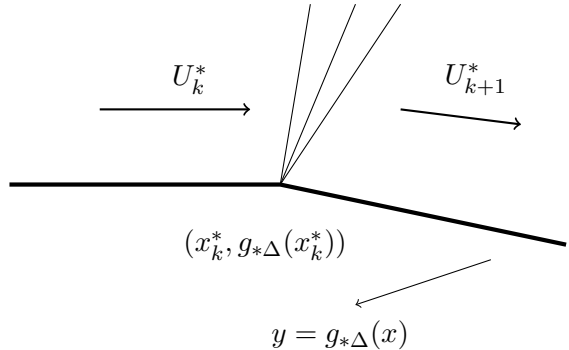
{Fig. 3.1. Steady supersonic Euler flow past a piecewise convex wall};
\end{tikzpicture}
\end{center}
Then we can employ the methods that have been given in \cite{courant, wy} to construct
the solution. We summary this as the following theorem.

\begin{theorem}\label{thm:3.1}
Suppose that $g_{*}(x)\equiv g_{*\Delta}(x)$ and $\theta^{*}_{k}\in (\theta_{\rm crit}, 0), k\geq 0$.
Let $U^{*}_{k}=(u^{*}_{k},v^{*}_{k},p^{*}_{k},\rho^{*}_{k},0)$ be the incoming states that
depend on the initial data $U_{\infty}$.
Then there exists a unique solution $U^{*}_{\Delta}(x,y)$ for the Euler equations \eqref{eq:2.2}
over the convex wall $g_{*\Delta}(x)$
in the supersonic region $\{U: u>c_{*},\ q<q_{*}\}$
which consists of strong rarefaction waves that are generated by the
corner points $(x^{*}_{k}, g_{*\Delta}(x^{*}_{k})), k\geq0$,
where $\theta_{\rm crit}$, $c_{*}$, and $q_{*}$ are defined
as in \eqref{eq:2.16}--\eqref{eq:2.17}.
\end{theorem}

\vspace{10pt}
\begin{center}
\begin{tikzpicture}[scale=1.8]
\draw [thick][<->] (0,1.9) --(0,0)-- (2.2,0);
\draw[thick] (0,0) circle (0.5);
\draw [thick](0,0) circle (1.5);
\draw [thick][->] (0,0) -- (0.4, 0.3);
\draw [thick][->] (0,0) -- (-0.5, 1.414);
\draw [thick][->] (0,0) -- (1, -0.76);
\draw [thick](0.4, 0.3)to [out=0,in=125](1,0)to [out=-55,in=-10](-0.25,-1.48);
\draw[thick](0.1, 0)to [out=100,in=35](0.07,-0.05);
\node at (1.2,0.15) {$(u_{\infty}, 0)$};
\node at (2.2,-0.1) {$u$};
\node at (-0.1, 1.9) {$v$};
\node at (-0.1, -0.1) {$O$};
\node at (-0.6, 0.9) {$q=q_{*}$};
\node at (0.5, 0.5) {$q=c_{*}$};
\node at (0.22, -0.07) {$\theta$};
\node [below] at (0.2, -1.8)
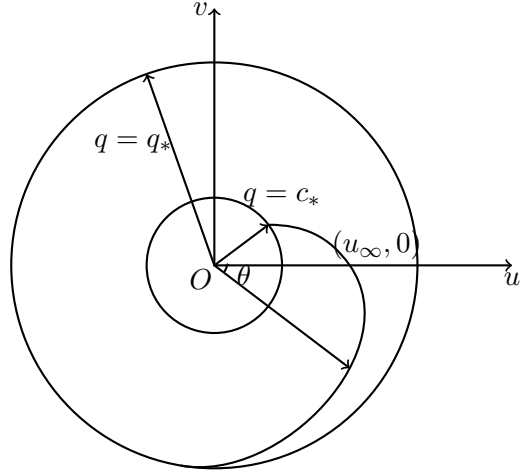
{Fig. 3.2. The epicycloid issues from $(u_{\infty}, 0)$ in the $(u, v)$--plane};
\end{tikzpicture}
\end{center}

As a corollary, we can prove a similar result for the general case of Theorem \ref{thm:3.1}.

\begin{corollary}\label{cor:3.1}
Suppose that $\arctan(g'_{*}(x))\in (\theta_{\rm crit}, 0)$,
and $g'_{*}(\cdot)$ is monotone decreasing.
If $TV.(g'_{*}(\cdot))<\infty$,
then there exists a unique solution $U^{*}(x,y)$ in the supersonic
region $\{U: u>c_{*},\ q<q_{*}\}$
which connects $U_{\infty}$ by a strong rarefaction wave $R^{+}_{5}(U_{\infty})$
for the Euler equations \eqref{eq:2.2} over the convex wall
$y=g_{*}(x)$,
where $\theta_{\rm crit}$, $c_{*}$, and $q_{*}$
are defined as in \eqref{eq:2.16}--\eqref{eq:2.17}.
\end{corollary}
\vspace{10pt}
\begin{center}
\begin{tikzpicture}[scale=1.6]
\draw [thick][->] (-1.6,1.2)--(-0.3,1.2);
\draw [thick][->] (-1.6,0.8)--(-0.3,0.8);
\draw [thick][->] (-1.6,0.4)--(-0.3,0.4);
\draw [thick][<-] (-0.1,-0.7) --(1,-0.4);
\draw [thick][->] (2.6,-0.1)--(3.5,-0.5);
\draw [line width=0.1cm](-1.8,0)to [out=0, in=0](0, 0)to [out=0,in=145](2.8,-1.1);
\draw [thick](0,0)--(1.5, 1.8);
\draw  [thick](0.6,-0.06)--(2, 1.5);
\draw  [thick](1.1,-0.18)--(2.7, 1.3);
\draw  [thick](1.6,-0.4)--(3.3, 1);
\node at (-1, 1.5) {$U_{\infty}$};
\node at (3.3, -0.1) {$U^{*}(x,y)$};
\node at (0.1, -1) {$y=g_{*}(x)$};
\node at (2.5,1.7) {$R^{+}_{5}(U_{\infty})$};
\node [below] at (0.3, -1.4)
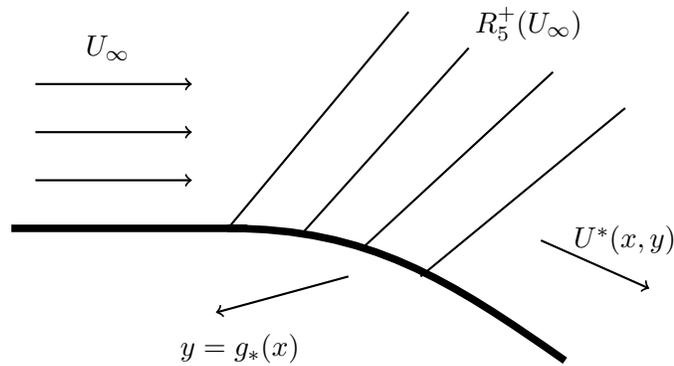
{Fig. 3.3. Supersonic Euler flow past a convex bending wall};
\end{tikzpicture}
\end{center}

\section{Construction of Approximate Solutions}\setcounter{equation}{0}

In this section, we construct a family of global approximate solutions to system \eqref{eq:1.1}
by employing a fractional-step scheme based on the wave-front tracking algorithm.
As a first step, we need to approximate domain $\Omega$ by $\Omega_{h}$ defined below.

Let $h=\Delta x>0$ be the mesh length in the $x$--direction. Then
we choose a set of points $\{A_{k}\}_{k=0}$ with $A_{k}=(x_k, g_k):=(kh, g(kh))$
on the boundary that connects $A_{k}$ to $A_{k+1}$ in order.
Denote
\begin{eqnarray*}
&&\theta_{0}=\arctan g_{0},\,\, \quad\theta_{k}=\arctan(\frac{g_{k+1}-g_{k}}{h}),\ k\ge 1,\\
&&\omega_{0}=\arctan(\frac{g(x_0)-g(0)}{x_{0}}),\,\,\quad  \omega_{k}=\theta_{k}-\theta_{k-1},\ k\ge 1,
\end{eqnarray*}
where $\omega_k$ represents the change of angle $\theta_{k-1}$ at the turning points $A_k, k\ge 1$.
Define
\begin{eqnarray*}
&&g_{h}(x)=g_k+(x-kh)\tan\theta_{k}\qquad \mbox{for any $x\in[kh, (k+1)h), k\geq0$},\\[2mm]
&&\Omega_{h, k}=\{(x,y)\,:\, kh\leq x< (k+1)h, \ y> g_{h}(x)\},\\[2mm]
&&\Gamma_{h, k}=\{(x,y)\,:\, kh\leq x< (k+1)h, \ y=g_{h}(x)\},\\[2mm]
&&\Omega_{h}=\bigcup_{k\geq0}\Omega_{ h, k}, \,\,\quad\Gamma_{h}=\bigcup_{k\geq0}\Gamma_{h, k}.
\end{eqnarray*}
Let $\textbf{n}_{k}$ be the outer unit normal vector to $\Gamma_{h, k}$:
\begin{equation*}
\textbf{n}_{k}=\frac{(g_{k+1}-g_{k},-x_{k+1}+x_{k})}{\sqrt{(g_{k+1}-g_{k})^2+(x_{k+1}-x_{k})^2}}=(\sin\theta_{k},-\cos\theta_{k}).
\end{equation*}

The approximate solutions for system \eqref{eq:2.2} are constructed by an induction procedure,
together with the fractional-step wave-front tracking scheme.
That is, if the solution has been constructed for $x\leq (k-1)h$,
then, for fixed $h$,
we solve the homogeneous system \eqref{eq:2.2} between $(k-1)h<x<kh$.
Finally, we solve the nonhomogeneous
problem for system \eqref{eq:2.2} from $x=kh-$ to $x=kh+$ with initial data taking at $x=kh-$.
The details of the construction can be seen in \S 4.1 below.
\vspace{10pt}
\begin{center}
\begin{tikzpicture}[scale=1.4]
\draw [dashed][ultra thick](-2.5,-1.5)--(-2.5,2);
\draw [dashed][ultra thick](0.5,-1.5)--(0.5,2);
\draw [thin](0.5,-1.2)--(1.5,-0.6);
\draw [ultra thick] (-4,-0.5)--(-2.5,-0.5)--(0.5,-1.2)--(1.9,-1);
\draw [thin](-1.7,0)--(-1.3,-0.8)--(-0.6,0)--(-1.3, 0.3);
\draw [thin](0.5,0.3)--(1.2, 1);
\draw [thin](0.5,0.3)--(1.5, 0.5);
\draw [thin](0.5,0.3)--(1.5, 0);
\draw [thin](-0.6,0)--(0.5, 0.3);
\draw [thin](-0.6,0)--(0.1, 0);
\draw [thin](-0.6,0)--(0.1, -0.3);
\draw [thin](-2.5,-0.5)--(-1.8,0.3);
\draw [thin](-2.5,1)--(-1.6,0.7);
\draw [thin](-2.5,1)--(-1.6,1.4);
\draw [thin](-2.5,1)--(-1.6,1.1);
\draw [thin](-1.5,1)--(0,1.5);
\draw [thin](-1.5,1.5)--(0,1);
\draw [thin](-0.75,1.25)--(-0.2,0.7);
\draw [dashed][thick](-0.4,-0.7)--(0.5,-0.6)--(1.4, -0.2);
\node at (-2.5, -1.7) {$x=(k-1)h$};
\node at (0.5, -1.7) {$x=kh$};
\node at (1.8, -1.3) {$y=g_{h}(x)$};
\node [below] at (-1.0, -2)
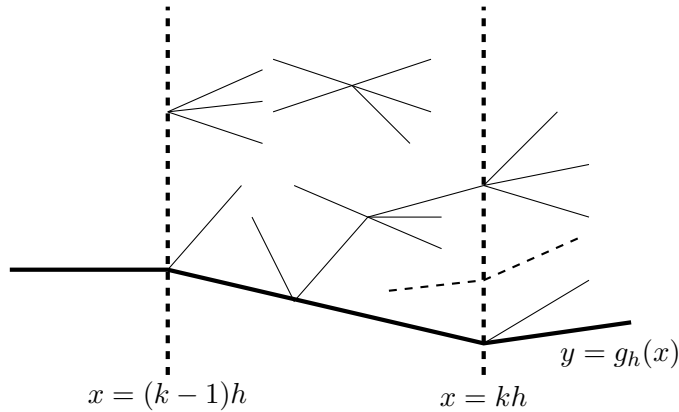
{Fig. 4.1.  Wave-front algorithm with the reacting steps};
\end{tikzpicture}
\end{center}

\subsection{Riemann solvers for the homogeneous system \eqref{eq:2.2}}

As mentioned in \S 2, the Riemann problem:
\begin{eqnarray}\label{eq:4.1}
W(U)_x+H(U)_y=0,\qquad
\left.U\right|_{x=\hat{x}}=\left\{
\begin{array}{llll}
      U_a, \quad  &y>\hat{y},\\[1mm]
      U_b, \quad  &y<\hat{y},
      \end{array}
 \right.
\end{eqnarray}
admits a unique self-similar solution given by at most four states separated
by shocks, contact discontinuities, or rarefaction waves. More precisely, the solution
is inductively defined by
\begin{eqnarray}\label{eq:4.2}
U_{0}=U_b, \qquad   U_{i}=\Phi_{i}(\alpha_{i},U_{i-1}),\   1\leq i\leq5,\, \qquad U_{5}=U_a.
\end{eqnarray}
Following \cite{amadori,bressan}, there are two procedures to define the approximate
solutions to the Riemann problem \eqref{eq:4.1}: The accurate Riemann solver
and the simplified Riemann solver.

\smallskip
$\mathbf{Accurate}$ $\mathbf{Riemann}$ $\mathbf{solver}$.
For any $\nu\in\mathbb{N}$, a $\nu$--approximate solution $U^{\nu}$ to the Riemann problem
at any jump point $(\hat{x}, \hat{y})$ is defined
by dividing every rarefaction wave
into $\nu$ parts.
That is, if $\alpha_{1}>0$, then set $U_{0,0}=U_b$, $U_{0,\nu}=U_b$, and
\begin{eqnarray*}
U_{0, k}=\Phi_{1}(\frac{k}{\nu}\alpha_{1}, U_{0, k-1}),
\quad y_{1,k}=\hat{y}+(x-\hat{x})\lambda_{1}(U_{0, k})\qquad
\mbox{for $1\leq k\leq\nu$.}
\end{eqnarray*}
In place of the $1$-rarefaction wave, define
\begin{eqnarray}\label{eq:4.3}
U^{\nu}_{A}(U_b, U_a)=\left\{
\begin{array}{lllll}
      U_{l}, \  &y<y_{1,1},\\[1mm]
      U_{0,k}, \  &y_{1,k}<y<y_{1,k+1},\\[1mm]
      U_{1},\  &y_{1,\nu}<y<\hat{y}+(x-\hat{x})\lambda^{*}_{1}
\end{array}
 \right.
\end{eqnarray}
for some $\lambda^{*}_{1}\in(\max\lambda_{1}, \min\lambda_{2(3,4)})$.
Then $U^{\nu}_{A}$
in $\{y: y <\hat{y}+(x-\hat{x})\lambda^{*}_{1}\}$ given by \eqref{eq:4.3}
is called an approximate $1$-rarefaction wave,
which stands for the accurate $\nu$-Riemann solver.
For simplicity, we still use $\alpha_{1}$ to denote both the wave given by \eqref{eq:4.3}
and its magnitude.
The discontinuities in \eqref{eq:4.3} are
called $1$-rarefaction fronts, and the magnitude of each front is $\frac{\alpha_{1}}{\nu}$.

In the same way, we can define
the approximate $5$-rarefaction wave $U^{\nu}_{A}$  with magnitude $\alpha_{5}$
in domain $\{y: y > \hat{y}+(x-\hat{x})\lambda^{*}_{5}\}$ for some $\lambda^{*}_{5}\in(\max\lambda_{2(3,4)}, \min\lambda_{5})$.
It also contains $\nu$ $5$-rarefaction fronts, and the magnitude of each front is $\frac{\alpha_{5}}{\nu}$.

In this construction, the $i$-shocks, $i=1,5$, and $j$-contact discontinuities, $j=2,3,4$, are not modified at all.

\medskip
$\mathbf{Simplified}$ $\mathbf{Riemann}$ $\mathbf{solver}$.
As described in \cite{amadori,bressan}, the simplified Riemann solver puts together all the new waves
in a single non-physical front which travels faster than all the characteristic speeds.
This is defined for the following two cases:

\smallskip
\par $\mathbf{Case\ 1.}$ Let $1\leq i\leq j\leq 5$ be the families of the two incoming wave-fronts
interacting at $(\hat{x}, \hat{y})$.
Assume that the below, middle, and above states $\{U_b, U_{m}, U_a\}$ before interaction
are connected by
\begin{eqnarray}\label{eq:4.4}
U_{m}=\Phi_{j}(\beta, U_b), \qquad U_a=\Phi_{i}(\alpha, U_{m}).
\end{eqnarray}
Define the auxiliary above state:
\begin{eqnarray}\label{eq:4.5}
U'_a=
\begin{cases}
      \Phi_{j}(\beta,\Phi_{i}(\alpha, U_b)), \ &j>i,\\[3pt]
      \Phi_{j}(\alpha+\beta, U_b), \ &j=i.
\end{cases}
\end{eqnarray}
Choose a constant $\hat{\lambda}$ such that
\begin{eqnarray*}
\hat{\lambda}>\underset{i, U}{\sup}\{\lambda_{i}(U)\,:\, U \in D_{\delta_0}(U_{\infty}),\ 1\leq i\leq 5\}.
\end{eqnarray*}
Then, in a forward neighbourhood of point $(\hat{x}, \hat{y})$,
we define the approximate solution $U_{S}(U_b, U_{m}, U_a)$ to problem \eqref{eq:4.1} as follows:
\begin{eqnarray}\label{eq:4.6}
U_{S}(U_b, U_{m}, U_a)=\left\{
\begin{array}{lll}
     U^{\nu}_{A}(U_b,U'_a), \ &y-\hat{y}<\hat{\lambda}(x-\hat{x}),\\[3pt]
      U_a, \ &y-\hat{y}>\hat{\lambda}(x-\hat{x}),
\end{array}
 \right.
\end{eqnarray}
where $U^{\nu}_{A}(U_b,U'_a)$ is the accurate Riemann solver constructed as in \eqref{eq:4.3}
which contains at most two wave-fronts. The part
\begin{eqnarray}\label{eq:4.7}
U_{\rm np}=\left\{
\begin{array}{lll}
     U'_a, \ &y-\hat{y}<\hat{\lambda}(x-\hat{x}),\\[3pt]
      U_a, \ &y-\hat{y}>\hat{\lambda}(x-\hat{x})
\end{array}
 \right.
\end{eqnarray}
is called a non-physical wave-front whose strength is defined to be $|U_a-U'_a|$.

\smallskip
$\mathbf{Case\ 2}.$  A non-physical wave-front with strength $\epsilon$ hits  a wave-front
of the $i$-characteristic family at $(\hat{x},\hat{y})$ from the below for some $1\leq i\leq 5$.

Suppose that the below, middle, and above states $\{U_b, U_{m}, U_a\}$ before the interaction
are connected by
\begin{eqnarray*}
|U_{m}-U_b|=\epsilon,\quad  U_a=\Phi_{i}(\alpha, U_{m}).
\end{eqnarray*}
Then the approximate solution $U_{S}(U_b, U_{m}, U_a)$ to problem \eqref{eq:4.1}
is defined as follows:
\begin{eqnarray}\label{eq:4.8}
U_{S}(U_b, U_{m}, U_a)=\left\{
\begin{array}{lll}
    U_b,\ &y-\hat{y}<\lambda(U_b)(x-\hat{x}),\\[3pt]
    \Phi_{i}(\alpha, U_b),\ &\lambda(U_b)(x-\hat{x})<y-\hat{y}<\hat{\lambda}(x-\hat{x}), \\[3pt]
      U_a,\ &y-\hat{y}>\hat{\lambda}(x-\hat{x}).
\end{array}
 \right.
\end{eqnarray}

\subsection{Initial-boundary value problems for the homogeneous system \eqref{eq:2.2}}
$\,$ To study the flow past a corner point, we consider the following initial-boundary value problem:
\begin{eqnarray}\label{eq:4.9}
\left\{
\begin{array}{llll}
      W(U)_x+H(U)_y=0 \quad  &\mbox{in}\ \Omega_{ h, k},\\[3pt]
     \left.U\right|_{x=x_{k-1}}=U_{k-1} ,\\[3pt]
     (u,v)\cdot \textbf{n}_{k}=0 \quad  & \mbox{on}\ \Gamma_{h, k},
     \end{array}
     \right.
\end{eqnarray}
where $U_{k}$ is a constant state.

To describe the interaction and reflection on the boundary, we first need the following lemmas.

\begin{lemma}\label{lem:4.1}
Let $\{U_{1}, U_{2}\}$ be the two constant states related by
\begin{eqnarray}\label{eq:4.10}
U_{1}=\Phi_{5}(\alpha;U_{2}).
\end{eqnarray}
Then, for $|\varepsilon|\ll 1$ sufficiently small, there exists a $C^{2}$--function $\Psi$ such that
\begin{eqnarray}\label{eq:4.11}
U_{2}=\Psi(\alpha; U_{1}).
\end{eqnarray}
Moreover,
\begin{eqnarray}\label{eq:4.12}
\left.\frac{\partial \Psi}{\partial \alpha}\right|_{\alpha=0}=-\left.\frac{\partial\Phi_{5}}{\partial \alpha}\right|_{\alpha=0}=-\rr_{5}(U_{1}).
\end{eqnarray}
\end{lemma}

\begin{proof}
Since
\begin{eqnarray*}
\nabla_{U_{2}}\Phi_{5}(0; U_2)=I,
\end{eqnarray*}
then, by the implicit function theorem, for $|\alpha|$ sufficiently small,
we can find a $C^{2}$--function $\Psi$ such that
\begin{eqnarray*}
U_{2}=\Psi(\alpha; U_{1}) \qquad\,\, \mbox{for $|\alpha|\ll 1$}.
\end{eqnarray*}
Therefore, ${\eqref{eq:4.10}}$ can be reduce to
\begin{eqnarray*}
U_{1}=\Phi_{5}(\alpha;\Psi(\alpha; U_{1})),
\end{eqnarray*}
which leads to
\begin{eqnarray*}
\left.\frac{\partial \Psi}{\partial \alpha}\right|_{\alpha=0}
=-\left.\frac{\partial\Phi}{\partial \alpha}\right|_{\alpha=0}=-\rr_{5}(U_1).
\end{eqnarray*}
This completes the proof.
\end{proof}

Then we have
\begin{lemma}\label{lem:4.2}
Let $U_{k-1}=(u_{k-1},v_{k-1},p_{k-1},\rho_{k-1},Z_{k-1})$ be the constant state near the boundary $\Gamma_{h,k-1}$
with $U_{k-1}\cdot (\textbf{n}_{k-2}, 0,0,0)=0$.

\begin{enumerate}
\item[\rm (i)] If $|\omega_{k-1}|\ll1$, then there exists a unique solution $(\varepsilon_{5}, U_{k})$ that
solves the problem:
\begin{equation}\label{eq:4.13}
\begin{cases}
      \Phi_{5}(\varepsilon_{5}; U_{k})=U_{k-1},\\[3pt]
      U_{k}\cdot (\textbf{n}_{k-1},0,0,0)=0.
     \end{cases}
\end{equation}
Moreover, there exists $\tilde{K}_{\rm b}<0$ such that
\begin{eqnarray}\label{eq:4.14}
\varepsilon_{5}=\tilde{K}_{\rm b}\omega_{k-1},
\end{eqnarray}
where the bound of $\tilde{K}_{\rm b}$ depends only on the system.
\vspace{25pt}
\begin{center}
\begin{tikzpicture}[scale=1.8]
\draw [ultra thick] (-2.4,0)--(-0.6,0)-- (1.2,0.6);
\draw [thick](-0.6,0) --(0.2,1.2);
\draw [thick][->](-1.8,0.6) --(-1,0.6);
\draw [thick][->](0,0.4) --(1,0.8);
\draw [thick][->](-1.5,0) --(-1.5,-0.5);
\draw [thick][->](0.5,0.35) --(0.8,-0.25);
\draw [dashed][ultra thick](-0.6,1.4)--(-0.6,-0.6);
\node at (-1.2, -0.3) {$\textbf{n}_{k-2}$};
\node at (1, 0) {$\textbf{n}_{k-1}$};
\node at (-1.2, 0.9) {$U_{k-1}$};
\node at (0.6, 0.9) {$U_{k}$};
\node at (0.3, 1.3) {$\varepsilon_{5}$};
\node at (-0.6, -0.7) {$x=(k-1)h$};
\node [below] at (-0.8, -0.9)
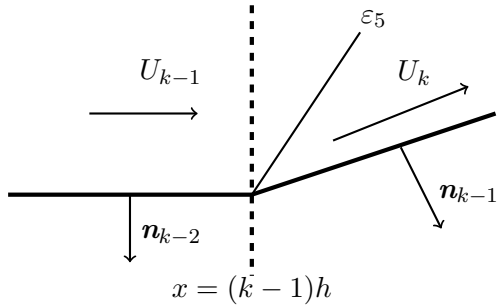
{\rm Fig. 4.2.  Small waves generated by the corner points};
\end{tikzpicture}
\end{center}

\item[\rm (ii)] If $\omega_{k-1}\in (\theta_{\rm crit}+\theta_{k-1}, 0)$ with
$\theta_{k-1}=\sum^{k-1}_{i=0}|\omega_{i}|$ and $\theta_{k-1}<-\theta_{\rm crit}$,
then there exists a unique solution $(\varepsilon_{5}, U_{k})$ that
consists of a strong rarefaction wave generated by the
corner point $(x_{k-1}, g_{k-1})$ and satisfies the equations:
\begin{eqnarray}
&&\theta(U_{\rm Ra}(-\varepsilon_{5}, U_{k-1}))-\theta(U_{k-1})=\omega_{k-1},\label{eq:4.15-1} \\[3pt]
&&\big(J(q, \frac{q^{2}}{2}+\frac{c^{2}}{\gamma-1})+\theta\big)(U_{\rm Ra}(-\varepsilon_{5}, U_{k-1}))
     =\big(J(q,\frac{q^{2}}{2}+\frac{c^{2}}{\gamma-1})+\theta\big)(U_{k-1}),\quad \label{eq:4.15-2} \\[3pt]
&&(\frac{q^{2}}{2}+\frac{c^{2}}{\gamma-1})(U_{\rm Ra}(-\varepsilon_{5}, U_{k-1}))
    =\big(\frac{q^{2}}{2}+\frac{c^{2}}{\gamma-1}\big)(U_{k-1}),\label{eq:4.15-3} \\[3pt]
&&\big(\frac{p}{\rho^{\gamma}}\big)(U_{\rm Ra}(-\varepsilon_{5}, U_{k-1}))=\big(\frac{p}{\rho^{\gamma}}\big)(U_{k-1}),\label{eq:4.15-4}\\[3pt]
&& Z_{k}=Z_{k-1}, \label{eq:4.15-5}
\end{eqnarray}
and
\begin{eqnarray}\label{eq:4.16}
U_{k}=U_{\rm Ra}(-\varepsilon_{5}, U_{k-1}),
\end{eqnarray}
where $U_{\rm Ra}(-\varepsilon_{5}, U_{k-1})$ is the parametrization of $R^{+}_{5}(U_{k-1})$ given in Remark {\rm \ref{rem:2.1}}.
Moreover,
\begin{eqnarray}\label{eq:4.17}
\varepsilon_{5}=K'_{\rm b}\omega_{k-1},
\end{eqnarray}
where the bound of $K'_{\rm b}$ depends only on the system.
\end{enumerate}
\end{lemma}
\vspace{5pt}
\begin{center}
\begin{tikzpicture}[scale=2.0]
\draw [ultra thick] (-2.4,0)--(-0.6,0)-- (1.2,-0.4);
\draw [thick](-0.6,0) --(0.4,0.8);
\draw [thick](-0.6,0) --(0.2,1.2);
\draw [thick][->](-1.8,0.6) --(-0.8,0.6);
\draw [thick][->](0.1,0.2) --(1.0,-0.05);
\draw [thick][->](-1.5,0) --(-1.5,-0.5);
\draw [thick][->](0.7,-0.3)--(0.5,-0.75);
\draw [dashed][ultra thick](-0.6,1.4)--(-0.6,-0.6);
\node at (-1.2, -0.3) {$\textbf{n}_{k-2}$};
\node at (0.9, -0.6) {$\textbf{n}_{k-1}$};
\node at (-1.2, 0.9) {$U_{k-1}$};
\node at (0.6, 0.3) {$U_{k}$};
\node at (0.6, 1.4) {$R^{+}_{5}(U_{k-1})$};
\node at (-0.6, -0.7) {$x=(k-1)h$};
\node [below] at (-1.0, -0.9)
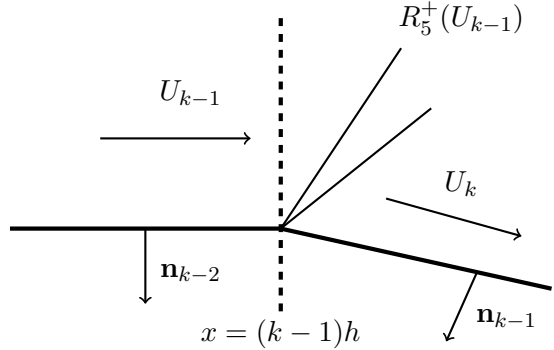
{Fig. 4.3.  Strong rarefaction waves generated by the corner points};
\end{tikzpicture}
\end{center}

\begin{proof} We divide the proof into two steps.

\smallskip
{\bf 1.} By Lemma \ref{lem:4.1},
there exists a $C^{2}$--function $\Psi$ such that
\begin{eqnarray*}
&&U_{k}=\Psi(\varepsilon_{5}; U_{k-1}) \qquad\,\,\mbox{for $|\varepsilon_{5}|\ll 1$},\\[2mm]
&&\left.\frac{\partial \Psi}{\partial \varepsilon_{5}}\right|_{\varepsilon_{5}=0}
=-\left.\frac{\partial\Phi}{\partial \varepsilon_{5}}\right|_{\varepsilon_{5}=0}=-\mathbf{r}_{5}(U_{k-1}).
\end{eqnarray*}
Then \eqref{eq:4.13}
can be reduced to
\begin{eqnarray}\label{eq:4.18}
\Psi(\varepsilon_{5}; U_{k-1})\cdot (\textbf{n}_{k-1},0,0,0)=0.
\end{eqnarray}
Note that $U_{k-1}\cdot (\textbf{n}_{k-2}, 0,0,0)=0$. Then
\begin{eqnarray*}
\arctan(\frac{v_{k-1}}{u_{k-1}})=\theta_{k-2},\qquad \lambda_{5}(U_{k-1})=\tan(\theta_{k-2}+\theta_{\rm ma}(U_{k-1})).
\end{eqnarray*}
Therefore, using Lemma \ref{lem:2.1}, we have
\begin{eqnarray*}
&&\left.\frac{\partial}{\partial \varepsilon_{5}}\Big(\Psi(\varepsilon_{5}; U_{k-1})
\cdot (\textbf{n}_{k-1},0,0,0)\Big)\right|_{\varepsilon_{5}=\omega_{k-1}=0}\\[1mm]
&&\,\,\,=-\mathbf{r}_{5}(U_{k-1})\cdot (\textbf{n}_{k-2},0,0,0)
=-\frac{\kappa_{5}(U_{k-1})\cos(\theta_{\rm ma}(U_{k-1}))}{\cos(\theta_{k-2}+\theta_{\rm ma}(U_{k-1}))}
<0.
\end{eqnarray*}
Then, by the implicit function theorem, $\varepsilon_{5}$ can be solved as a $C^{2}$--function of $(\omega_{k-1}, U_{k-1})$ with
\begin{eqnarray*}
\varepsilon_{5}=\varepsilon_{5}(\omega_{k-1}, U_{k-1}),
\end{eqnarray*}
which leads to the existence of solution $(\varepsilon_{5}, U_{k})$.

Finally, to establish estimate \eqref{eq:4.14},
we need to compute $\left.\frac{\partial\varepsilon_{5} }{\partial \omega_{k-1}}\right|_{\omega_{k-1}=0}$.
To this end, differentiating \eqref{eq:4.13} with respect to $\omega_{k-1}$ and letting $\omega_{k-1}=0$
yield
\begin{eqnarray*}
-\mathbf{r}_{5}(U_{k})\cdot (\textbf{n}_{k-1},0,0,0)\left.\frac{\partial\varepsilon_{5} }{\partial \omega_{k-1}}\right|_{\omega_{k-1}=0}
+U_{k}\cdot (\cos \theta_{k-2}, \sin\theta_{k-2},0,0,0)=0.
\end{eqnarray*}
Therefore,
\begin{eqnarray*}
\left.\frac{\partial\varepsilon_{5} }{\partial \omega_{k-1}}\right|_{\omega_{k-1}=0}
=-\frac{\cos(\theta_{k-2}+\theta_{\rm ma}(U_{k-1}))}{\kappa_{5}(U_{k-1})q(U_{k-1})\cos(\theta_{\rm ma}(U_{k-1}))}<0.
\end{eqnarray*}

\smallskip
{\bf 2.} Since
\begin{eqnarray*}
&& \left.\nabla_{U}\theta(U)\right|_{U=U_{\rm Ra}(-\varepsilon_{5},U_{k-1})}
=(-\frac{\sin(\theta(U_{\rm Ra}(-\varepsilon_{5}, U_{k-1})))}{q(U_{\rm Ra}(-\varepsilon_{5},U_{k-1}))},
\frac{\cos(\theta(U_{\rm Ra}(-\varepsilon_{5}, U_{k-1})))}{q(U_{\rm Ra}(-\varepsilon_{5},U_{k-1}))} , 0,0,0),\qquad\\[1mm]
&& \left.\frac{\partial U}{\partial\varepsilon_{5}}\right|_{U=U_{\rm Ra}(-\varepsilon_{5}, U_{k-1})}
=-\mathbf{r}_{5}(U_{\rm Ra}(-\varepsilon_{5},U_{k-1})).
\end{eqnarray*}
Then
\begin{eqnarray*}
\left.\frac{\partial\theta(U)}{\partial\varepsilon_{5}}\right|_{U=U_{\rm Ra}(-\varepsilon_{5}, U_{k-1})}
&=&\left.\Big(\nabla_{U}\theta(U)\cdot\frac{\partial U}{\partial\varepsilon_{5}}\Big)\right|_{U=U_{\rm Ra}(-\varepsilon_{5},U_{k-1})}\\[2pt]
&=&-\left.\frac{\kappa_{5}(U)\cos\theta_{\rm ma}(U)}{q(U)\cos(\theta(U)+\theta_{\rm ma}(U))}\right|_{U=U_{\rm Ra}(-\varepsilon_{5},U_{k-1})}<0
\end{eqnarray*}
for any $U_{k}\in D_{\delta_{0}}(U_{\infty})$,
which implies that $\eqref{eq:4.15-1}$ has a unique solution $\varepsilon_{5}$
with
\begin{eqnarray*}
\varepsilon_{5}=\varepsilon_{5}(\omega_{k-1},U_{k-1}).
\end{eqnarray*}
Moreover, we have
\begin{eqnarray*}
\varepsilon_{5}=\Big(\int^{1}_{0}\left.\frac{\partial\theta(U)}{\partial\varepsilon_{5}}\right|_{U=U_{\rm Ra}(-\sigma, U_{k-1})}d\sigma\Big)^{-1}\omega_{k}.
\end{eqnarray*}
Hence, from equations $\eqref{eq:4.15-2}$--$\eqref{eq:4.15-5}$,
we can obtain the solution to problem \eqref{eq:4.13} for $\omega_{k}<0$.
This completes the proof.
\end{proof}

Next, we consider some waves that interact the boundary and then reflect.
Assume that a $1$-wave $\alpha_{1}$  hits $\Gamma_{h,k-1}$ at the non-corner
point $A(\hat{x}, \hat{y})$ with $(k-1)h<\hat{x}<kh$ for some $k>0$, and let
\begin{eqnarray}\label{eq:4.19}
U_a=\Phi_{1}(\alpha_1; U_b), \qquad U_b\cdot(\textbf{n}_{k-1},0,0,0)=0.
\end{eqnarray}

We consider the problem:
\begin{eqnarray}\label{eq:4.20}
\left\{
\begin{array}{llll}
      W(U)_x+H(U)_y=0 \quad  &\mbox{for $x>\hat{x},\ y>\hat{y}$},\\[3pt]
     \left.U\right|_{x=\hat{x}}=U_a, \\[3pt]
     (u,v)\cdot \textbf{n}_{k-1}=0 \quad  &\mbox{on}\ \Gamma_{h,k-1}.
     \end{array}
     \right.
\end{eqnarray}

\begin{lemma}\label{lem:4.3}
Let $\{U_b, U_a\}$ be the constant states given as above and satisfy \eqref{eq:4.19}.
Then problem \eqref{eq:4.20} admits a unique solution $(\varepsilon_{5}, U'_a)$ that
solves the following problem{\rm :}
\begin{eqnarray}\label{eq:4.21}
\left\{
\begin{array}{llll}
      \Phi_{5}(\varepsilon_{5}; U'_a)=U_a,\\[3pt]
      U'_a\cdot (\textbf{n}_{k-1},0,0,0)=0.
     \end{array}
     \right.
\end{eqnarray}
Moreover,
\begin{eqnarray}\label{eq:4.22}
\varepsilon_{5}=K_{b1}\alpha_{1},
\end{eqnarray}
with $\left.K_{b1}\right|_{\alpha_{1}=0}=1$, and
  the bound of $K_{b1}$ depends only on the system.
\end{lemma}
\vspace{10pt}
\begin{center}
\begin{tikzpicture}[scale=2.0]
\draw [ultra thick] (-3,0.5)--(0,0);
\draw [thin](-2.6,1.2) --(-1.5,0.25)--(-0.4,0.8);
\draw [thick][->](-3,0.8) --(-2.2,0.7);
\draw [thick][->](-0.6,0.4) --(0,0.3);
\draw [thick][->](-1,0.15)--(-1.15,-0.25);
\node at (-2.6, 0.9) {$U_b$};
\node at (-1.5, 0.9) {$U_a$};
\node at (-0.2, 0.5) {$U'_a$};
\node at (-2.6, 1.3) {$\alpha_{1}$};
\node at (-0.3, 0.9) {$\varepsilon_{5}$};
\node at (-0.75, -0.15) {$\textbf{n}_{k-1}$};
\node [below] at (-1.0, -0.4)
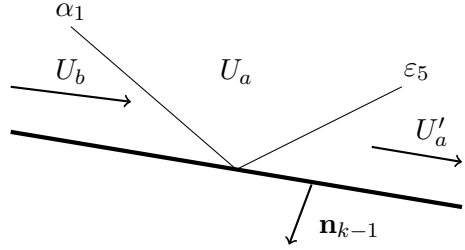
{Fig. 4.4. Weak waves reflected and physical wave emerged at the boundary};
\end{tikzpicture}
\end{center}

\begin{proof}
By Lemma \ref{lem:4.1},
 there exists a $C^{2}$--function $\Psi$ such that
\begin{eqnarray*}
U'_a=\Psi(\varepsilon_{5}; U_a)  \qquad\,\, \mbox{for $|\varepsilon_{5}|\ll 1$}.
\end{eqnarray*}
Then problem \eqref{eq:4.20} can be reduced to solving the equation:
\begin{eqnarray*}
\Psi\big(\varepsilon_{5};\Phi_{1}(\alpha_1; U_b)\big)\cdot (\textbf{n}_{k-1},0,0,0)=0.
\end{eqnarray*}
Thus, in a similar way as in Lemma \ref{lem:4.2}, we obtain the existence of $(\varepsilon_{5}, U'_a)$
and estimate \eqref{eq:4.22}. This completes the proof.
\end{proof}

\subsection{Wave-front tracking algorithm with fractional-step}
In this subsection, we construct the approximate solutions to system \eqref{eq:1.1}
by combining the fractional-step technique with the wave-front tracking algorithm.

For given initial data $\bar{U}$, we define a sequence of piecewise constant functions $\bar{U}^{\nu}$
such that
$$
\|\bar{U}^{\nu}-\bar{U}\|_{L^{1}(R_{+})}<2^{-\nu}
\qquad \mbox{for $\nu \in\mathbb{N_{+}}$ and the mesh length $h>0$ defined above},
$$
and approximate the wedge boundary as in \S 4.
Then the $(\nu, h)$--approximate solution $U^{\nu,h}(x,y)$ is constructed as follows:

\smallskip
When $x=0$, all the Riemann solutions in $\bar{U}^{\nu}$ are obtained
by the accurate Riemann solver, in which the number of wave-fronts is finite.
Then, for fixed $h$, we construct the approximate solution $U^{\nu,h}$ in strip $0<x<h$,
according to the ways as in \S 4.1--4.2 by solving the initial-boundary value problem.
After that, we need to solve the nonhomogeneous problem on line $x=h$
for $k=1$.
Inductively, assuming that the approximate solution $U^{\nu,h}(x,y)$ has been
constructed for $0<x\leq (k-1)h$ and contains the jumps of shock fronts,
weak rarefaction fronts, strong rarefaction fronts, and non-physical fronts,
which are denoted by
$\mathcal{J}=\mathcal{S}\cup\mathcal{R}\cup\mathcal{R}_{\rm b}\cup\mathcal{NP}$,
and the number of wave-fronts is finite in each interval $((m-1)h, mh)$ with $m\leq k-1$.
Now we construct the approximate solution $U^{\nu,h}(x,y)$ for $(k-1)h< x\leq kh$.
In $((k-1)h,kh)$, we
solve the homogeneous problem \eqref{eq:2.2} with initial data $U^{\nu,h}((k-1)h+,y)$.
Since $U^{\nu,h}((k-1)h+,y)$ is piecewise constant and may
be discontinuous on line $x=(k-1)h+$,
then we solve the initial Riemann problem \eqref{eq:4.1}
and the lateral Riemann problem \eqref{eq:4.9} with the initial data $U^{\nu,h}((k-1)h+,\cdot)$.
The approximate solution $U^{\nu,h}(x,y)$ consists of a finite number
of wave-fronts whose proof will be given in \S 6,
for which some may interact at $x=\tau$ for $(k-1)h<\tau <kh$, and
the corresponding Riemann problem is solved when two wave-fronts interact,
or
a wave-front hits the boundary.
Then, owing to the reaction process, we solve the nonhomogeneous problem
\eqref{eq:2.2} from $x=kh-$ to $x=kh+$ with initial data $U^{\nu,h}(kh-,\cdot)$,
that is,
\begin{eqnarray}\label{eq:4.24}
W(U^{\nu,h}(kh+,y))=W(U^{\nu,h}(kh-,y))+G(U^{\nu,h}(kh-,y))h.
\end{eqnarray}
In each step of construction, we can change the speed of a single wave-front
at a point by a quantity less than $2^{-\nu}$, so that no more than two wave-fronts
interact, and that only one wave-front hits the boundary $\Gamma_{h, k-1}$
at the non-corner point, or only one wave is generated by the corner point and
also only one wave crosses line $x=kh$.
In order to avoid the number of wave-fronts to be infinite in a finite time
for solving the homogeneous system \eqref{eq:2.2},
we also need to assign a generation order for each wave-front as stated in \cite{bressan}.
The outgoing wave-fronts are constructed and the generation orders are defined
according to the following:

\smallskip
\par\rm (i) All wave-fronts generated by the corner points can be constructed
in the same way as in \S 4.2 and have order one.

\smallskip
\par\rm (ii) A wave-front of order $k$ hits the boundary at the non-corner
point $(\tau, g_{h}(\tau))$.
Then the
generated
lateral Riemann problem
is solved by the accurate $\nu$-Riemann solver
as in \S 4.2 and the generation order of the new wave
from $(\tau, g_{h}(\tau))$ is set to be $k+1$.

\smallskip
\par\rm (iii) The $i_{1}$-wave-front $\alpha_{i_{1}}$ and $i_{2}$-wave
front $\beta_{i_{2}}$, with order $n_{1}$ and $n_{2}$, respectively,
interact at $(\tau, y_{\tau})$ for some $y_{\tau}$.
Assume that $\alpha_{i_{1}}$ lies below $\beta_{i_{2}}$.
Then the new wave-fronts and their generation orders are given as follows:

\par$\quad$ $\rm (iii)_{1}$  For the case $n_{1}, n_{2}<\nu$,
the outgoing wave-fronts generated from $(\tau, y_{\tau})$
are constructed by the accurate Riemann solver, and
the generation order of the $j$-th wave is given by
\begin{eqnarray}\label{eq:4.26}
\left\{
\begin{array}{llll}
     \max\{n_{1}, n_{2}\}+1,\quad &j\neq i_{1},\ i_{2}, \\[1mm]
\min\{n_{1}, n_{2}\}, \quad &j= i_{1}= i_{2},\\[1mm]
n_{1},\quad &j= i_{1}\neq i_{2}, \\[1mm]
  n_{2}, \quad &j\neq i_{1}= i_{2}.
     \end{array}
     \right.
\end{eqnarray}

\par$\quad$ $\rm (iii)_{2}$ If $\max\{n_{1}, n_{2}\}=\nu$, the outgoing wave-fronts
generated from $(\tau, y_{\tau})$ are given by the simplified Riemann solver,
and the outgoing $i_{1}$-wave and $i_{2}$-wave are defined according to \eqref{eq:4.6}.
The generation order of the outgoing non-physical wave-front is $\nu+1$.

\par$\quad$ $\rm (iii)_{3}$ If $n_{1}=\nu+1$ and $n_{2}\leq\nu$, then $\alpha_{i_{1}}$ 
is a non-physical wave-front.
We use the simplified Riemann solver to construct the outgoing wave-fronts from $(\tau, y_{\tau})$.
The generation order of the outgoing non-physical wave-front is $\nu+1$,
while the outgoing physical wave-front has the same generation order with  $n_{2}$.

\smallskip
\par $\rm (iv)$ If the change of the angle of boundary is larger than $\frac{1}{\nu}$
and the weak wave is physical, then we employ the accurate Riemann solver to
solve the Riemann problem. If the change of the angle of the boundary is
less than $\frac{1}{\nu}$, then we ignore this perturbation.

\smallskip
\par $\rm (v)$  When a wave-front $\alpha$ of order $n$ hits line $x=kh$,
then the new wave-fronts and their generation orders are defined as follows:

\par$\quad$ $\rm (v)_{1}$  When $n\leq\nu$, the outgoing wave-fronts generated from the points
on line $x=kh+$ are constructed by the accurate Riemann solver, and
the generation order is defined as in \eqref{eq:4.26}.

\par$\quad$ $\rm (v)_{2}$ When the incoming wave-front $\alpha$ is non-physical,
then let it cross line $x=kh$ with the generation order unchanged.

\par$\quad$ $\rm (v)_{3}$ For the strong rarefaction fronts cross line $x=kh$,
its generation order is unchanged with order $1$.

This completes the construction of the approximate solution.

\medskip
\par In each step of construction, we have the following property for $Z^{\nu,h}(x,y)$:

\begin{lemma}\label{lem:4.4}
If $0\leq Z^{\nu,h}((k-1)h+,y)\leq 1$, then $Z^{\nu,h}(x,y)$ is unchanged
when crossing the non-physical waves so that
\begin{eqnarray}\label{eq:4.23}
0\leq Z^{\nu,h}(x,y)\leq 1
\end{eqnarray}
for all $k\in \mathbb{N}_{+}$ and $(x, y)\in((k-1)h, kh)\times(g_h(\tau), \infty)$.
\end{lemma}

\begin{proof}
We consider the Riemann problem with $ Z^{\nu,h}(kh,y)$ as its initial data.
In solving the accurate Riemann problem, we
notice that $Z^{\nu,h}$ is unchanged when it crosses the $j$-wave-fronts, $j=1,5$,
and $i$-contact discontinuities, $i=2,3$.
The discontinuities of $Z^{\nu,h}$ propagate only
along the $4$-contact discontinuities.  Thus, $0\leq Z^{\nu,h}(x,y)\leq 1$.

We now consider the simplified Riemann solution. Without loss of generality,
suppose that $\alpha_{2(3,4)}$, which connects $\{U_{m}, U_a\}$,
interacts with $5$-shock $\alpha_5$  with $\{U_b, U_{m}\}$
as its below and above states at $x=\tau$.
Then it is solved by the simplified Riemann solver.
\vspace{5pt}
\begin{center}\label{fig7}
\begin{tikzpicture}[scale=1.4]
\draw [ thick] (-3,1)--(1,-0.5);
\draw [thick](-3,0)--(1.5, 0.5);
\draw [dashed][thick](-0.95,0.25)--(1.05,1.3);
\node at (-1.3, -0.4) {$U_b$};
\node at (-2.3, 0.4) {$U_{m}$};
\node at (-1, 1) {$U_a$};
\node at (0.9, 0.8) {$U'_a$};
\node at (1, 0) {$U'_{m}$};
\node at (-3.1, -0.1) {$\alpha_{5}$};
\node at (-3, 1.2) {$\alpha_{2(3,4)}$};
\node at (1.7, 0.6) {$\alpha_{5}$};
\node at (1.5, -0.5) {$\alpha_{2(3,4)}$};
\node at (1.2, 1.4) {$\epsilon$};
\node [below] at (-1.0, -0.8)
{Fig. 4.5};
\end{tikzpicture}
\end{center}
Assume that $Z_a$ is the 5th component of $U_a$,
and also $Z_m,\ Z_b,\ Z'_{m},\ Z'_b$ etc. denote the same meaning.
Then, by the wave curves defined in \S 2, we have
\begin{eqnarray*}
Z_{m}=Z_b,\,\,\, Z_a=Z_m+\alpha_{4},\qquad\,\, Z'_{m}=Z_b+\alpha_{4},\,\,\,  Z'_a=Z'_m,
\end{eqnarray*}
which leads to $Z_a=Z'_a=Z_b+\alpha_{4}$.
This implies that $Z^{\nu+, h}(\tau, \cdot)$ is unchanged when it crosses the non-physical waves.
Therefore,
\begin{eqnarray*}
0\leq Z^{\nu,h}(\tau+,y)\leq 1 \qquad\,\, \mbox{for}\,\,  (\tau, y)\in((k-1)h, kh)\times(g_h(\tau), \infty).
\end{eqnarray*}
This completes the proof.
\end{proof}

At $x=kh$, we use \eqref{eq:4.24} to obtain
\begin{eqnarray}
&&\rho_{+}^{\nu, h} u_{+}^{\nu,h}=\rho_{-}^{\nu,h}u_{-}^{\nu,h},\label{eq:4.27-1} \\[3pt]
&&\rho_{+}^{\nu,h} (u_{+}^{\nu,h})^{2}+p_+^{\nu,h}=\rho_{-}^{\nu,h} (u_{-}^{\nu,h})^{2}+p_-^{\nu,h},\label{eq:4.27-2} \\[5pt]
&&\rho_{+}^{\nu,h} u_{+}^{\nu,h} v_{+}^{\nu,h}=\rho_{-}^{\nu,h} u_{-}^{\nu,h}v_{-}^{\nu,h},\label{eq:4.27-3}\\[5pt]
&&\rho_{+}^{\nu,h}u_{+}^{\nu,h}\Big(\tilde{h}_{+}^{\nu,h}+
\frac{(u_{+}^{\nu,h})^2+(v_{+}^{\nu,h})^2}{2}\Big) \nonumber\\
&&\qquad =\rho_{-}^{\nu,h}u_{-}^{\nu,h}\Big(\tilde{h}_{-}^{\nu,h}+
\frac{(u_{-}^{\nu,h})^2+(v_{-}^{\nu,h})^2}{2}\Big)
+\tilde{q}\rho_-^{\nu,h}Z_{-}^{\nu,h}\phi(T_{-}^{\nu,h})h,\label{eq:4.27-4}
\\[5pt]
&&\rho_{+}^{\nu,h}u_{+}^{\nu,h}Z_{+}^{\nu,h}=\rho_{-}^{\nu,h}u_{-}^{\nu,h}Z_{-}^{\nu,h}
-\rho_{-}^{\nu,h}Z_{-}^{\nu,h}\phi(T_{-}^{\nu,h})h,\label{eq:4.27-5}
\end{eqnarray}
where $U_{+}^{\nu,h}$ and $U_{-}^{\nu,h}$ denote $U^{\nu,h}(x,y)$ taking values at $x=kh\pm$, respectively.

For equations \eqref{eq:4.27-1}--\eqref{eq:4.27-5}, we have the following properties
that indicate the change in $U^{\nu,h}$ due to the reaction process.

\begin{lemma}\label{lem:4.5}
For $h>0$ sufficiently small,
there exist positive constants $L$ and $\bar{K}$, independent of $(\nu,h)$, such that
\begin{eqnarray}\label{eq:4.28}
\begin{split}
 &T^{\nu,h}(kh+,y)\ge T^{\nu,h}(kh-,y) \ge \bar{K}>0,\\[1mm]
 &u^{\nu,h}(kh+,y)-u^{\nu,h}(kh-,y) =O(1)\|\bar{Z}\|_{\infty}e^{-L kh}h ,\\[1mm]
 &v^{\nu,h}(kh+,y)-v^{\nu,h}(kh-,y) =0,\\[5pt]
 &p^{\nu,h}(kh+,y)-p^{\nu,h}(kh-,y) =O(1)\|\bar{Z}\|_{\infty}e^{-L kh}h,\\[1mm]
 &\rho^{\nu,h}(kh+,y)-\rho^{\nu,h}(kh-,y) =O(1)\|\bar{Z}\|_{\infty}e^{-L kh}h,\\[1mm]
 &Z^{\nu,h}(kh+,y)- Z^{\nu,h}(kh-,y)=O(1)\|\bar{Z}\|_{\infty}e^{-L kh}h.
\end{split}
\end{eqnarray}
In addition, for $k\in \mathbf{N}_{+}$ and $y\geq g_{h}(kh)$,
\begin{eqnarray}\label{eq:4.29}
\begin{split}
&0 \leq Z^{\nu,h}(kh+,y)\leq Z^{\nu,h}(kh-,y)e^{-Lh},\\[1mm]
&0 \leq Z^{\nu,h}(kh+,y)\le O(1)\|\bar{Z}\|_{\infty}e^{-L kh}h\leq1.
\end{split}
\end{eqnarray}
\end{lemma}

\begin{proof} We carry out the proof by induction on $k$.

For $k=0$, by $\mathbf{(H_{2})}$, we have
\begin{eqnarray*}
T^{\nu,h}(0+,y)=\bar{T}>T^{*}>0,\qquad
0\leq Z^{\nu,h}(0+,y)=\bar{Z}\leq 1.
\end{eqnarray*}

Suppose that \eqref{eq:4.28}--\eqref{eq:4.29} hold for $k=l-1$.
We now prove \eqref{eq:4.28}--\eqref{eq:4.29} for $k=l$.
To see this, for $h>0$ sufficiently small, we first have
\begin{eqnarray*}
&&T^{\nu,h}(lh+,y)-T^{\nu,h}(lh-,y)\\[3pt]
&&=\frac{(\gamma -1)\tilde{q}\big((u^{\nu,h})^2-\frac{c^{2}(T^{\nu,h})}{\gamma}\big)(lh-,y)}
{R\rho^{\nu,h}u^{\nu,h}\big((u^{\nu,h})^2-c^{2}(T^{\nu,h})\big)(lh-,y)}
\big(\rho^{\nu,h}Z^{\nu,h}\phi(T^{\nu,h})\big)(lh-,y)h+O(h^{2}).
\end{eqnarray*}
Since $\big(u^{\nu,h}\big)^2(kh-,y)>c^{2}(T^{\nu,h}(kh-,y))$ and $\gamma>1$, then
\begin{eqnarray*}
T^{\nu,h}(kh+,y)\ge T^{\nu,h}(kh-,y),
\end{eqnarray*}
which shows that the temperature $T^{\nu,h}$ does not decrease owing to the reaction process.
Then, by the induction hypothesis, there exists a constant $\bar{K}$
independent of $(\nu,h)$ such that the first estimate in $\eqref{eq:4.28}$ holds.
\end{proof}

By $\eqref{eq:4.27-1}$ and $\eqref{eq:4.27-5}$, we have
\begin{eqnarray*}
Z^{\nu,h}(lh+,y)=\Big(1-\frac {\phi(T^{\nu,h})(lh-,y)}{u^{\nu,h}(lh-,y)}h\Big)Z^{\nu,h}(lh-,y).
\end{eqnarray*}
Since $\phi(T)$ is assumed to be Lipschitz continuous, nonnegative, and increasing, then
there exists a constant $L >0$ such that
\begin{eqnarray*}
Z^{\nu,h}(lh+,y)-Z^{\nu,h}(lh-,y)\leq -Z^{\nu,h}(lh-,y)Lh,
\end{eqnarray*}
which implies that
\begin{eqnarray*}
Z^h(lh+,y)\leq Z^{\nu,h}(lh-,y)(1-Lh)\leq Z^{\nu,h}(lh-,y)e^{-Lh}.
\end{eqnarray*}
On the other hand, by Lemma \ref{lem:4.4}, we have
\begin{eqnarray*}
Z^{\nu,h}(lh-,y)\le Z^h((l-1)h+,y).
\end{eqnarray*}
Then, by induction,
\begin{eqnarray*}
Z^h(lh-,y)\le \|\bar{Z}\|_{\infty}e^{-L(l-1)h}O(h).
\end{eqnarray*}
Thus, for $h$ sufficiently small, we have
$$
Z^h(lh+,y)\leq Z^{\nu,h}(lh-,y)e^{-Lh}
\leq Z^{\nu,h}((l-1)h+,y)e^{-Lh}
\leq \|\bar{Z}\|_{\infty}e^{-Llh}O(h),
$$
which leads to the last estimate in ${\eqref{eq:4.28}}$
and the esimates in \eqref{eq:4.29}.

To prove the remaining part of \eqref{eq:4.28},
let $U_{+}^{\nu,h}=(u_{+}^{\nu,h}, v_{+}^{\nu,h}, p_{+}^{\nu,h}, \rho_{+}^{\nu,h}, Z_{+}^{\nu,h})$
and $U_{-}^{\nu,h}=(u_{-}^{\nu,h}, v_{-}^{\nu,h}, p_{-}^{\nu,h}, \rho_{-}^{\nu,h}, Z_{-}^{\nu,h})$
denote $U^{\nu,h}(x,y)$ taking values at $x=kh\pm$, respectively, and also $T_{\pm}^{\nu,h}=T^{\nu,h}(lh\pm,y)$.
Then, by $\eqref{eq:4.27-1}$--$\eqref{eq:4.27-3}$, we have
$$
u_{+}^{\nu,h}=\frac{\rho_{-}^{\nu,h}u_{-}^{\nu,h}}{\rho_{+}^{\nu,h}},\quad v_{+}^{\nu,h}=v_{-}^{\nu,h},\quad
p_{+}^{\nu,h}=p_{-}^{\nu,h}+ \rho_{-}^{\nu,h}(u_{-}^{\nu,h})^2-\frac{\big(\rho_{-}^{\nu,h}u_{-}^{\nu,h}\big)^2}{\rho_{+}^{\nu,h}}.
$$

Substitute $(u_{+}^{\nu,h},v_{+}^{\nu,h}, p_{+}^{\nu,h})$  into $\eqref{eq:4.27-4}$ , we have
\begin{eqnarray*}
&&\frac{(\rho_{-}^{\nu,h} u_{-}^{\nu,h})^2}{(\rho_{+}^{\nu,h})^2}
-\frac{2\gamma\big(\rho_{-}^{\nu,h}(u_{-}^{\nu,h})^2+p_{-}^{\nu,h}\big)}{(\gamma+1)\rho_{+}^{\nu,h}}
+\frac{2\gamma+p_{-}^{\nu,h}+(\gamma-1)\rho_{-}^{\nu,h}(u_{-}^{\nu,h})^2}{(\gamma+1)\rho_{-}^{\nu,h}}\\[3pt]
&&\quad +\frac{2(\gamma-1)\tilde{q}Z_{-}^{\nu,h}\phi(T_{-}^{\nu,h})}{(\gamma+1)u_{-}^{\nu,h}}h=0.
\end{eqnarray*}
Therefore, we obtain
\begin{eqnarray*}
\frac{1}{\rho_{+}^{\nu,h}}=\frac{\gamma\big(\rho_{-}^{\nu,h}(u_{-}^{\nu,h})^2+p_{-}^{\nu,h}\big)
+\sqrt{\big(\rho_{-}^{\nu,h}(u_{-}^{\nu,h})^2-\gamma p_{-}^{\nu,h}\big)^2-2\tilde{q}(\gamma^2-1)(\rho_{-}^{\nu,h})^2u_{-}^{\nu,h}Z_{-}^{\nu,h}\phi(T_{-}^{\nu,h})}}
{(\gamma+1)(\rho_{-}^{\nu,h}u_{-}^{\nu,h})^2}.
\end{eqnarray*}
Using the Taylor expansion, we have
\begin{eqnarray*}
\frac{1}{\rho_{+}^{\nu,h}}=\frac{1}{\rho_{-}^{\nu,h}}+Z_{-}^{\nu,h}O(h).
\end{eqnarray*}
By \eqref{eq:4.29}, it follows that
\begin{eqnarray*}
\rho_{+}^{\nu,h}-\rho_{-}^{\nu,h}=\|\bar{Z}\|_{\infty}e^{-Llh}O(h).
\end{eqnarray*}
For $u_{+}^{\nu,h}$, we have
\begin{eqnarray*}
u_{+}^{\nu,h}-u_{-}^{\nu,h}=\Big(\frac{\rho_{-}^{\nu,h}}{\rho_{+}^{\nu,h}}-1\Big)u_{-}^{\nu,h},
\end{eqnarray*}
which leads to the desire result for $u_{+}^{\nu,h}$.

Similar arguments also apply to $p_{+}^{\nu,h}$. This completes the proof.

\section{ Glimm-Type Functional}\setcounter{equation}{0}

In this section, we denote $U^{\nu, h}(x,y)$ as the $(\nu, h)$--approximate solution constructed by
the fractional-step wave-front tracking method as in \S 3.
On every interval $((k-1)h, kh), k\geq1$, $U^{\nu, h}(x,y)$ consists of shocks, rarefaction fronts,
and non-physical fronts. We distinguish strong $5$-rarefaction wave-fronts from the others as follows:

\begin{definition}\label{def:5.1}
A wave-front $s$ in $U^{\nu, h}(x,y)$ is said to be a front of the strong rarefaction wave,
provided that $s$ is a $5$-rarefaction wave-front with order $1$.
If a wave $\alpha$ is a shock or $5$-rarefaction wave-front with the generation order $n\geq 2$,
or $\alpha$ is a non-physical front, then
$\alpha$ is called a weak wave.
\end{definition}

Let $y_{\alpha}(t)$ and $\alpha(t)$ be the location and magnitude for any weak wave $\alpha$ in $U^{\nu, h}(x,y)$
on line $x=t$ respectively.
For a front of the strong rarefaction wave, denote the magnitude
and location of $s$ on $x=t$ by $s(t)$ and $Y_s(t)$, respectively.
In addition, define
\begin{eqnarray*}
\Omega_{\rm Ra}(x)=\big\{\omega(A_{k})\,:\,\omega(A_{k})\leq0, A_{k}=(kh, g_h(kh)), kh\geq x\big\},
\end{eqnarray*}
where $\omega(A_{k})=\theta_{k}-\theta_{k-1}$ and {\it Ra} stands for the strong rarefaction wave.

\smallskip
We now introduce some weights for the weak waves as follows:
For any $i$-weak wave $\alpha_{i}, 1\leq i\leq4$,
and any non-physical wave $\epsilon$ at a non-interaction location $x$,
let
\begin{eqnarray*}
&&R(x,\alpha_{i},-)=\big\{s\,:\,\mbox{$s$ is a front of strong rarefaction wave with $Y_{s}(x)\leq y_{\alpha_{i}}(x)$}\big\},\\[1.5mm]
&&R(x,\epsilon,+)=\big\{s\,:\,\mbox{$s$ is a front of strong rarefaction wave with $Y_{s}(x)\geq y_{\epsilon}(x)$}\big\},
\end{eqnarray*}
and, for $K_{\rm b}>0,\ K_{\rm w}>0$, and $K_{\rm np}>0$, define
\begin{eqnarray*}
&&W(\alpha_{i},x,-)=\exp\Big(K_{\rm b}\sum\big\{|\omega|\,:\,\omega\in\Omega_{\rm Ra}(x)\big\}+K_{\rm w}\sum\big\{|s|\,:\, s\in R(x,\alpha_{i},-)\big\}\Big),\\[1mm]
&&W(\epsilon,x,+)=\exp\Big(K_{\rm np}\sum\big\{|s|\,:\, s\in R(x,\epsilon,+)\big\}\Big).
\end{eqnarray*}

\begin{definition}[Approaching waves]\label{def:5.2}
We say that two weak wave-fronts $\alpha$ and $\beta$, located at points $x_{\alpha}<x_{\beta}$
and belonging to the characteristic families $i_{\alpha}, i_{\beta}\in \{1,2, \cdots, 6\}$, respectively,
approach each other if either $i_{\alpha}> i_{\beta}$  or $i_{\alpha}=i_{\beta}$, and at least one of them is a shock.
In this case, we write $(\alpha,\beta)\in \mathcal{A}$.
\end{definition}

Then we have 
\begin{definition}\label{def:5.3}  
The strengths of weak waves in $U^{h}(x,y)$ are defined by 
\begin{eqnarray*}
\begin{split}
L^{j}_{\rm w}(U;x)&=\sum\{|\alpha_j|\,:\,\mbox{\rm $\alpha_{j}$ is a weak and physical $j$-wave}\},\quad 1\leq j\leq 5,\\[1mm]
L_{\rm np}(U;x)&=\sum\{|\epsilon|\,:\,\mbox{\rm $\epsilon$ is a non-physical wave}\},
\end{split}
\end{eqnarray*}
\end{definition}
and the wave potentials are defined by
\begin{eqnarray*}
\begin{split}
&Q_{0}(U; x)=\sum\{|\alpha_i||\beta_j|\,:\,\mbox{$\alpha_{i}$ and $\beta_j$ approach weak waves}\},\\[3pt]
&Q_{\rm Bi}(U;x)=\sum\{W(\alpha_{i},x,-)|\alpha_{i}|\,:\,\mbox{$\alpha$ is a weak $i$-wave}\},\quad 1\leq i\leq 4,\\[3pt]
&Q_{\rm B5}(U;x)=\sum\{|\alpha_{5}|\,:\,\mbox{$\alpha_{5}$ is a $5$-weak wave}\},\\[3pt]
&Q_{\rm BNP}(U; x)=\sum\{W(\epsilon,x,+)|\epsilon|\,:\,\mbox{$\epsilon$ is a non-physical wave}\},\\[3pt]
&Q_{\rm c}(U; x)=\sum\{|\omega(A_{k})|\,:\,\mbox{$A_{k}=(kh, g_h(kh))$ with $\omega(A_{k})>0$ for $kh>x$}\}.
\end{split}
\end{eqnarray*}
In order to measure the amounts of the $5$-rarefaction waves expanding before $x$, we introduce the following:

\begin{definition}\label{def:5.5}
For $x\notin \{kh:k\in\mathbf{N}_{+}\}$, define
\begin{eqnarray*}
\begin{split}
F_{1}(U;x)=\Big|TV.\{\theta(U(x,\cdot));\, [g_h(x),\infty)\}-\hat{\theta}(x)\Big|,
\end{split}
\end{eqnarray*}
where
\begin{eqnarray*}
\hat{\theta}(x)=\sum\{|\omega(A_{k})|\,:\,\mbox{\rm $A_{k}$ is a corner point $(kh, g_h(kh))$ with $\omega(A_{k})\leq0$ for $kh<x$}\}.
\end{eqnarray*}
\end{definition}

Finally, we turn to the construction of the Glimm-type functional at the non-interaction point $x$.

\begin{definition}\label{def:5.6} Define
\begin{eqnarray*}
&&L_{\rm w}(U;x)=\sum^{5}_{i=1}L^{i}_{\rm w}(U;x)+L_{\rm np}(U; x),\\
&& Q(U;x)=K_{0}Q_{0}(U; x)+K_{1}\sum^{4}_{i=1}Q_{\rm Bi}(U;x)+Q_{\rm B5}(U;x)+Q_{\rm BNP}(U;x)+K_{\rm c}Q_{\rm c}(U;x),\\
&&F_{0}(U;x)=L_{\rm w}(U;x)+KQ(U;x),\\[1mm]
&&F(U;x)=F_{1}(U;x)+C_{*}F_{0}(U;x),
\end{eqnarray*}
where $C_{*}$, $K, K_0, K_1$, and $K_{\rm c}$ are positive constants specified later.
\end{definition}

\section{Estimates of the non-reacting step}
In this section, our main task is to prove the decreasing of the Glimm-type functional
for the homogeneous system \eqref{eq:2.2} by the induction procedure.
To this end, let $U^{\nu,h}(x,y)$ be the approximate solution of system
\eqref{eq:2.2} constructed above. For any $(k-1)h<\tau<kh$, we make the following inductive assumptions:

\smallskip
\par $A_{1}(\tau-)$: Before $\tau$, the strength of every front is less than $\delta_{*}$;

\smallskip
\par $A_{2}(\tau-)$:  $\left.U^{\nu,h}(x,\cdot)\right|_{x<\tau}\in D_{\delta_{**}}(U_{\infty})$,

\smallskip
\noindent
where $\delta_{*}\in(0, \delta_{0})$ and $\delta_{**}\in(0, \min\{\delta_{0},\delta_{1}\})$
are well chosen such that the Riemann problems considered in \S 4 with the initial data satisfying $A_{1}(\tau-)$--$A_{2}(\tau-)$
are solvable.

\smallskip
\par $A_3(\tau-)$:  For any
$0<\delta_{0*}<\frac{1}{2}\big(\lim_{x\rightarrow \infty}g'_{*}(x+)-\theta_{\rm crit}\big)$,
\begin{eqnarray*}
\sum_{k}|\omega(A_{k})|\leq-\theta_{\rm crit}-\delta_{0*},
\end{eqnarray*}
where $\theta_{\rm crit}$ is a critical angle given in \eqref{eq:2.14}.

\par In order to obtain the global interaction estimates and the bound
on the amount of strong rarefaction wave-fronts,
we need some auxiliary lemmas whose proofs are similar to \cite{zh1}.

\begin{lemma}\label{lem:6.1}
There is $\delta'_{*}>0$ such that function $\theta(\Phi_{5}(\alpha_5, U))$ is a strictly increasing function of
$\alpha_5$ in $\{\alpha_5\,:\, \alpha_5 \geq-\delta'_{*},\ \Phi_{5}(\alpha_5, U)\in D_{\delta_{0}}(U_{\infty})\}$
for any $U\in D_{\delta_{0}}(U_{\infty})$. Moreover, there exist two positive constants $C_{1}$ and $C_{2}$ such that
\begin{eqnarray}\label{eq:6.1}
C_{1}|\alpha_5|\leq|\theta(\Phi_{5}(\alpha_5, U))- \theta(U)|\leq C_{2}|\alpha_5|.
\end{eqnarray}
\end{lemma}

\begin{proof}
For $\alpha_{5}\geq0$, by Lemma \ref{lem:2.1} and Lemma \ref{lem:10.1} in the appendix, we have
\begin{eqnarray*}
\begin{split}
\left.\frac{\partial\theta(\tilde{U})}{\partial\alpha_{5}}\right|_{\tilde{U}=\Phi_{5}(\alpha_5, U)}
&=\left.\Big(\nabla_{\tilde{U}}\theta(\tilde{U})\cdot\frac{\partial \tilde{U}}{\partial\alpha_{5}}\Big)\right|_{\tilde{U}=\Phi_{5}(\alpha_5, U)}\\[4pt]
&=\left.\frac{\kappa_{5}(\tilde{U})\cos\theta_{\rm ma}(\tilde{U})}{q(\tilde{U})\cos(\theta(\tilde{U})+\theta_{\rm ma}(\tilde{U}))}\right|_{\tilde{U}=\Phi_{5}(\alpha_5, U)}>0
\end{split}
\end{eqnarray*}
for any $\Phi_{5}(\alpha_5, U)\in D_{\delta_{0}}(U_{\infty})$ and $U\in D_{\delta_{0}}(U_{\infty})$.
A similar proof works for $-\delta'_{*}\leq\alpha_{5}<0$.
This completes the proof.
\end{proof}

\begin{lemma}\label{lem:6.2}
Suppose that $U_{1}, U_{2}, U_{3}\in D_{\delta_{0}}(U_{\infty})$ are three constant states such that
$\{U_{1}, U_{2}\}$ are connected by $U_{2}=\Phi(\alpha_5, \alpha_4,\alpha_3,\alpha_2,\alpha_1; U_{1})$ and
$\{U_{2}, U_{3}\}$ are connected by a non-physical wave-front with strength $\epsilon$. Then
\begin{eqnarray}\label{eq:6.2}
\left\{
\begin{array}{llll}
\Big|\Big(J\big(q,\frac{q^{2}}{2}+\frac{c^{2}}{\gamma-1}\big)+\theta\Big)(U_3)
-\Big(J\big(q,\frac{q^{2}}{2}+\frac{c^{2}}{\gamma-1}\big)+\theta\Big)(U_1)\Big|\\
$\qquad\qquad\,\,$ =O(1)\Big(\sum_{i=1}^{4}|\alpha_{i}|+|\alpha^{-}_{5}|+|\epsilon|\Big), \\[8pt]
\Big|\Big(\frac{q^{2}}{2}+\frac{c^{2}}{\gamma-1}\Big)(U_3)
   -\Big(\frac{q^{2}}{2}+\frac{c^{2}}{\gamma-1}\Big)(U_1)\Big|
   =O(1)\Big(\sum^{4}_{i=1}|\alpha_{i}|+|\alpha^{-}_{5}|+|\epsilon|\Big), \\[8pt]
\big|\frac{p_3}{\rho^{\gamma}_{3}}-\frac{p_{1}}{\rho^{\gamma}_{1}}\big|=O(1)\Big(\sum^{4}_{i=1}|\alpha_{i}|+|\alpha^{-}_{5}|+|\epsilon|\Big),\\[8pt]
\big|Z_{3}-Z_{1}\big|=O(1)\Big(\sum^{4}_{i=1}|\alpha_{i}|+|\alpha^{-}_{5}|+|\epsilon|\Big),
     \end{array}
     \right.
\end{eqnarray}
where $\alpha^{-}_{5}=\min\{\alpha_5, 0\}$, 
and the bound of $O(1)$ is independent of $U_{k}, 1\leq k\leq 3$.
\end{lemma}

\begin{proof}
By the mean-value theorem and the fact that, for $\alpha_{5}>0$,
\begin{eqnarray*}
\begin{split}
&\Big(J\big(q,\frac{q^{2}}{2}+\frac{c^{2}}{\gamma-1}\big)+\theta\Big)\big(\Phi_{5}(\alpha_5, U)\big)
=\Big(J\big(q,\frac{q^{2}}{2}+\frac{c^{2}}{\gamma-1}\big)+\theta\Big)(U),\\[1mm]
&\Big(\frac{q^{2}}{2}+\frac{c^{2}}{\gamma-1}\Big)\big(\Phi_{5}(\alpha_5, U)\big)
   =\Big(\frac{q^{2}}{2}+\frac{c^{2}}{\gamma-1}\Big)(U),
\end{split}
\end{eqnarray*}
we conclude the proof.
\end{proof}

Let $\alpha_i$ and $\beta_j$, $1\leq i,j\leq 5$,
be the approaching weak waves, and let $\epsilon$ be a non-physical wave.
Let $s$ denote the strong wave-front and $\omega_{k-1}>0$ for some $k\geq1$.
Define
\begin{eqnarray}\label{eq:6.3}
E_{\nu, h}(\tau)=\left\{
\begin{array}{llll}
|\alpha_i||\beta_j|, &\quad (\alpha_i,\beta_j)\in \mathcal{A},1\leq i, j\leq5,\\[4pt]
|\alpha_i||\epsilon|, &\quad 1\leq i\leq5, \\[4pt]
|\alpha_i||s|, &\quad 1\leq i\leq5,\\[4pt]
|\alpha_1|,\\[4.5pt]
|\omega_{k-1}|, &\quad \omega_{k-1}>0, \   \  k\geq1,\\[4pt]
|\alpha_5|,
\end{array}\right.
\end{eqnarray}
where $\tau \in ((k-1)h, kh)$ act as the times 
when two waves interact
or a weak wave reflects at a non-corner point,
or weak waves generated by the corner points.
Here we assume that only one interaction occurs at $x=\tau$.
Then, for the non-reacting flow past a bending wall, we have the following proposition.

\begin{proposition}\label{prop:6.1}
Assume that the hypotheses $A_{1}(\tau-)$--$A_{2}(\tau-)$ hold for any $(k-1)h<\tau<kh$.
Then there exist positive constants $\delta_2, C_{*}, K_{0}, K_{1}, K_{2}$, and $K_{\rm c}$,
independent of $(\nu,\tau)$, such that, if
 $F(U^{\nu, h}; \tau-)<\delta_2$,
\begin{eqnarray}
&&F(U^{\nu, h}; \tau+)\leq F(U^{\nu, h}; \tau-)-\frac{1}{4}E_{\nu, h}(\tau),\label{eq:6.4}\\
&& Q(U^{\nu, h}; \tau+)\leq Q(U^{\nu, h}; \tau-)-\frac{1}{4}E_{\nu, h}(\tau).\label{eq:6.5}
\end{eqnarray}
Moreover, for $(k-1)h<x<kh$, the following hold{\rm :}
\begin{enumerate}
\item[\rm (i)] Before $\tau_1$, the strength of every weak wave-front is less than $\delta_{*}${\rm ;}

\smallskip
\item[\rm (ii)] There is a constant $C_{3}>0$ independent of $(\nu,h)$ such that
 \begin{eqnarray}
|s_{0}(x)|\leq\frac{C_{3}}{\nu}
\label{eq:6.6}
\end{eqnarray}
for any strong rarefaction front $s_0(x)${\rm ;}

\smallskip
\item[\rm (iii)] For given $\nu\in\mathbb{N}_{+}$ sufficiently large,
the number of wave-fronts in $U^{\nu,h}(x,y)$ is finite,
and the total amount of a non-physical wave $\epsilon_{0}(x)$
is of order $O(1)2^{-\nu}$, that is,
 there exists a constant $C_{\rm np}>0$ such that
\begin{eqnarray}
\sum_{\epsilon_{0}}|\epsilon_{0}(x)|\leq\frac{C_{\rm np}}{2^{\nu}}
\qquad \mbox{for $x<\tau_1$},
\label{eq:6.7}
\end{eqnarray}
where $\tau_1<kh$ is the point next to $\tau$ such that the interaction occurs on $x=\tau_1$.
\end{enumerate}
\end{proposition}

\begin{proof}
Without loss of generality, we assume that there is only one interaction of waves at $x=\tau$.
We now focus only on the following three cases, since the others are similar to \cite{zh1}.

\smallskip
\par $\mathbf{Case}\ 1.$ {\it Strong rarefaction wave and weak $2$-wave interaction}.
Suppose that a front of strong rarefaction wave $s_0$ with $\{U_b, U_m\}$ as its below and above states
interacts with a weak wave-front $\alpha_2$ that connects $\{U_m, U_a\}$ on line $x=\tau$.
Let the outgoing waves be denoted by $s'_{0}$ and $\gamma_{j}, 1\leq j\leq 4$,
and $\epsilon$ be the nonphysical wave, and let $\{U'_{m1}, U'_{m2}, U'_a\}$ be denoted as
the constant states between $\gamma_{1}$ and $\gamma_{2(3,4)}$, $\gamma_{2(3,4)}$ and $s'_0$, and $s'_0$ and $\epsilon$,
respectively.
Then, by the Glimm interaction estimates, we have
\begin{eqnarray*}
\begin{split}
&\gamma_2=\alpha_2+O(1)|\alpha_2||s_0|,\qquad \gamma_j=O(1)|\alpha_2||s_0|, \,\,\, j=1,3,4,\\[3pt]
&s'_0=s_0+O(1)|\alpha_2||s_0|, \qquad  \epsilon=O(1)|\alpha_2||s_0|,
\end{split}
\end{eqnarray*}
where the bound of $O(1)$ depends only on the system.
With the above estimates, we conclude
\begin{eqnarray*}
&&L_{\rm w}(U^{\nu, h};\tau+)-L_{\rm w}(U^{\nu, h};\tau-)=O(1)|\alpha_{2}||s_{0}|,\\[7pt]
&&K_{0}Q_0(U^{\nu, h};\tau+)+Q_{B5}(U^{\nu, h};\tau+)+Q_{BNP}(U^{\nu, h};\tau+)+K_{\rm c}Q_{\rm c}(U^{\nu, h};\tau+)\\[3pt]
&&\qquad\,\, -K_{0}Q_0(U^{\nu, h};\tau-)-Q_{B5}(U^{\nu, h};\tau-)-Q_{BNP}(U^{\nu, h};\tau-)-K_{\rm c}Q_{\rm c}(U^{\nu, h};\tau-)\\[3pt]
&&\quad =O(1)|\alpha_2||s_0|L_{\rm w}(U^{\nu, h};\tau-),
\end{eqnarray*}
and, for $K_{\rm w}$ sufficiently large and $L_{\rm w}(U^{\nu,h};\tau-)$ sufficiently small,
\begin{eqnarray*}
\begin{split}
&\sum_{1\leq i\leq 4}Q_{Bi}(U^{\nu, h};\tau+)-\sum_{1\leq i\leq 4}Q_{Bi}(U^{\nu, h};\tau-)\\
&=|\alpha_2|W(\alpha_2,\tau-,-)\Big(\big(1+O(1)|s_0|\big)e^{-K_{\rm w}|s_0|}-1\Big)\\[1mm]
&\,\, \, \,\,\, +\sum_{1\leq i\leq 4}\sum_{\beta_{i}\neq \gamma_i}|\beta_{i}|
W(\beta_{i},-, \tau-)\big(e^{O(1)K_{\rm w}|\alpha_2||s_0|}-1\big)\\
&\leq-\frac{K_{\rm w}}{2}|\alpha_2||s_0|,
\end{split}
\end{eqnarray*}
which leads to
\begin{eqnarray*}
\begin{split}
Q(U^{\nu, h};\tau+)-Q(U^{\nu, h};\tau-)\leq-\frac{K_{\rm w}}{2}|\alpha_2||s_0|.
\end{split}
\end{eqnarray*}
To estimate $F_{1}(U^{\nu, h}, \tau+)$, let
\begin{eqnarray*}
\begin{split}
\theta(\Phi(\alpha; U))=f(\alpha_{5},\alpha_{2(3,4)},\alpha_{1};U).
\end{split}
\end{eqnarray*}
Then we have
\begin{eqnarray*}
\begin{split}
& \theta(U'_a)-\theta(U'_{m2})-\theta(U_{m})+\theta(U_b)\\[3pt]
&=f(s'_0,\gamma_{2(3,4)},\gamma_{1};U_b)-f(0,\gamma_{2(3,4)},\gamma_{1};U_b)
-f(s_0,0,0;U_b)+f(0,0,0;U_b)\\[3pt]
&=f(s_0,\alpha_2,0;U_b)-f(0,\alpha_2,0;U_b)
-f(s_0,0,0;U_l)+f(0,0,0;U_b)+O(1)|\alpha_2||s_0|\\[3pt]
&=O(1)|\alpha_2||s_0|.
\end{split}
\end{eqnarray*}
Similarly,
\begin{eqnarray*}
\theta(U'_{m1})-\theta(U_b)-\theta(U_a)+\theta(U_{m})
=O(1)|\alpha_2||s_0|,\qquad
\theta(U_a)-\theta(U'_a)=O(1)|\alpha_2||s_0|.
\end{eqnarray*}
Then,  by Lemma \ref{lem:6.1}, we have
\begin{eqnarray*}
\begin{split}
&TV.\{\theta(U^{\nu,h}(\tau+,\cdot));\, [g_{h}(\tau+),\infty)\}-TV.\{\theta(U^{\nu,h}(\tau-,\cdot));\,[g_{h}(\tau-),\infty)\}\\[3pt]
&=|\theta(U_a)-\theta(U'_a)|+|\theta(U'_a)-\theta(U'_{m2})|+|\theta(U'_{m1})-\theta(U_b)|\\[3pt]
&\ \ \  -|\theta(U_a)-\theta(U_{m})|-|\theta(U_{m})-\theta(U_b)|\\[3pt]
&=O(1)|\alpha_2||s_0|,
\end{split}
\end{eqnarray*}
where we have used the fact that $\theta(U'_{m2})=\theta(U'_{m1})$. Therefore, we obtain
\begin{eqnarray*}
F_{1}(U^{\nu, h}; \tau+)-F_{1}(U^{\nu, h}; \tau-)=O(1)|\alpha_2||s_0|.
\end{eqnarray*}
Finally, with the estimates on $L_{\rm w}(U^{\nu,h};\tau+)$, $Q(U^{\nu,h};\tau+)$,
and $F_{1}(U^{\nu,h}; \tau+)$,
and choosing $K$ and $C_{*}$ sufficiently large, we conclude
\begin{eqnarray*}
F(U^{\nu, h}; \tau+)-F(U^{\nu, h}; \tau-)\leq-\frac{1}{4}|\alpha_2||s_0|.
\end{eqnarray*}

\smallskip
\par $\mathbf{Case}\ 2.$ {\it Reflection at the approximate boundary}. Let a weak wave $\alpha_1$
that hits the approximate boundary at a non-corner point $(\tau, g_h(\tau))$
be the outgoing wave. Then, by Lemma \ref{lem:4.3}, we have
\begin{eqnarray*}
\begin{split}
\varepsilon_5=O(1)\alpha_1,
\end{split}
\end{eqnarray*}
where the bound of $O(1)$ depends only on the system.
Moreover, for $L_{\rm w}(U^{\nu,h};\tau-)$ sufficiently small and $K_1$ sufficiently large, we have
\begin{eqnarray*}
&&L_{\rm w}(U^{\nu, h};\tau+)-L_{\rm w}(U^{\nu, h};\tau-)=-|\alpha_1|+O(1)|\alpha_1|,\\[1mm]
&&Q(U^{\nu, h};\tau+)-Q(U^{\nu, h};\tau-)\leq-\frac{1}{2}|\alpha_1|, \\[1mm]
&&F_1(U^{\nu, h};\tau+)-F_1(U^{\nu, h};\tau-)=O(1)|\alpha_1|.
\end{eqnarray*}
Therefore, we can choose $K$ and $C_{*}$ sufficiently large to obtain
\begin{eqnarray*}
\begin{split}
 F(U^{\nu, h};\tau+)-F(U^{\nu, h};\tau-)\leq-\frac{1}{4}|\alpha_1|.
\end{split}
\end{eqnarray*}

\smallskip
\par $\mathbf{Case}\ 3.$ {\it New waves generated by the corner points}.
We now consider the non-reaction flow past a corner point $A_{k-1}=((k-1)h, g_h((k-1)h))$ for some $k>0$,
that is, $\tau=(k-1)h$. 
Let $\varepsilon_5$ be the new wave,
and let $\omega_{k-1}=\omega(A_{k-1})$ be the change of angle.

\par If $\omega_{k-1}\leq 0$, then, by Lemma \ref{lem:4.2}(2), we can choose suitable $K_b$ such that
\begin{eqnarray*}
 Q(U^{\nu, h};\tau+)-Q(U^{\nu, h};\tau-)\leq-\frac{K_b}{2}|\omega_k|L_{\rm w}(U^{\nu, h};\tau-).
\end{eqnarray*}
In addition,
\begin{eqnarray*}
\begin{split}
 L_{\rm w}(U^{\nu, h};\tau+)-L_{\rm w}(U^{\nu, h};\tau-)=0, \quad  \ F_0(U^{\nu, h};\tau+) \leq F_0(U^{\nu, h};\tau-).
\end{split}
\end{eqnarray*}
Hence, for $K>0$ and $C_*>0$, it follows that
\begin{eqnarray*}
\begin{split}
 F(U^{\nu, h};\tau+)-F(U^{\nu, h};\tau-)\leq 0.
\end{split}
\end{eqnarray*}
\par If $\omega_{k-1}>0$, then, by Lemma \ref{lem:4.2}(i), for $L_{\rm w}(U^{\nu, h};\tau-)$ sufficiently small,
 we can choose suitable $K_{\rm c}$ such that
\begin{eqnarray*}
\begin{split}
 Q(U^{\nu, h};\tau+)-Q(U^{\nu, h};\tau-)\leq-\frac{1}{2}|\omega_{k-1}|.
\end{split}
\end{eqnarray*}
Moreover,
\begin{eqnarray*}
\begin{split}
 L_{\rm w}(U^{\nu, h};\tau+)-L_{\rm w}(U^{\nu, h};\tau-)=O(1)|\omega_{k-1}|, \  \ F_1(U^{\nu, h};\tau+)- F_1(U^{\nu, h};\tau-)=O(1)|\omega_{k-1}|.
\end{split}
\end{eqnarray*}
Thus we can choose $K$ and $C_*$ large enough such that
\begin{eqnarray*}
\begin{split}
 F(U^{\nu, h};\tau+)-F(U^{\nu, h};\tau-)\leq -\frac{1}{4}|\omega_{k-1}|.
\end{split}
\end{eqnarray*}

With the estimates in $\mathbf{Cases}$ $1$--$3$,
we can choose positive constants $\delta_2, C_{*}, K_{0}, K_{1}, K_{2}$, and $K_{\rm c}$
independent of $(\tau, \nu)$ such that, if
$F(U^{\nu, h}; \tau-)<\delta_2$, then \eqref{eq:6.5}--\eqref{eq:6.6} hold and, before $\tau_1$,
the strength of every weak wave-front is less than $\delta_*$.

\smallskip
To prove statement {\rm (ii)}, we introduce the following functional for any strong rarefaction front $s_0(x)$:
\begin{eqnarray*}
F_s(U^{\nu, h}; x)=|s_0(x)|e^{L_s(U^{\nu, h};x)+K_{2}Q(U^{\nu, h};x)},
\end{eqnarray*}
where
\begin{eqnarray*}
\begin{split}
L_s(U^{\nu, h};x)=\sum_{\alpha}\{|\alpha(x)|\,:\,\mbox{$\alpha$ is a weak wave approaching $s_0$}\},
\end{split}
\end{eqnarray*}
and $K_{2}$ is a positive constant. Then we can carry out the steps as in \cite{amadori, bressan}
to obtain \eqref{eq:6.6}.

\smallskip
Finally, we complete the proof of {\rm (iii)}. From estimate \eqref{eq:6.5},
we have known that the total interaction potential $Q(U^{\nu, h}; x)$ is decreasing
and bounded when it crosses an interaction time $x$.
Then we conclude that, when the orders of the incoming waves are less than $\nu$
or the change of the angle of the boundary is larger than $\frac{1}{\nu}$,
and $Q(U^{\nu, h}; x)$ decreases by at least $-\frac{2^{-\nu}}{4}$ in
these interactions.
Therefore, we can follow the argument in \cite{amadori, bressan} to see that
new physical waves can be only generated by this kind of interactions.
When the weak wave $\alpha$ of
$1$-family interacts with the boundary and reflects,
by solving the Riemann problem, we know that
there is only a reflected wave of $5$-family with the reflected coefficient $1$, so that
the number of the waves keeps the same, which implies that the
number of the waves is finite between $(k-1)h<x<kh$.
Since non-physical waves are generated only when physical waves
interact, the number of non-physical waves is also finite.
With the same procedure used in \cite{amadori,bressan},
it can be proved that there exists a constant $C_{\rm np}>0$
such that \eqref{eq:6.7} holds.
This completes the proof.
\end{proof}

To conclude the inductive hypothesis that $A_{2}(\tau-)$ holds and the total strength of the strong rarefaction
is finite, we need the following proposition.

\begin{proposition}\label{prop:6.2}
For any $\tau<x<\tau_{1}$ and sufficiently small $\delta_2$  satisfying $0<\delta_2<\frac{\delta_{**}}{4}$,
\begin{enumerate}
\item[\rm (i)]  There exists  $\eta_{0}=\eta(x)>0$ such that, for $y>\eta_{0}$,
\begin{eqnarray}\label{eq:6.8}
\Big|U^{\nu,h}(x, y)-U_{\infty}\Big|\leq \delta_2;
\end{eqnarray}
\item[\rm (ii)]  $\left.U^{\nu,h}(x,\cdot)\right|_{x<\tau_1}\in D_{\delta_{**}}(U_{\infty})$;
\item[\rm (iii)]  For any strong rarefaction wave-front $s_{0}(x)$,
\begin{eqnarray}\label{eq:6.9}
\sum_{s_0}\big|s_{0}(x)\big|\leq O(1)\Big(|\theta_{\rm crit}|+\delta_{0}+\delta_{2}+L_{\rm w}(U^{\nu,h};x)\Big),
\end{eqnarray}
where $\tau_1$ is the same as in Proposition {\rm \ref{prop:6.1}}.
\end{enumerate}
\end{proposition}

\begin{proof}
It suffices to prove \rm (i).
By Proposition \ref{prop:6.1}, we can choose $\eta^{*}_{0}=\eta^{*}(x)>0$ such
that there is no strong wave-front for $y>\eta^{*}_{0}$.
Then there exists $0<\varepsilon_{*}<\frac{\delta_2}{8}$  such that
\begin{eqnarray*}
TV. \big\{U^{\nu,h}(x,\cdot);\, [\eta^{*}_{0},\infty)\big\}\le O(1)\varepsilon_{*}.
\end{eqnarray*}
Then, by $\mathbf{(H_{2})}$, we can choose a constant $0<\delta_{2**}<\frac{\delta_2}{2}$ such that, for $y>\eta^{*}_{0}$,
\begin{eqnarray*}
\big|U^{\nu,h}(x,y)-U_{\infty}\big|&\leq& TV. \big\{U^{\nu,h}(x,\cdot);\, [\eta^{*}_{0},\infty)\big\}+\big|U^{\nu,h}(x,\infty)-U_{\infty}\big|\\[1mm]
&\leq& O(1)\varepsilon_{*}+\delta_{2**}<\delta_2.
\end{eqnarray*}
Facts (ii) and (iii) are the direct consequences of Lemma \ref{lem:6.1}--\ref{lem:6.2}, Proposition \ref{prop:6.1}, and \rm (i).
This completes the proof.
\end{proof}

\section{BV Bounds for the Reacting Step}\setcounter{equation}{0}

In this section, we study the BV stability of the approximate solutions $U^{\nu,h}$ for the reacting step.
As in \S 6, we have proved that the solution $U^{\nu, h}(x,y)$ is $BV$--stable for $x<kh,\ k\in\mathbf{N}_{+}$,
and the Glimm-type functional is decreasing and the total number of wave-fronts is finite for $(m-1)h<x<mh, m\leq k$.
Now we study the uniform bound on the total
variation of the approximate solutions $U^{\nu,h}(x,y)$ with respect to the mesh
length $h$ on line $x=kh$.
As a first step, we analyze the change of the wave strengths
before and after the reaction when the mesh length $h$ is sufficiently small.
Then we study the change of the Glimm-type functional for the reacting step to obtain
the uniform bound on the total variation of the approximate solutions
by employing the monotonicity of the Glimm-type functional that has been proved in \S 6.

\subsection{Local estimates on the reacting step}
We now study the change of the wave strengths at the reaction steps.
Arguing as in the construction of the approximation solutions, we assume
that there is no wave interaction on line $x=kh$.
The analysis is divided into four cases.

\smallskip
\par $\mathbf{Case}\ 1.$  {\it The change of $i$-weak wave $\tilde{\alpha}_{i}, 1\leq i\leq 5$,
after reaction} (see Fig. 7.1).
Let $\alpha_{i}$ be the $i$--weak incoming wave with order $k\leq\nu$
before reaction, which connects $\tilde{\alpha}_{i}$ to $\tilde{U}_a$ with
\begin{eqnarray*}
\tilde{U}_a=\Phi_{i}(\tilde{\alpha}_{i};\tilde{U}_b).
\end{eqnarray*}
Suppose that $\alpha_{i}, 1\leq j\leq5$, are the $j$--weak outgoing wave after reaction and related by
\begin{eqnarray*}
U_a=\Phi(\alpha; U_b).
\end{eqnarray*}
Then, by the construction of the approximate solutions and Lemma \ref{lem:4.5}, we have
\begin{eqnarray} \label{eq:7.1}
\Phi(\alpha;\tilde{U}_b+O(1)\|\bar{Z}\|_{\infty}e^{-Lkh}h)
=\Phi_{i}(\tilde{\alpha}_{i};\tilde{U}_b)+O(1)\|\bar{Z}\|_{\infty}e^{-Lkh}h.
\end{eqnarray}
\vspace{3pt}
\begin{center}\label{fig8}
\begin{tikzpicture}[scale=1.25]
\draw [ultra thick] (-1,1.8)--(-1,-1.5);
\draw [thin](-2.8,-0.8)--(-1, 0);
\draw [thin](-1,0)--(0.5,1.5);
\draw [thin](-1,0)--(1,0.5);
\draw [thin](-1,0)--(1,-0.5);
\node at (-1.8, -0.8) {$\tilde{U}_b$};
\node at (-2.1, 0) {$\tilde{U}_a$};
\node at (0.2, -0.8) {$U_b$};
\node at (-0.2, 1.3) {$U_a$};
\node at (-3, -0.7) {$\tilde{\alpha}_{i}$};
\node at (0.7, 1.7) {$\alpha_{5}$};
\node at (1.5, 0.5) {$\alpha_{2(3,4)}$};
\node at (1.3, -0.5) {$\alpha_{1}$};
\node at (-1.0, -1.7){$x=kh$};
\node [below] at (-1.0, -2.0)
{Fig. 7.1};
\end{tikzpicture}
\end{center}

\begin{lemma}\label{lem:7.1}
Equation \eqref{eq:7.1} admits a unique $C^{2}$--solution $\alpha=\alpha(\tilde{\alpha}_{i}; e^{-Lkh}h;\tilde{U}_b)$
in a neighbourhood of $(\tilde{\alpha}_{i}, h)=(0,0)$. Moreover,
\begin{eqnarray} \label{eq:7.2}
\alpha_{i}=\tilde{\alpha}_{i}+O(1)|\tilde{\alpha}_{i}|\|\bar{Z}\|_{\infty}e^{-Lkh}h,\quad
\alpha_{j}=O(1)|\tilde{\alpha}_{i}|\|\bar{Z}\|_{\infty}e^{-Lkh}h,\qquad  1\leq j \neq i\le 5.
\end{eqnarray}
\end{lemma}

\begin{proof} Notice that
\begin{eqnarray*}
&&\left.\det\Big( \frac{\partial\Phi\big(\alpha_5,\alpha_4,\alpha_3,\alpha_2,\alpha_1;U_b\big)}
{\partial(\alpha_5,\alpha_4,\alpha_3,\alpha_2,\alpha_1)}\Big)\right|_{\alpha_1=\alpha_2=\alpha_3=\alpha_4=\alpha_5=h=0}\\
&&=\det\big(\rr_5,\rr_4, \rr_3,\rr_2,\rr_1\big)(\tilde{U}_b)
\neq 0.
\end{eqnarray*}
Then, by the implicit function theorem, there exists $(\alpha_1,\alpha_2,\alpha_3, \alpha_4, \alpha_5)$
as a $C^{2}$--function of $(\tilde{\alpha}_{i}, e^{-Lkh}h)$ satisfying \eqref{eq:7.1}.

\par We now estimate $\alpha_{j}, 1\leq j\leq5$. For $j=i$, we have
 \begin{eqnarray*}
\alpha_{i}=\alpha_{i}(\tilde{\alpha},0;\tilde{U}_b)+\alpha_{i}(0,\|\bar{Z}\|_{\infty}e^{-Lkh}h;\tilde{U}_b)-\alpha_{i}(0,0;\tilde{U}_b)
+O(1)|\tilde{\alpha}| \|\bar{Z}\|_{\infty}e^{-Lkh}h.
\end{eqnarray*}
By \eqref{eq:7.1}, we have
\begin{eqnarray*}
\alpha_{i}(\tilde{\alpha},0;\tilde{U}_b)=\tilde{\alpha},\qquad
\alpha_{i}(0,\|\bar{Z}\|_{\infty}e^{-Lkh}h;\tilde{U}_b)=\alpha_{i}(0,0;\tilde{U}_b)=0.
\end{eqnarray*}
Notice that
\begin{eqnarray*}
\begin{split}
\alpha_{j}(\tilde{\alpha},0;\tilde{U}_b)=\alpha_{j}(0,\|\bar{Z}\|_{\infty}e^{-Lkh}h;\tilde{U}_b)=\alpha_{j}(0,0;\tilde{U}_b)=0.
\end{split}
\end{eqnarray*}
Then, in a similar way, we obtain the estimates for $\alpha_{j}, 1\leq j\neq i\leq5$.
This completes the proof.
\end{proof}

\par $\mathbf{Case}\ 2.$ {\it The change of $5$--strong rarefaction wave-fronts  after reaction} (see Fig. 7.2).
Let $\tilde{s}$ be the $5$-strong incoming rarefaction fronts with order $k=1$
before reaction, which connects $\tilde{U}_b$ to $\tilde{U}_a$ with
\begin{eqnarray*}
\tilde{U}_a=U_{\rm Ra}(\tilde{s};\tilde{U}_b).
\end{eqnarray*}
Let $s$ be the $5$-strong rarefaction wave-front after reaction
and related by
\begin{eqnarray*}
U_a=U_{\rm Ra}(s;\Phi(0,\alpha_{4},\alpha_{3},\alpha_{2},\alpha_{1}; U_b)).
\end{eqnarray*}
Then, by the construction of the approximate solutions and Lemma \ref{lem:4.5}, we have
\begin{eqnarray} \label{eq:7.3}
U_{\rm Ra}(s;\Phi(0,\alpha_{2,3, 4},\alpha_{1}; \tilde{U}_{l}+O(1)\|\bar{Z}\|_{\infty}e^{-Lkh}h))=U_{\rm Ra}(\tilde{s};\tilde{U}_b)+O(1)\|\bar{Z}\|_{\infty}e^{-Lkh}h.
\end{eqnarray}
\begin{center}\label{fig9}
\begin{tikzpicture}[scale=1.4]
\draw [ultra thick](-1,1.7)--(-1,-1.5);
\draw [ thick](-2.8,-0.8)--(-1, 0);
\draw [ thick](-1,0)--(0.5,1.5);
\draw [ thin](-1,0)--(1,0.5);
\draw [ thin](-1,0)--(1,-0.5);
\node at (-1.8, -0.8) {$\tilde{U}_b$};
\node at (-2.1, 0) {$\tilde{U}_a$};
\node at (-0.2, 1.3) {$U_a$};
\node at (0.2, -0.8) {$U_b$};
\node at (-3, -0.8) {$\tilde{s}$};
\node at (0.7, 1.7) {$s$};
\node at (1.5, 0.5) {$\alpha_{2(3,4)}$};
\node at (1.3, -0.5) {$\alpha_{1}$};
\node at (-1.0, -1.7){$x=kh$};
\node [below] at (-1.0, -2.0)
{Fig. 7.2};
\end{tikzpicture}
\end{center}

Similar to Lemma \ref{lem:7.1}, we obtain

\begin{lemma}\label{lem:7.2}
Let $\{\tilde{U}_b, \tilde{U}_a\}$ and $\{U_b, U_a\}$
be the constant states before and after reaction. Then
\begin{eqnarray} \label{eq:7.4}
\begin{split}
\alpha_{i}&=
O(1)|\tilde{s}|\|\bar{Z}\|_{\infty}e^{-Lkh}h,\,\, i\neq 5,\qquad
s=\tilde{s}+O(1)|\tilde{s}|\|\bar{Z}\|_{\infty}e^{-Lkh}h.
\end{split}
\end{eqnarray}
\end{lemma}

\par $\mathbf{Case}\ 3.$ {\it The change of non-physical wave-fronts $\epsilon$ after reaction} (see Fig. 7.3).
Let $\tilde{\epsilon}$  be the incoming non-physical wave-front
with $\{\tilde{U}_b, \tilde{U}_a\}$ as its below and above states, respectively.
Let $\epsilon$ be the outgoing non-physical wave-fronts after reaction
and denote by  $\{U_b, U_a\}$ as its below and above states, respectively.
\vspace{3pt}
\begin{center}\label{fig10}
\begin{tikzpicture}[scale=1.3]
\draw [ultra thick] (-1,1.6)--(-1,-1.4);
\draw [dashed][thick](-2.5,-0.8)--(-1, 0);
\draw [dashed] [thick](-1,0)--(0.3,1.3);
\node at (-1.8, -0.8) {$\tilde{U}_b$};
\node at (-2.1, 0) {$\tilde{U}_a$};
\node at (-0.2, 1.3) {$U_a$};
\node at (0, 0.4) {$U_b$};
\node at (-2.7, -0.8) {$\tilde{\epsilon}$};
\node at (0.4, 1.4) {$\epsilon$};
\node at (-1.0, -1.7){$x=kh$};
\node [below] at (-1.0, -2.0)
{Fig. 7.3};
\end{tikzpicture}
\end{center}
Then we have the following lemma whose proof is similar to Lemma \ref{lem:7.1}.
\begin{lemma}\label{lem:7.3}
For $\tilde{\epsilon}$ and  $\epsilon$ defined above, we have
\begin{eqnarray} \label{eq:7.5}
\epsilon=\tilde{\epsilon}+O(1)|\tilde{\epsilon}|\|\bar{Z}\|_{\infty}e^{-Lkh}h.
\end{eqnarray}
\end{lemma}

\par$\mathbf{Case}\ 4.$ {\it New waves generated by the corner points after reaction} (see Fig. 7.4).
We consider the flow past a corner point $A_{k}=(kh,g_{h}(kh))$ after reaction for some $k\in \mathbf{N}_{+}$.
Denote by $\tilde{U}_a$ and $U_a$ the states before and after reaction.
Let $\varepsilon$ be $5$-waves generated from $A_{k}$ after reaction.
\vspace{5pt}
\begin{center}\label{fig11}
\begin{tikzpicture}[scale=1.8]
\draw [ultra thick] (-2.4,0)--(-0.6,0)-- (1.2,-0.4);
\draw [ultra  thick](-0.6,-0.8) --(-0.6,1.4);
\draw [thin](-0.6,0) --(0.6,0.8);
\draw [thin](-0.6,0) --(0.8,0.6);
\draw [thick][->](-1.8,0.6) --(-0.8,0.6);
\draw [thick][->](-0.4,0.6) --(0.1,0.6);
\draw [thick][->](0.4,0.2) --(1.0,-0.05);
\node at (-1.2, 0.9) {$\tilde{U}_a$};
\node at (-0.2, 0.8) {$U_a$};
\node at (-0.6, -1){$x=kh$};
\node [below] at (-0.8, -1.2)
{Fig. 7.4};
\end{tikzpicture}
\end{center}
Then we have
\begin{eqnarray} \label{eq:7.6}
U_a=\tilde{U}_a+O(1)\|\bar{Z}\|_{\infty}e^{-Lkh}h,\quad
\Phi_{5}(\varepsilon;U_{a})\cdot (\mathbf{n}_{k}, 0, 0,0)=0.
\end{eqnarray}
Using Lemmas \ref{lem:4.2}--\ref{lem:4.5} and by direct computation, we have

\begin{lemma} \label{lem:7.4}
Equation \eqref{eq:7.6} admits a unique solution $\varepsilon$ such that
\begin{equation} \label{eq:7.7}
\varepsilon=O(1)\omega_{k}+O(1)\|\bar{Z}\|_{\infty}e^{-Lkh}h.
\end{equation}
\end{lemma}

\subsection{Global estimates on the reacting step}
In this subsection, we study the uniform bound
on the total variation of the approximate solutions for the nonhomogeneous system \eqref{eq:1.1}.
To do this, we first establish the relation of two functionals before and after reaction.
To begin with, denote
$$
\omega^{+}_{k}={\rm max}\{\omega_{k}, 0\}  \,\,\qquad \mbox{for $\omega_{k}=\omega_{k}(A_{k})$}.
$$
Then we have
\begin{lemma}\label{lem:7.5}
For sufficiently small  $h$  and sufficiently large $ K_{\rm c}$ and $K$,
there exists a constant $\hat{K}>0$ depending only on
the system such that
\begin{eqnarray} \label{eq:7.8}
\begin{split}
&F_0(U^{\nu, h}; kh+)-F_0(U^{\nu,h};kh-)\\[2pt]
&\leq O(1)\|\bar{Z}\|_{\infty}e^{-Lkh}h
\Big(L_{\rm w}(U^{\nu,h};kh-)+1\Big)\Big(L_{\rm w}(U^{\nu,h};kh-)+3\Big)-\hat{K}|\omega^{+}_{k}|.
\end{split}
\end{eqnarray}
\end{lemma}

\begin{proof}
The proof is based on Lemmas \ref{lem:7.1}--\ref{lem:7.4}.  First we have
\begin{eqnarray*}
&&L_{\rm w}(U^{\nu, h}; kh+)-L_{\rm w}(U^{\nu, h};kh-)\\[1mm]
&&=\sum_{1\leq i\leq5}L^{i}_{\rm w}(U^{\nu, h};kh+)+L_{\rm np}(U^{\nu, h};kh+)
-\sum_{1\leq i\leq5}L^{i}_{\rm w}(U^{\nu, h}; kh-)-L_{\rm np}(U^{\nu, h};kh-)\\
&&=O(1)\|\bar{Z}\|_{\infty}e^{-Lkh}h\sum_{1\leq i\leq5}\sum_{\alpha_{i}}|\alpha_{i}|
+O(1)\|\bar{Z}\|_{\infty}e^{-Lkh}h\sum_{\epsilon}|\epsilon|+O(1)\|\bar{Z}\|_{\infty}e^{-Lkh}h\\
&&\leq O(1)\|\bar{Z}\|_{\infty}e^{-Lkh}h\Big(L_{\rm w}(U^{\nu, h}; kh-)+1\Big)+O(1)|\omega^{+}_{k}|,
\end{eqnarray*}
and
\begin{eqnarray*}
&&Q_{0}(U^{\nu, h};kh+)-Q_{0}(U^{\nu, h};kh+)\\
&&=\sum_{(\alpha_{i},\beta_{j})\in\mathcal{A}}|\alpha_{i}||\beta_{j}|-
\sum_{(\tilde{\alpha}_{i},\tilde{\beta}_{j})\in\mathcal{A}}|\tilde{\alpha}_{i}||\tilde{\beta}_{j}|\\
&&=\sum_{i,j}\Big(|\alpha_{i}|(|\beta_{j}|-|\tilde{\beta_{j}}|)+|\tilde{\beta}_{j}|(|\alpha_{i}|-|\tilde{\alpha}_{i}|)\Big)\\
&&\leq\Big(L_{\rm w}(U^{\nu, h};kh+)-L_{\rm w}(U^{\nu, h};kh-)\Big)L_{\rm w}(U^{\nu, h};kh-)\\[5pt]
&&\leq O(1)\|\bar{Z}\|_{\infty}e^{-Lkh}h \Big(L_{\rm w}(U^{\nu, h};kh-)+1\Big)L_{\rm w}(U^{\nu, h};kh-)\\
&&\leq O(1)\|\bar{Z}\|_{\infty}e^{-Lkh}h \Big(L_{\rm w}(U^{\nu, h};kh-)+1\Big)^2+O(1)|\omega^{+}_{k}|.
\end{eqnarray*}

For $1\leq i\leq4$, we have
\begin{eqnarray*}
\begin{split}
&Q_{Bi}(U^{\nu, h};kh+)-Q_{Bi}(U^{\nu, h};kh-)\\[2pt]
&=\sum_{\alpha_{i}}|\alpha_{i}|W(\alpha_{i},kh+,-)-
\sum_{\tilde{\alpha}_{i}}|\tilde{\alpha}_{i}|W(\tilde{\alpha}_{i},kh-,-)\\
&=\sum_{\tilde{\alpha}_{i}}|\tilde{\alpha}_{i}|W(\tilde{\alpha}_{i},kh-,-)
\Big(\big(1+O(1)\|\bar{Z}\|_{\infty}e^{-Lkh}h\big)\exp\big(O(1)\|\bar{Z}\|_{\infty}e^{-Lkh}h\big)-1\Big)\\
&\leq O(1)\|\bar{Z}\|_{\infty}e^{-Lkh}hL_{\rm w}(U^{\nu, h};kh-).
\end{split}
\end{eqnarray*}
Similarly,
\begin{eqnarray*}
&&Q_{B5}(U^{\nu, h};kh+)-Q_{B5}(U^{\nu, h};kh-)\leq O(1)\|\bar{Z}\|_{\infty}e^{-Lkh}hL_{\rm w}(U^{\nu, h};kh-)+O(1)|\omega^{+}_{k}|,\\[1mm]
&&Q_{BNP}(U^{\nu, h};kh+)-Q_{BNP}(U^{\nu, h};kh-)\\[1mm]
&&\quad=\sum_{\epsilon}|\epsilon|W(\epsilon,kh+,+)-
\sum_{\tilde{\epsilon}}|\tilde{\epsilon}|W(\tilde{\epsilon},kh-,+)\\
&&\quad=\sum_{\tilde{\epsilon}}|\tilde{\epsilon}|W(\tilde{\epsilon},kh-,+)
 \Big(\big(1+O(1)\|\bar{Z}\|_{\infty}e^{-Lkh}h\big)
  \exp\big(O(1)\|\bar{Z}\|_{\infty}e^{-Lkh}h\big)-1\Big)\\
&&\quad\leq O(1)\|\bar{Z}\|_{\infty}e^{-Lkh}hL_{\rm w}(U^{\nu, h};kh-),\\[1.5mm]
&& Q_{\rm c}(U^{\nu, h};kh+)-Q_{\rm c}(U^{\nu, h};kh-)=-|\omega^{+}_{k}|.
\end{eqnarray*}
Therefore, it follows that
\begin{eqnarray*}
&&Q(U^{\nu, h};kh+)-Q(U^{\nu, h}; kh-)\\
&&\,\,\leq O(1)\|\bar{Z}\|_{\infty}e^{-Lkh}h \Big(L_{\rm w}(U^{\nu, h};kh-)+1\Big)^{2} +O(1)\|\bar{Z}\|_{\infty}e^{-Lkh}hL_{\rm w}(U^{\nu, h};kh-)\\
&&\,\,\leq O(1)\|\bar{Z}\|_{\infty}e^{-Lkh}h \Big(L_{\rm w}(U^{\nu, h};kh-)+1\Big)\Big(L_{\rm w}(U^{\nu, h};kh-)+2\Big)
-\big(K_{\rm c}-O(1)\big)|\omega^{+}_{k}|,
\end{eqnarray*}
and
\begin{eqnarray*}
&&F_{0}(U^{\nu, h};kh+)-F_{0}(U^{\nu, h};kh-)\\
&&=L_{\rm w}(U^{\nu, h};kh+)-L_{\rm w}(U^{\nu, h};kh-)+K\Big(Q(U^{\nu, h};kh+)-Q(U^{\nu, h};kh-)\Big)\\
&&\leq O(1)\|\bar{Z}\|_{\infty}e^{-Lkh}h\Big(L_{\rm w}(U^{\nu, h};kh-)+1\Big)\\
&&\ \ \ +O(1)e^{-Lkh}h \Big(L_{\rm w}(U^{\nu, h};kh+)+1\Big)\Big(L_{\rm w}(U^{\nu, h};kh+)+2\Big)\\
&&\ \ \ -K\big(K_{\rm c}-O(1)\big)|\omega^{+}_{k}|+O(1)|\omega^{+}_{k}|\\
&&\leq O(1)\|\bar{Z}\|_{\infty}e^{-Lkh}h\Big(L_{\rm w}(U^{\nu, h};kh-)+1\Big)\Big(L_{\rm w}(U^{\nu, h};kh-)+3\Big)-\hat{K}|\omega^{+}_{k}|,
\end{eqnarray*}
by choosing $K$ and $K_{\rm c}$ sufficiently large.
 This completes the proof.
\end{proof}

To obtain the bound on the functional, we also need to estimate $F_{1}(U^{\nu, h};kh)$.

\begin{lemma}\label{lem:7.6}
For the approximate solutions $U^{\nu, h}$,
\begin{eqnarray}
&& F_{1}(U^{\nu, h};kh+)-F_{1}(U^{\nu, h};kh-)\nonumber \\[2mm]
&&\leq O(1)\|\bar{Z}\|_{\infty}e^{-Lkh}h\big(L_{\rm w}(U^{\nu, h};kh-)
+1\big)+O(1)|\omega^{+}_{k}|.  \label{eq:7.9}
\end{eqnarray}
\end{lemma}

\begin{proof}
We complete the proof by combining  Lemmas \ref{lem:7.1}--\ref{lem:7.4} together.
Here we only give a detail proof for $\mathbf{Case}\ 1$, since the other cases can be treated in the same way.

Without loss of generality, consider a $1$-wave-front $\tilde{\alpha}_1$
that connects $\{\tilde{U}_b, \tilde{U}_a\}$ and
intersects on line $x=kh$.
Suppose that the new Riemann problem is solved by waves $\alpha_i, 1\leq i\leq 5$,
with  $\{U_b, U_{m1}, U_{m2}, U_a\}$ as its below, middle, and above states, respectively (see Fig. 7.5).
\begin{center}\label{fig12}
\begin{tikzpicture}[scale=1.2]
\draw[ultra  thick] (-1,1.8)--(-1,-1.5);
\draw [thin](-2.8,-0.8)--(-1, 0);
\draw [thin](-1,0)--(0.5,1.5);
\draw [thin](-1,0)--(1,0.5);
\draw [thin](-1,0)--(1,-0.5);
\node at (-1.8, -0.8) {$\tilde{U}_b$};
\node at (-2.1, 0) {$\tilde{U}_a$};
\node at (0.2, -0.8) {$U_b$};
\node at (0.6, 0) {$U_{m1}$};
\node at (0.6, 0.8) {$U_{m2}$};
\node at (-0.2, 1.3) {$U_a$};
\node at (-3, -0.7) {$\tilde{\alpha}_{1}$};
\node at (0.7, 1.7) {$\alpha_{5}$};
\node at (1.5, 0.5) {$\alpha_{2(3,4)}$};
\node at (1.3, -0.5) {$\alpha_{1}$};
\node at (-1.0, -1.7){$x=kh$};
\node [below] at (-1.0, -2.0)
{Fig. 7.5};
\end{tikzpicture}
\end{center}
Denote
\begin{eqnarray*}
\theta(\Phi(\alpha; U))=f(\alpha_{5},\alpha_{2(3,4)},\alpha_{1};U).
\end{eqnarray*}
Then,  by Lemma \ref{lem:7.1}, we have
\begin{eqnarray*}
\begin{split}
&\Big |\theta(U_{m1})-\theta(U_b)\Big|-\Big |\theta(\tilde{U}_a)-\theta(\tilde{U}_b)\Big|\\
&\leq\Big |\theta(U_{m1})-\theta(U_b)-\theta(\tilde{U}_a)+\theta(\tilde{U}_b)\Big|\\
&=\Big |f(0,0,\alpha_{1};U_b)-f(0,0,0;U_b)
-f(0,0,\tilde{\alpha}_1;\tilde{U}_b)+f(0,0,0;\tilde{U}_b)\Big|\\
&\le \Big|f(0,0,\tilde{\alpha}_{1};\tilde{U}_b+O(1)\|\bar{Z}\|_{\infty}e^{-Lkh}h)
-f(0,0,0;\tilde{U}_b+O(1)\|\bar{Z}\|_{\infty}e^{-Lkh}h)\Big|\\
&\  \  \ +\Big|f(0,0,\tilde{\alpha}_1;\tilde{U}_b)-f(0,0,0;\tilde{U}_b)\Big|+O(1)\|\bar{Z}\|_{\infty}|\tilde{\alpha}_{1}|e^{-Lkh}h\\
&=O(1)\|\bar{Z}\|_{\infty}|\tilde{\alpha}_{1}|e^{-Lkh}h,
\end{split}
\end{eqnarray*}
and
\begin{eqnarray*}
\begin{split}
\theta(U_a)-\theta(U_{m2})&=f(\alpha_5,0,0;U_{m2})-f(0,0,0;U_{m2})\\[1mm]
&=f(O(1)\|\bar{Z}\|_{\infty}|\tilde{\alpha}_{1}|e^{-Lkh}h,0,0;U_{m2})-f(0,0,0;U_{m2})\\[1mm]
&=O(1)\|\bar{Z}\|_{\infty}|\tilde{\alpha}_{1}|e^{-Lkh}h.
\end{split}
\end{eqnarray*}
Since $\theta(U_{m2})=\theta(U_{m1})$, we have
\begin{eqnarray*}
&&TV.\{\theta(U^{\nu,h}(\tau+,\cdot));\,[g_{h}(\tau+),\infty)\}-TV.\{\theta(U^{\nu,h}(\tau-,\cdot));\,[g_{h}(\tau-),\infty)\}\\[3pt]
&&=\sum_{\tilde{\alpha}_{1}}\Big(\big|\theta(U_a)-\theta(U_{m2})\big|
+\big|\theta(U_{m2})-\theta(U_{m1})\big|+\big|\theta(U_{m1})-\theta(U_b)\big|
-\big|\theta(\tilde{U}_a)-\theta(\tilde{U}_b)\big|\Big)\\[1mm]
&&=O(1)\|\bar{Z}\|_{\infty}\sum_{\tilde{\alpha}_{1}}|\tilde{\alpha}_{1}|e^{-Lkh}h=O(1)\|\bar{Z}\|_{\infty}e^{-Lkh}hL_{\rm w}(U^{\nu, h};kh-),
\end{eqnarray*}
which leads to
\begin{eqnarray*}
\begin{split}
F_{1}(U^{\nu, h}; \tau+)-F_{1}(U^{\nu, h}; \tau-)=O(1)\|\bar{Z}\|_{\infty}e^{-Lkh}hL_{\rm w}(U^{\nu, h};kh-).
\end{split}
\end{eqnarray*}
Similarly, by Lemma \ref{lem:6.1} and Proposition \ref{prop:6.2},  we find that, for $\mathbf{Cases}$ $2$--$4$,
\begin{eqnarray*}
\begin{split}
F_{1}(U^{\nu, h}; \tau+)-F_{1}(U^{\nu, h}; \tau-)&=O(1)\|\bar{Z}\|_{\infty}\sum_{\tilde{s}}|\tilde{s}|e^{-Lkh}h=O(1)\|\bar{Z}\|_{\infty}e^{-Lkh}h,\\
F_{1}(U^{\nu, h}; \tau+)-F_{1}(U^{\nu, h}; \tau-)&=O(1)\|\bar{Z}\|_{\infty}\sum_{\tilde{\epsilon}}|\tilde{\epsilon}|e^{-Lkh}h\\
&=O(1)\|\bar{Z}\|_{\infty}e^{-Lkh}hL_{\rm w}(U^{\nu, h};kh-),\\[2mm]
F_{1}(U^{\nu, h}; \tau+)-F_{1}(U^{\nu, h}; \tau-)&=O(1)\omega^{+}_{k}+O(1)\|\bar{Z}\|_{\infty}e^{-Lkh}h.
\end{split}
\end{eqnarray*}
Finally, combining the above cases together, we conclude \eqref{eq:7.9}.
\end{proof}

With Lemma \ref{lem:7.5}--\ref{lem:7.6},
we are now able to estimate on the change of functional $F(U^{\nu,h};x)$ at $x=kh$.

\begin{lemma}\label{lem:7.7}
For  sufficiently small $h$, there exist
$C_4$ independent of $(k, h)$
and $C_{*}$ sufficiently large
such that, on line $x=kh$,
\begin{eqnarray} \label{eq:7.10}
F(U^{\nu, h}; kh+)&\leq F(U^{\nu,h};kh-)+C_{4}\|\bar{Z}\|_{\infty}e^{-Lkh}h\Big(F(U^{\nu,h};kh-)+4\Big)^2.
\end{eqnarray}
\end{lemma}

\begin{proof}
By Lemmas \ref{lem:7.5}--\ref{lem:7.6}, we have
\begin{eqnarray*}
\begin{split}
&F(U^{\nu, h};kh+)-F(U^{\nu, h};kh-)\\
&=F_{0}(U^{\nu, h};kh+)-F_{0}(U^{\nu, h};kh-)+C_{*}\Big(F_{1}(U^{\nu, h};kh+)-F_{1}(U^{\nu, h};kh-)\Big)\\
&\leq O(1)(1+C_{*})\|\bar{Z}\|_{\infty}e^{-Lkh}h\Big(L_{\rm w}(U^{\nu, h};kh-)+1\Big)+\big(O(1)-C_{*}\hat{K}\big)|\omega^{+}_{k}|\\
&\ \ \ +O(1)\|\bar{Z}\|_{\infty}e^{-Lkh}h\Big(L_{\rm w}(U^{\nu, h};kh-)+1\Big)\Big(L_{\rm w}(U^{\nu, h};kh-)+3\Big).
\end{split}
\end{eqnarray*}
Then we choose $C_{*}$ and $\hat{K}$ sufficiently large so that there exists $C_4$ such that
\begin{eqnarray*}
\begin{split}
 F(U^{\nu, h};kh+)-F(U^{\nu, h};kh-)&\leq C_4\|\bar{Z}\|_{\infty}e^{-Lkh}h\Big(F(U^{\nu, h};kh-)+4\Big)^{2}.
\end{split}
\end{eqnarray*}
\end{proof}

To obtain the uniform bound on the total variation of the approximate solutions $U^{\nu, h}$,
we introduce the following functional:
\begin{eqnarray}\label{eq:7.11}
 \mathcal{F}(U^{\nu, h};\tau)=F(U^{\nu, h};\tau)+\mathcal{K}\sum_{kh>\tau}\|\bar{Z}\|_{\infty}e^{-Lkh}h,
\end{eqnarray}
where $\mathcal{K}>0$ is a unknown constant to be determined later.

\begin{proposition}\label{prop:7.1}
Suppose that $U^{\nu, h}(x,y)$ is the approximate solution for the initial-boundary value problem \eqref{eq:1.1}
and \eqref{eq:1.8}--\eqref{eq:1.9}.
Then, for $h$ and $\delta_2$  sufficiently small,
there exists a positive constant $\mathcal{K}$ independent of $(h, \nu, \delta_2)$ such that, if
\begin{eqnarray}  \label{eq:7.12}
\begin{split}
\mathcal{F}(U^{\nu, h};\tau)\leq \delta_2,
\end{split}
\end{eqnarray}
then
\begin{eqnarray} \label{eq:7.13}
\mathcal{F}(U^{h};\tau_2)\leq \mathcal{F}(U^{h};\tau_1)\qquad \mbox{for any $\tau_2\geq\tau_1$}.
\end{eqnarray}
Moreover, on $x=kh+$,
\begin{enumerate}
\item[\rm (i)] $U^{\nu, h}(x,y)\in D_{\delta_{0}}(U_{\infty})${\rm ;}

\smallskip
\item[\rm (ii)] There exist constants $C_{6}, C_{7}>0$ independent of $(\nu,h)$ such that
 \begin{eqnarray}\label{eq:7.14}
\big|s\big|\leq\frac{C_{6}}{\nu},\qquad \sum_{s}\big|s\big|\leq C_{7},
\end{eqnarray}
for any strong rarefaction front $s(x)${\rm ;}

\smallskip
\item[\rm (iii)] There exists a constant $C_{9}>0$ depending only on the system such that
\begin{eqnarray}\label{eq:7.15}
\sum_{\epsilon}\big|\epsilon\big|\leq \frac{C_{9}}{2^{\nu}}
\end{eqnarray}
for any non-physical wave $\epsilon$ after reaction.
\end{enumerate}
\end{proposition}

\begin{proof}
We divide the proof into two cases. Suppose that $\mathcal{F}(U^{h};\tau)<\delta_2$ for sufficiently small $\delta_2$.

\smallskip
{\bf 1.} If $(k-1)h\leq\tau_1\leq\tau_2<kh$, then, by Proposition \ref{prop:6.1}, we have
\begin{eqnarray*}
\mathcal{F}(U^{\nu,h};\tau_2)- \mathcal{F}(U^{h};\tau_1)
=F(U^{\nu, h};\tau_2)-F(U^{\nu, h};\tau_1)-\mathcal{K}(\tau_2-\tau_1)\|\bar{Z}\|_{\infty}e^{-Lkh}h
\leq 0.
\end{eqnarray*}

{\bf 2.} If $(k-1)h\leq\tau_1< kh<\tau_2<(k+1)h$, then, by Proposition \ref{prop:6.1} and Lemma \ref{lem:7.7},
we can choose $\mathcal{K}$ suitably large such that
\begin{eqnarray*}
\begin{split}
\mathcal{F}(U^{\nu,h};\tau_2)- \mathcal{F}(U^{h};\tau_1)&=F(U^{\nu, h};\tau_2)-F(U^{\nu, h};\tau_1)-\mathcal{K}(\tau_2-\tau_1)\|\bar{Z}\|_{\infty}e^{-Lkh}h\\[1mm]
&\leq F(U^{\nu, h};kh+)-F(U^{\nu, h};kh-)-\mathcal{K}(\tau_2-\tau_1)\|\bar{Z}\|_{\infty}e^{-Lkh}h\\
&\leq \Big(C_4\big(\delta_2+4\big)^{2}-\mathcal{K}(\tau_2-\tau_1)\Big)\|\bar{Z}\|_{\infty}e^{-Lkh}h\leq 0.
\end{split}
\end{eqnarray*}
Since $\bar{Z}(\infty)=0$, then $U^{\nu,h}(kh+,\infty)=U^{\nu,h}(kh-,\infty)=\bar{U}(\infty)$.
On the other hand, by $\mathbf{(H_{2})}$, we can choose suitably $\delta_2$ such that
\begin{eqnarray*}
\|\bar{U}-U_{\infty}\|_{\infty}<\delta_2.
\end{eqnarray*}
Therefore, by Lemma \ref{lem:6.2} and Proposition \ref{prop:7.1},
for $\delta_2$ sufficiently small, we have
\begin{eqnarray*}
\begin{split}
&\Big|\Big(J\big(q,\frac{q^{2}}{2}+\frac{c^{2}}{\gamma-1}\big)+\theta\Big)(U^{\nu,h}(kh+,y))
-J\big(q_{\infty}, \frac{q^{2}_{\infty}}{2}+\frac{c^{2}_{\infty}}{\gamma-1}\big)\Big|\\
& \leq\Big|\Big(J\big(q,\frac{q^{2}}{2}+\frac{c^{2}}{\gamma-1}\big)+\theta\Big)(U^{\nu,h}(kh+,y))
-\Big(J\big(q,\frac{q^{2}}{2}+\frac{c^{2}}{\gamma-1}\big)+\theta\Big)(U^{\nu,h}(kh+,\infty))\Big|\\
&\ \ \ +\Big|\Big(J\big(q,\frac{q^{2}}{2}+\frac{c^{2}}{\gamma-1}\big)+\theta\Big)(U^{\nu,h}(kh+,\infty))-J\big(q_{\infty}, \frac{q^{2}_{\infty}}{2}+\frac{c^{2}_{\infty}}{\gamma-1}\big)\Big|\\
& \leq O(1)\mathcal{F}(U^{\nu,h};kh+)+\Big|\Big(J\big(q,\frac{q^{2}}{2}+\frac{c^{2}}{\gamma-1}\big)+\theta\Big)(\bar{U}(\infty))-J\big(q_{\infty}, \frac{q^{2}_{\infty}}{2}+\frac{c^{2}_{\infty}}{\gamma-1}\big)\Big|\leq \delta_0.
\end{split}
\end{eqnarray*}
Similarly,
\begin{eqnarray*}
\begin{split}
&\Big|\Big(\frac{q^{2}}{2}+\frac{c^{2}}{\gamma-1}\Big)(U^{\nu,h}(kh+,y))
   -\frac{q^{2}_{\infty}}{2}-\frac{c^{2}_{\infty}}{\gamma-1}\Big|\leq \delta_0,\\
&\Big|\frac{p^{\nu,h}}{{(\rho^{\nu,h})}^{\gamma}}(kh+,y)
-\frac{p_{\infty}}{\rho^{\gamma}_{\infty}}\Big|\leq \delta_0,
\qquad 0\leq Z^{\nu,h}(kh+,y)\leq \delta_0.
\end{split}
\end{eqnarray*}

Notice that
\begin{eqnarray*}
\theta(U^{\nu,h}(kh+;y))=\big(\theta(U^{\nu,h}(kh+;y))-\theta(U^{\nu,h}(kh+,\infty))\big)
+\big(\theta(U^{\nu,h}(kh+,\infty))-\theta(U_{\infty})\big).
\end{eqnarray*}
Then, by Lemma \ref{lem:6.1} and Proposition \ref{prop:7.1}, for sufficiently small $\delta_2$, we have
\begin{eqnarray*}
\theta(U^{\nu,h}(kh+;y))\leq O(1)\mathcal{F}(U^{\nu,h};kh+)+O(1)\delta_2\leq \delta_0.
\end{eqnarray*}
On the other hand, we have
\begin{eqnarray*}
\begin{split}
\theta(U^{\nu,h}(kh+;y))
&=\big(\theta(U^{\nu,h}(kh+;y))-\theta(U^{\nu,h}(kh+,\infty))\big)
+\big(\theta(U^{\nu,h}(kh+,\infty))-\theta(U_{\infty})\big)\\[1mm]
&\geq - TV.\{\theta(U^{\nu,h}(kh+,,y))\}+\theta_{\rm crit}+O(1)\delta_{2}\\[1mm]
&\geq -F(U^{h};kh+)+\theta_{\rm crit}+O(1)\delta_{2}\geq \theta_{\rm crit}+\delta_{0}.
\end{split}
\end{eqnarray*}
Finally, by Lemma \ref{lem:7.2}, there exists a constant $C_{5}>0$ such that
\begin{eqnarray*}
\begin{split}
\big|s\big|&\leq\big(1+O(1)\|\bar{Z}\|_{\infty}e^{-Lkh}h\big)\big|\tilde{s}\big|
\leq\big(1+O(1)\|\bar{Z}\|_{\infty}e^{-Lkh}h\big)\frac{C_{3}}{\nu}\\
&\leq \frac{C_{3}\exp(C_5/L)}{\nu}.
\end{split}
\end{eqnarray*}
Then we can find a constant $C_{6}>C_{3}\exp(C_5/L)$ such that
\begin{eqnarray*}
\begin{split}
|s|&\leq\frac{C_{6}}{\nu}.
\end{split}
\end{eqnarray*}
Finally, by Proposition \ref{prop:6.2} and Lemma \ref{lem:7.2},
we can find $C_{7}>C_{4}\exp(C_5/L)$ such that
 \begin{eqnarray*}
\sum |s|&\leq \big(1+O(1)\|\bar{Z}\|_{\infty}e^{-Lkh}h\big)\sum |\tilde{s}|
\leq \exp\Big(C_5\int_{0}^{\infty}e^{-x}dx\Big)\sum |\tilde{s}|\leq C_{7}.
\end{eqnarray*}
Also, by Proposition \ref{prop:6.1} and Lemma \ref{lem:7.3},
we can find positive constants $C_{8}$ and $C_{9}$
such that
\begin{eqnarray*}
&&\sum |\epsilon|\leq \big(1+O(1)\|\bar{Z}\|_{\infty}e^{-Lkh}h\big)\sum |\tilde{\epsilon}|
   \leq \big(1+C_{8}\|\bar{Z}\|_{\infty}e^{-Lkh}h\big)\sum |\tilde{\epsilon}|\\
&&\leq \exp\Big(C_{\rm np}\int_{0}^{\infty}e^{-x}dx\Big)\sum |\tilde{\epsilon}|
\leq \frac{C_{9}}{2^{\nu}}.
\end{eqnarray*}
This completes the proof.
\end{proof}

\begin{corollary}\label{coro:7.1}
There exist positive constants $\varepsilon$ and $C_{0}$ independent of $(\nu,h)$ such that,
if
\begin{eqnarray}\label{eq:7.16}
\|\bar{U}-U_{\infty}\|_{\infty}+\big|g'_{+}(0)-g'_{*+}(0)\big|
+TV.\{\bar{U}(\cdot)\}+TV. \{(g'_{+}-g'_{*+})(\cdot)\}<\varepsilon,
\end{eqnarray}
then
\begin{eqnarray}\label{eq:7.17}
TV.\{U^{\nu,h}(x,\cdot)\}<C_{0}.
\end{eqnarray}
\end{corollary}

\begin{proof}
By Proposition \ref{prop:7.1}, if $\mathcal{F}(U^{\nu, h};0+)$ sufficiently small, we deduce
\begin{eqnarray*}
\mathcal{F}(U^{\nu, h};x)<\mathcal{F}(U^{\nu, h};0+).
\end{eqnarray*}
Since
\begin{eqnarray*}
\mathcal{F}(U^{\nu, h};0+)&=&F(U^{\nu, h};0+)+ C_4\|\bar{Z}\|_{\infty}\sum^{l=k}_{l=0}e^{-Lkh}h\\
&\leq& F(U^{\nu, h};0+)+ C_4\|\bar{Z}\|_{\infty}\int^{\infty}_0 e^{-\tau}d\tau\\
&\leq& F(U^{\nu, h};0+)+ \frac{C_4}{L}\|\bar{Z}\|_{\infty}\\
&\leq& O(1)F(U^{\nu, h};0+),
\end{eqnarray*}
where we have used that $\bar{Z}(\infty)=0$.
On the other hand,
\begin{eqnarray*}
&&F(U^{\nu,h};0+)=O(1)\Big(\|\bar{U}-U_{\infty}\|_{\infty}
+TV.\{\bar{U}(\cdot)\}+\sum_{k\geq 0}\big|\omega^{+}_{k}\big|\Big),\\
&&\sum_{k\geq 0}\big|\omega^{+}_{k}\big|=O(1)\Big(TV.\{(g'_{+}-g'_{*+})(\cdot)\}+\big|g'_{+}(0)-g'_{*+}(0)\big|\Big).
\end{eqnarray*}
Thus, combining the above estimates, we obtain the desire result.
\end{proof}

\section{Convergence and Consistency}\setcounter{equation}{0}

In this section, we show that there exists a subsequence of the approximate solutions
which converges to an entropy solution
to the initial-boundary value problem \eqref{eq:1.8}--\eqref{eq:1.9} for system \eqref{eq:1.1}.
To see this, we proceed by two steps:

\smallskip
\par {\rm (a)} We establish the consistency of our algorithm by keeping $h>0$ fixed and
letting $\nu\to \infty$.
\par {\rm (b)} We employ the compactness theorem to show the convergence of a subsequence
of the approximate solutions as $h\rightarrow 0$.

\smallskip
\par First, if line $x=a>0$ intersects with $\Gamma_{h}$ at point $(a, g_h(a))$, we have

\begin{lemma}\label{lem:8.1}
For any $x',\ x''\in((k-1)h,kh), \ k\in\mathbb{N}_{+}$, there exists a constant $M_{0}>0$ independent of $(\nu, h)$ such that
\begin{eqnarray} \label{eq:8.1}
\int_{0}^{\infty}|U^{\nu,h}(x', y+g_h(x'))-U^{\nu, h}(x'', y+g_h(x''))|dy\leq M_{0}|x'-x''|.
\end{eqnarray}
\end{lemma}
The proof directly follows from the construction of the approximate solutions $U^{\nu,h}$
and the bound on its total variation, which leads to the uniform
Lipschitz continuity.

\begin{proposition}\label{prop:8.1}
Let $U^{\nu,h}$ be a family of the approximate solutions of problem \eqref{eq:1.8}--\eqref{eq:1.9} for system \eqref{eq:1.1}.
Then there exists a subsequence $\{\nu_{i}\}^{\infty}_{i=1}$ such that $U^{\nu_{i},h}$ converges
to a function $U^{h}$ in $\mathrm{L^{1}_{loc}}(\Omega_{h})$ as $i\rightarrow \infty$, and
the limit function $U^h(x,y)$ satisfies that,
for any $\psi \in C_0^{\infty}(\Omega_h)$ with $\psi\geq 0$ and for any convex entropy pair $(\eta,\ q)$,
\begin{eqnarray} \label{eq:8.15}
\begin{split}
\iint_{\Omega_{h}}
\Big(\eta(W(U^{h}))\psi_x+q(W(U^{h}))\psi_y\Big) dxdy+\int_{y>0}\eta(W(\bar{U}(y)))\psi(0,y) dy\\[5pt]
+\sum_{k\geq1}h\int_{y>g_{k}}\nabla_{W} \eta(W(U^{h}(kh-,y)))G(U^{h}(kh-,y))\psi(kh,y)dy\\[5pt]
+\int_{y>0}\eta(W(\bar{U}(y)))\psi(0,y) dy
+\sum_{k\geq1}h\int_{y>g_{k}}\mathcal{O}^\eta(h,kh-,y)\psi(kh,y)dy
\geq0,
\end{split}
\end{eqnarray}
where $\mathcal{O}^\eta(h,kh-,y)\rightarrow 0$ as $h\rightarrow 0$,
and $\mathcal{O}^{\eta}(h,kh-,y)\equiv 0$ when $\eta=\pm W$.
\end{proposition}

\begin{proof}
We first denote
\begin{eqnarray}\label{eq:8.16}
\begin{split}
\mathcal{N}^{\nu, h}(\psi)&:=\iint_{\Omega_{h}}\Big(\eta(W(U^{\nu,h}))\psi_x
+q(W(U^{\nu,h}))\psi_y\Big)dxdy+\int_{y>0}\eta(W(\bar{U}^{\nu}(y)))\psi(0,y)dy\\[5pt]
&\ \ \  +\sum_{k\geq1}h\int_{y>g_{k}}\nabla_{W}\eta (W(U^{\nu,h}(kh-,y)))G(U^{\nu,h}(kh-,y))\psi(kh,y)dy\\[5pt]
&\ \ \  +\sum_{k\geq1}h\int_{y>g_{k}}\mathcal{O}_{\nu}^\eta(h,kh,y)\psi(kh,y)dy\\[5pt]
&=:\mathcal{N}_{1}^{\nu}(\psi)+\mathcal{N}_{2}^{\nu}(\psi),
\end{split}
\end{eqnarray}
where
\begin{eqnarray*}
\begin{split}
\mathcal{N}_{1}^{\nu}(\psi)&=\iint_{\Omega_{h}}\Big(\eta(W(U^{\nu,h}))\psi_x
+q(W(U^{\nu,h}))\psi_y\Big)dxdy+\int_{y>0}\eta(W(\bar{U}^{\nu}(y)))\psi(0,y)dy,\\[5pt]
\mathcal{N}_{2}^{\nu}(\psi)&=\sum_{k\geq1}h\int_{y>g_{k}}\nabla_{W}\eta (W(U^{\nu,h}(kh-,y)))G(U^{\nu,h}(kh-,y))\psi(kh,y)dy\\[5pt]
 &\ \ \ \ +\sum_{k\geq1}h\int_{y>g_{k}}\mathcal{O}_{\nu}^\eta(h,kh,y)\psi(kh,y)dy.
\end{split}
\end{eqnarray*}
Let
$\Xi(x,y)=\Big(\eta(W(U^{\nu,h}))\psi, q(W(U^{\nu,h}))\psi\Big)(x,y)$.
Then, by the divergence theorem, we have
\begin{eqnarray*}
\begin{split}
\mathcal{N}_{1}^{\nu,h}(\phi)&=\sum_{k}\iint_{\Omega_{h,k}}\Big(\eta(W(U^{\nu,h}))\psi_x
+q(W(U^{\nu,h}))\psi_y\Big)dxdy+\int_{y>0}\eta(W(\bar{U}^{\nu}(y)))\psi(0,y)dy\\[5pt]
&=\sum_{k, i}\int_{\Gamma^{i}_{h,k}}div\,\Xi(x,y)\cdot\textbf{n}_{k, i} dS
-\sum_{k}\iint_{\Omega_{h,k}}\Big(\eta(W(U^{\nu,h}))_x
+q(W(U^{\nu,h}))_y\Big)\psi dxdy\\[5pt]
&\ \ \  +\int_{y>0}\eta (W(\bar{U}(y)))\psi(0,y)dy.
\end{split}
\end{eqnarray*}
Since $U^{\nu,h}$ is an entropy solution for the homogeneous system $W(U)_{x}+H(U)_{y}=0$ in $\Omega_{h,k}$,
then
\begin{eqnarray*}
\begin{split}
\mathcal{N}_{1}^{\nu,h}(\phi)
&\geq\sum_{k}\int^{(k+1)h}_{kh}\sum_{\alpha\in\mathcal{J}}\mathcal{E}_{\nu, \alpha}(x)\,\psi(x,y_{\alpha}(x))dx\\[5pt]
&\ \ \ -\sum_{k\geq 1}\int_{y>g_{k}}\Big(\eta(W(U^{\nu,h}(kh+,y))-\eta(W(U^{\nu,h}(kh-,y)))\Big)\psi(kh,y)dy,
\end{split}
\end{eqnarray*}
where
\begin{eqnarray*}
\mathcal{E}_{\nu, \alpha}(x)=\dot{y}_{\alpha}(x)[\eta(W(U^{\nu,h}(x,y_{\alpha}(x))))]-[q(W(U^{\nu,h}(x,y_{\alpha}(x))))]
\end{eqnarray*}
with
$[f(U^{\nu,h}(x,y_{\alpha}(x)))]:= f(U^{\nu,h}(x,y_{\alpha}(x)+))-f(U^{\nu,h}(x,y_{\alpha}(x)-))$
for any function $f$.

Using the properties of the $(\nu, h)$-approximate solutions, then we have
\begin{eqnarray*}
\sum_{\alpha \in \mathcal{S}\cup\mathcal{R}\cup\mathcal{R}_{b}}\mathcal{E}_{\nu, \alpha}(x)
\geq -\sum_{\alpha \in \mathcal{S}\cup\mathcal{R}\cup\mathcal{R}_{b}}|\sigma_{\alpha}|\frac{O(1)}{2^{\nu}},\quad
\sum_{\alpha \in\mathcal{NP}}\mathcal{E}_{\nu, \alpha}(x)=-O(1)\sum_{\alpha \in\mathcal{NP}}|\sigma_{\alpha}|=\frac{O(1)}{2^{\nu}}.
\end{eqnarray*}
Thus it follows that
\begin{eqnarray*}
\sum_{\alpha \in \mathcal{J}}\mathcal{E}_{\nu, \alpha}(x)
=\sum_{\alpha \in \mathcal{S}\cup\mathcal{R}\cup\mathcal{R}_{b}}\mathcal{E}_{\nu, \alpha}(x)
+\sum_{\alpha \in\mathcal{NP}}\mathcal{E}_{\nu, \alpha}(x)
\geq-\frac{O(1)}{2^{\nu}}.
\end{eqnarray*}
On the other hand, by the Talyor formula, we also find
\begin{eqnarray*}
\begin{split}
&\eta(W(U^{\nu,h}(kh+,y)))-\eta(W(U^{\nu,h}(kh-,y)))\\[5pt]
&=\nabla_{W}\eta(W(U^{\nu,h}(kh-,y)))G(U^{\nu,h}(kh-,y))h+\mathcal{O}_{\nu}^\eta(h,kh-,y)h,
\end{split}
\end{eqnarray*}
where $\mathcal{O}_{\nu}^\eta(h,kh-,y)\rightarrow 0$ as $h\rightarrow 0$,
and $\mathcal{O}_{\nu}^\eta(h,kh-,y)\equiv 0$ when $\eta=\pm W$.

\par Finally, combining these estimates altogether, we deduce
\begin{eqnarray*}
\mathcal{N}_{1}^{\nu}(\phi)+\mathcal{N}_{2}^{\nu}(\phi)\geq -\frac{O(1)}{2^{\nu}}.
\end{eqnarray*}

Since the total variation of $U^{\nu,h}$ is uniformly bounded, then, by Lemma \ref{lem:8.1} and Helly's theorem, we can
choose a subsequence $U^{\nu_{i}}$ such that
\begin{eqnarray*}
U^{\nu_{i},h}(x,y)\rightarrow U^{h}(x,y), \quad U^{\nu_{i},h}(kh-,y)\rightarrow U^{h}(kh,y)
\qquad \text{a.e.  in $\Omega_h$ as $i\rightarrow \infty$},
\end{eqnarray*}
and define $\mathcal{O}^\eta(h,kh-,y)=\lim_{\nu_i\rightarrow \infty}\mathcal{O}_{\nu_i}^\eta(h,kh-,y)$,
which leads to the desired result.
\end{proof}

Now we are in position to establish an existence theorem for the initial-boundary
value problem \eqref{eq:1.1} and  \eqref{eq:1.8}--\eqref{eq:1.9}.
We can apply again the Helly compactness theorem to
obtain a further subsequence $U^{h_{i}}$
converging to some function $U$ in $\mathbf{L}^{1}_{loc}$
whose total variation in $y$ is uniformly bounded
for all $x\geq0$. Then we have the following proposition.

\begin{proposition}\setcounter{theorem}{0}\label{prop:8.2}
Let $U^{h}$ be a sequence of approximate solutions determined by \eqref{eq:8.15}
with uniformly bounded total variation.
Then there exists a subsequence $h_{i}$ such that $U^{h_i}\rightarrow U$
in $\mathbf{L}^{1}_{loc}$ as $h_{i}\rightarrow 0$, and $U$ is a weak solution to
the  initial-boundary value problem \eqref{eq:1.1} and
\eqref{eq:1.8}--\eqref{eq:1.9} satisfying the entropy inequality{\rm :}
For any convex entropy pair $(\eta, q)$,
\begin{eqnarray} \label{eq:8.17}
\begin{split}
\iint_{\Omega}
\Big(\eta(W(U))\psi_x+q(W(U))\psi_y\Big) dxdy&+\iint_{\Omega}\nabla_{W}\eta(W(U))G(U)\psi dxdy\\[5pt]
&+\int_{y>0}\eta(W(\bar{U}(y)))\psi(0,y) dy\geq0,
\end{split}
\end{eqnarray}
for any $\psi \in C_0^{\infty}(\Omega)$ with $\psi\geq 0$.
\end{proposition}

\begin{proof}
For
any $\psi \in C_0^{\infty}(\Omega_h)$ with $\psi\geq 0$, we have
 \begin{eqnarray} \label{eq:8.19}
\begin{split}
 &\ \ \iint_{\Omega_{h}}\Big(\eta(W(U^{h}))\psi_x+q(W(U^{h}))\psi_y\Big) dxdy
+\iint_{\Omega_{h}}\nabla_{W}\eta(W(U^{h}))G(U^{h})\psi dxdy\\[5pt]
&\ \ \ +\int_{y>0}\eta(W(\bar{U}(y)))\psi(0,y) dy
+\sum_{k\geq1}h\int_{y>g_{k}}\mathcal{O}^\eta(h,kh-,y)\psi(kh,y)dy\\[5pt]
& \geq \iint_{\Omega_{h}}\nabla_{W}\eta(W(U^{h}))G(U^{h})\psi dxdy\\[5pt]
&\ \ \ \ -\sum_{k\geq1}h\int_{y>g_{k}}\nabla_{W} \eta(W(U^{h}(kh-,y)))G(U^{h}(kh-,y))\psi(kh,y)dy\\[5pt]
&=\sum_{k\geq 0}\int^{(k+1)h}_{kh}\int^{\infty}_{g_h(x)}
\Big(\nabla_{W}\eta(W(U^{h}))G(U^{h})\psi\\
&\qquad\qquad\qquad\qquad\qquad
  -\nabla_{W}\eta(W(U^{h}(kh-,y))G(U^{h}(kh-,y))\psi(kh,y)\Big) dydx \\[5pt]
&\ \ \ \ +\sum_{k\geq0}\int^{(k+1)h}_{kh}\int^{g_h(x)}_{g_{k}}\nabla_{W}\eta(W(U^{h}(kh-,y))G(U^{h}(kh-,y))\psi(kh,y)dydx\\[5pt]
&\ \ \ \ -h\int^{\infty}_{0}\nabla_{W}\eta(W(U^{h}(0,y))G(U^{h}(0,y))\psi(0,y)dy\\[5pt]
&=:\Sigma_{1}+\Sigma_{2}-\Sigma_{3}.
 \end{split}
\end{eqnarray}

For $\Sigma_{2}-\Sigma_{3}$,
we have
 \begin{eqnarray*}
\begin{split}
\Big|\Sigma_{2}-\Sigma_{3}\Big|
&\leq\sum_{k\geq0}\Big|\int^{(k+1)h}_{kh}\int^{g_h(x)}_{g_{k}}
  |\nabla_W\eta(W(U^h(kh-,y)))|
   |G(U^{h}(kh-,y))||\psi(kh,y)|dydx\Big|\\[5pt]
&\ \ \  +h\int^{\infty}_{0} |\nabla_W\eta(W(U^h(0,y)))| |G(U^{h}(0,y))||\psi(0,y)|dy\\[5pt]
&\leq O(1)\Big(\frac{{\rm diam}({\rm supp}\, \psi)}{h}\Big)|g_h(x)-g_{k}|h+O(1)h\\[5pt]
&\leq O(1)\big(|\tan(\theta_{\rm crit})|+1\big)h \rightarrow 0\qquad \mbox{as $h\rightarrow 0$}.
\end{split}
\end{eqnarray*}

Now we estimate $\Sigma_{1}$ in \eqref{eq:8.19}. By direct computation, we have
\begin{eqnarray} \label{eq:8.21}
\begin{split}
\Sigma_{1}&=\sum_{k\geq 0}\int^{(k+1)h}_{kh}\int^{\infty}_{0}\Big(\nabla_{W}\eta(W(U^{h}(x,y+g_h(x)))) G(U^{h}(x,y+g_h(x)))\psi(x,y+g_h(x))\\[5pt]
&\ \ \ \ -\nabla_{W}\eta(W(U^{h}(kh-,y+g_h(x))))G(U^{h}(kh-,y+g_h(x)))\psi(kh,y+g_h(x))\Big)dydx\\[5pt]
&=\Sigma^{1}_{1}+\Sigma^{2}_{1},
\end{split}
\end{eqnarray}
where
\begin{eqnarray*}
\begin{split}
\Sigma^{1}_{1}&=\sum_{k\geq 0}\int^{(k+1)h}_{kh}\int^{\infty}_{0}\Big(\nabla_{W}\eta(W(U^{h}(x,y+g_h(x))))
  -\nabla_{W}\eta(W(U^{h}(kh-,y+g_h(x))))\Big)\\[5pt]
&\qquad\qquad\qquad\qquad\quad
\times G(U^{h}(x,y+g_h(x)))\psi(kh,y+g_h(x)) dydx,\\[5pt]
\Sigma^{2}_{1}&=\sum_{k\geq0}\int^{(k+1)h}_{kh}\int^{\infty}_{0}\nabla_{W}\eta(W(U^{h}(kh-,y+g_h(x))))\\[5pt]
&\qquad\qquad\qquad\qquad\,\,\, \times\Big(G(U^{h}(kh+,y+g_h(x)))\psi(x,y+g_h(x))\\[5pt]
&\qquad\qquad\qquad\qquad\quad \, \, \, \,\,\, -G(U^{h}(kh-,y+g_h(x)))\psi(kh,y+g_h(x))\Big) dydx.
\end{split}
\end{eqnarray*}

Now we consider $\Sigma^2_1$.
By direct computation, we have
\begin{eqnarray} \label{eq:8.7a}
\Sigma^2_1=
I_{1}+I_{2}+I_{3},
\end{eqnarray}
where
\begin{eqnarray*}
\begin{split}
&I_{1}=\sum_{k\geq 0}\int^{(k+1)h}_{kh}\int^{\infty}_{0}\nabla_{W}\eta(W(U^{h}(kh-,y+g_h(x))))\\
&\qquad\qquad\quad
\times \Big(G(U^{h}(x,y+g_h(x)))
-G(U^{h}(kh+,y+g_h(x)))\Big)\psi(kh,y+g_h(x)) dydx,\\[5pt]
&I_{2}=\sum_{k\geq 0}\int^{(k+1)h}_{kh}\int^{\infty}_{0}
\nabla_{W}\eta(W(U^{h}(kh-,y+g_h(x))))\\
&\qquad\qquad\quad
\times\Big(G(U^{h}(kh+,y+g_h(x)))
-G(U^{h}(kh-,y+g_h(x)))\Big)\psi(kh,y+g_h(x)) dydx,\\[5pt]
&I_{3}=\sum_{k\geq 0}\int^{(k+1)h}_{kh}\int^{\infty}_{0}
\nabla_{W}\eta(W(U^{h}(kh-,y+g_h(x))))\\
&\qquad\qquad\quad
  \times\Big(\psi(x,y+g_h(x))
-\psi(kh,y+g_h(x))\Big)G(U^{h}(x,y+g_h(x))) dydx.
\end{split}
\end{eqnarray*}
For $I_{1}$, in view of Lemma \ref{lem:8.1}, we have
\begin{eqnarray} \label{eq:8.8}
\begin{split}
\big|I_{1}\big|&\leq O(1)\sum_{0\le k\le K}\int^{(k+1)h}_{kh}\|\psi\|_{\infty}\|\nabla\eta\|_{\infty}\\
&\qquad\qquad\qquad\qquad
\times \Big(\int^{\infty}_{0}|U^{h}(x,y+g_h(x))
-U^{h}(kh+,y+g_h(x))|dy\Big)dx\\[5pt]
&\leq O(1)\sum_{0\le k\le K}\int^{(k+1)h}_{kh}  (x-kh)dx
\leq O(1){\rm diam}({\rm supp}\, \psi)\, h.
\end{split}
\end{eqnarray}
Therefore, $I_{1}\rightarrow 0$ as $h\rightarrow 0$.

\par From the construction of the approximate solutions, we have
\begin{eqnarray} \label{eq:8.9}
\begin{split}
\big|I_{2}\big|&\leq O(1)\sum_{k\geq 0}\int^{(k+1)h}_{kh_{i}}\Big(\int^{\infty}_{0}|U^{h}(kh+,y+g_h(x))
-U^{h}(kh-,y+g_h(x))|\, |\psi| dy\Big)dx\\[5pt]
&\leq O(1)\|\bar{Z}\|_{\infty}\|\psi\|_{\infty} h\sum_{0\le k\le K}\int^{(k+1)h}_{kh}e^{-Lkh}dx
\leq O(1)\|\bar{Z}\|_{\infty}\|\psi\|_{\infty}h\int^{\infty}_{0}e^{-Lx}dx.
\end{split}
\end{eqnarray}
Thus, $I_{2}\rightarrow 0$ as $h\rightarrow 0$.

\par Similarly, for $I_{3}$, we have
\begin{eqnarray} \label{eq:8.10}
\begin{split}
\big|I_{3}\big|&\leq O(1)\sum_{0\le k\le K}\int^{(k+1)h}_{kh}\int^{\infty}_{0}|\psi(x,y+g_h(x))
-\psi(kh,y+g_h(x))|\, dydx\\[5pt]
&\leq O(1)h\sum_{0\le k\le K}\int^{(k+1)h}_{kh}\|\partial_{x}\psi\|_{\infty}\, dx
\leq O(1){\rm diam}({\rm supp}\, \psi)\, h.
\end{split}
\end{eqnarray}
Hence, it follows that $I_{3}\rightarrow 0$ as $h\rightarrow 0$.
Therefore,
$\Sigma^{2}_{1}\rightarrow 0$ as $h\rightarrow 0$.

\par For $\Sigma^{1}_{1}$, we have
\begin{eqnarray} \label{eq:8.23}
\begin{split}
\big|\Sigma^{1}_{1}\big|&\leq O(1)\|\psi\|_{\infty}  \sum_{k\geq 0}\int^{(k+1)h}_{kh}\Big(\int^{\infty}_{0}|U^{h}(x,y+g_h(x))
-U^{h}(kh+,y+g_h(x))|dy\Big)dx\\[5pt]
&\leq O(1){\rm diam}({\rm supp}\, \psi)\,h.
\end{split}
\end{eqnarray}
Thus, $\Sigma^{1}_{1}\rightarrow 0$ as $h\rightarrow 0$.

\par Finally, we also have
\begin{eqnarray} \label{eq:8.24}
\begin{split}
&\ \  \ \Bigg|\iint_{\Omega}
\Big(\eta(W(U^{h}))\psi_x+q(W(U^{h}))\psi_y+\nabla_{W}\eta(W(U^{h})) G(U^{h})\psi\Big) dxdy\\[5pt]
&\quad\,\, -\iint_{\Omega_{h}}
\Big(\eta(W(U^{h}))\psi_x+q(W(U^{h}))\psi_y+\nabla_{W}\eta(W(U^{h})) G(U^{h})\psi\Big) dxdy\Bigg|\\[5pt]
&\leq \iint_{\{(\Omega \setminus\Omega_{h})\cup(\Omega_{h} \setminus\Omega)\}\cap\{x\geq0\}}
\Big(\big|\eta(W(U^{h}))\big|\big|\psi_x\big|+\big|q(W(U^{h}))\big|\big|\psi_y\big|\\
&\qquad\qquad\qquad\qquad\qquad \qquad\,\, +\big|\nabla_{W}\eta(W(U^{h})) G(U^{h})\big|\big|\psi\big|\Big) dxdy\\[5pt]
&\rightarrow 0 \qquad \mbox{as $h\rightarrow 0$}.
\end{split}
\end{eqnarray}
Moreover, by the Helly's compactness theorem, we can find a subsequence $h_{i}\rightarrow 0$ as $i\rightarrow \infty$
such that $U^{h_{i}}\rightarrow U$ in ${L}^{1}_{loc}$ as $h_{i}\rightarrow 0$.
Then, from \eqref{eq:8.19}--\eqref{eq:8.24},
it follows that $U$ satisfies the entropy inequality \eqref{eq:8.17}, which implies that
$U$ is an entropy solution of the initial-boundary value problem \eqref{eq:1.1} and
\eqref{eq:1.8}--\eqref{eq:1.9}.
 This completes the proof.
 \end{proof}

\begin{remark}\label{rem:8.1}
In particular, we choose $\eta(W)=\pm W$ in \eqref{eq:8.17} to conclude
that the limit function $U=U(x,y)$ is a weak solution
of the initial-boundary value problem \eqref{eq:1.1} and \eqref{eq:1.8}--\eqref{eq:1.9}.
\end{remark}

By Corollary \ref{coro:7.1} and Proposition \ref{prop:8.1},
we obtain the following existence theorem.

\begin{theorem}[Existence]\setcounter{theorem}{0}\label{thm:8.1}
Under assumptions $\mathbf{(H_{1})}$--$\mathbf{(H_{2})}$
with $\arctan(g'_{*}(x))\in (\theta_{\rm crit}, 0)$ for $x\in[0, \infty)$,
there exist positive constants $\varepsilon$, $\delta_0$, and $C$ such that, if
\begin{eqnarray}\label{eq:8.13}
\|\bar{U}-U_{\infty}\|_{\infty}+\big|g'_{+}(0)-g'_{*+}(0)\big|
+TV.\{\bar{U}(\cdot)\}+TV.\big\{(g'_{+}-g'_{*+})(\cdot)\big\}<\varepsilon,
\end{eqnarray}
then the initial-boundary value problem \eqref{eq:1.1} and \eqref{eq:1.8}--\eqref{eq:1.9}
admits a global entropy solution $U(x,y)$ that satisfies \eqref{eq:1.11} in the sense of distributions.
The solution is composed of a strong rarefaction wave that is a small perturbation of the complete reaction one.
In addition,
\begin{eqnarray}\label{eq:8.14}
U(x,y)\in BV_{loc}(\Omega)\cap D_{\delta_0}(U_{\infty}),
\end{eqnarray}
and
\begin{eqnarray}\label{eq:8.15a}
TV.\{U(x,\cdot);\, [g(x),\infty)\}\leq C \qquad \mbox{for any $x>0$},
\end{eqnarray}
where
$D_{\delta_0}(U_\infty)$
is  given by \eqref{eq:2.14},
and $\theta_{\rm crit}$ is a critical angle defined by \eqref{eq:2.16}.
\end{theorem}

\section{Asymptotic Behavior of the Solution}\setcounter{equation}{0}

In this section, our main purpose is to investigate the asymptotic behavior of entropy solutions $U(x,y)$.
To achieve this, we need further estimates on $U^{\nu,h}(x,y)$.

\begin{lemma}\label{lem:9.1}
There exists a constant $M_{1}>0$, independent of $(\nu, h)$ and $U^{\nu,h}$, such that
\begin{equation}\label{eq:9.1}
\sum_{\tau>0}E_{ \nu,h}(\tau)<M_{1},
\end{equation}
where the summation is taken over all the interaction times,
and $E_{ \nu,h}(\tau)$ is defined as in \eqref{eq:6.3}.
\end{lemma}

\begin{proof}
First, by Proposition \ref{prop:6.1}, we  know that, for any $\tau \in[kh, (k+1)h),\ k\geq0$,
\begin{eqnarray*}
\begin{split}
F(U^{\nu, h};(k+1)h-)-F(U^{\nu, h};kh+)\leq-\frac{1}{4}\sum^{(k+1)h}_{kh}E_{\nu,h}(\tau).
\end{split}
\end{eqnarray*}
On the other hand, Lemma \ref{lem:7.7} implies that
\begin{eqnarray*}
F(U^{\nu, h};(k+1)h+)-F(U^{\nu, h};(k+1)h-)
\leq C_4\|\bar{Z}\|_{\infty}h\Big(F(U^{\nu, h};(k+1)h-)+4\Big)^2e^{-Lkh}.
\end{eqnarray*}
Then combining these two inequalities and taking the summation $k$ yield
\begin{eqnarray*}
\begin{split}
 \sum^{\infty}_{k=0}\sum^{(k+1)h}_{kh}E_{\nu,h}(\tau)
&\leq\sum^{\infty}_{k=0}\Big(F(U^{\nu,h};kh+)-F(U^{\nu, h};(k+1)h+)\Big)\\
&\ \  \ +C_4\|\bar{Z}\|_{\infty}h\sum^{\infty}_{k=0}\Big(F(U^{\nu, h};(k+1)h-)+4\Big)^{2}e^{-Lkh}\\
&\leq O(1)\big(F(U^{\nu,h};kh+)+C_4\big)<\infty.
\end{split}
\end{eqnarray*}
\end{proof}

Let $L_{\nu,h}\big(\Gamma_{h};[0, \infty)\big)$ be the summation of
the strengths of weak waves leaving the approximate boundary $\Gamma_{h}$.
Then, by Proposition \ref{prop:6.1} and Lemma \ref{lem:7.7}, we have

\begin{lemma}\label{lem:9.2}
There exists $M_{2}>0$, independent of $(\nu, h)$ and $U^{\nu,h}$, such that
\begin{equation}\label{eq:9.2}
L_{\nu,h}\big(\Gamma_{h}; [0,\infty)\big)\leq
M_{2}\Big(\sum_{\tau>0}E_{ \nu,h}(\tau)+\|\bar{Z}\|_{\infty}h\sum^{\infty}_{k=0}e^{-Lkh}\Big).
\end{equation}
\end{lemma}

Define
\begin{equation}\label{eq:9.3}
B_{\nu,h}(x)=U^{\nu, h}(x,g_h(x)).
\end{equation}
Then, by Lemmas \ref{lem:9.1}--\ref{lem:9.2},
we have

\begin{lemma}\label{lem:9.3}
There exists a positive constant $M_{3}$ depending only on the system such that
\begin{equation}\label{eq:9.4}
TV.\big\{B_{\nu,h}(\cdot);\, [0, \infty)\big\}<M_{3},
\end{equation}
which yields a subsequence $({\nu_{i}}, {h_i})$ such that
\begin{eqnarray}\label{eq:9.5}
B_{\nu_i,h_i}(x)\rightarrow B(x)
\qquad \mbox{as $\nu_{i}\rightarrow\infty$ and $h_i\rightarrow 0$}
\end{eqnarray}
for any $x\in[0, \infty)$, where $B(x)\in BV([0, \infty))$ and
\begin{eqnarray} \label{eq:9.26}
\begin{split}
B(x)\cdot(g'(x), -1,0, 0, 0)=0.
\end{split}
\end{eqnarray}
\end{lemma}
Let
\begin{eqnarray*}
&&B_{\infty}=\lim_{x\rightarrow \infty}B(x+),\quad\,  g'_{\infty}=\lim_{x\rightarrow \infty}g'(x+),\\
&& B_{*\infty}=\lim_{x\rightarrow \infty}U^{*}(x,g_{*}(x)),\quad\  g'_{*\infty}=\lim_{x\rightarrow \infty}g'_{*}(x+)
\end{eqnarray*}
with
\begin{eqnarray}\label{eq:9.7}
\begin{split}
B_{*\infty}\cdot(g'_{*\infty}, -1,0, 0, 0)=0,
\end{split}
\end{eqnarray}
where $U^{*}(x,y)$ is the background solution stated in Corollary \ref{cor:3.1}, and $y=g_{*}(x)$ is the non-perturbed
Lipschitz convex wall. Then, owing to the properties of the background solution, we have

\begin{lemma}\label{lem:9.4}
There exists a constant $M_{4}>0$ depending only on the system such that
\begin{eqnarray}\label{eq:9.8}
\begin{split}
g'_{*\infty}-\lambda_{1}(B_{*\infty})>M_{4}.
\end{split}
\end{eqnarray}
\end{lemma}

\begin{proof}
By Corollary \ref{cor:3.1} and \eqref{eq:9.7}, we know that
\begin{eqnarray*}
\begin{split}
g'_{*\infty}=\tan\theta^{*}_{\infty},
\quad \lambda_{1}(B_{*\infty})=\tan(\theta^{*}_{\infty}-\theta^{*}_{{\rm ma}, \infty}).
\end{split}
\end{eqnarray*}
Then
\begin{eqnarray*}
g'_{*\infty}-\lambda_{1}(B_{*\infty})=
\frac{(1+\tan^{2}\theta^{*}_{\infty})\tan\theta^{*}_{{\rm ma}, \infty}}{1+\tan\theta^{*}_{\infty}\tan\theta^{*}_{{\rm ma}, \infty}}.
\end{eqnarray*}
Since $0<\theta^{*}_{{\rm ma}, \infty}<\frac{\pi}{2}$
and $B_{*\infty}\in\{U: u>c_{*},\ q<q_{*}\}$, then
\begin{eqnarray*}
1+\tan\theta^{*}_{\infty}\tan\theta^{*}_{{\rm ma}, \infty}=
\frac{\cos(\theta^{*}_{\infty}-\theta^{*}_{{\rm ma}, \infty})}{\cos\theta^{*}_{\infty}\cos\theta^{*}_{{\rm ma}, \infty}}>0.
\end{eqnarray*}
Therefore, we have
\begin{eqnarray*}
g'_{*\infty}-\lambda_{1}(B_{*\infty})=
\frac{(1+\tan^{2}\theta^{*}_{\infty})\tan\theta^{*}_{{\rm ma}, \infty}}{1+\tan\theta^{*}_{\infty}\tan\theta^{*}_{{\rm ma}, \infty}}>\tan\theta^{*}_{{\rm ma}, \infty},
\end{eqnarray*}
where we have used $\theta_{\rm crit}<\theta^{*}_{\infty}<0$. This completes the proof.
\end{proof}

\begin{lemma}\label{lem:9.5}
For $\varepsilon$ given in Theorem {\rm \ref{thm:8.1}} sufficiently small,
\begin{eqnarray}\label{eq:9.9}
\begin{split}
g'_{\infty}-\lambda_{1}(B_{\infty})>\frac{M_{4}}{2},
\end{split}
\end{eqnarray}
where $M_{4}$ is the same as in \eqref{eq:9.8}.
\end{lemma}

\begin{proof}
By the assumption of Theorem \ref{thm:8.1}, we have
\begin{eqnarray*}
\begin{split}
g'_{\infty}-g'_{*\infty}=O(1)\varepsilon.
\end{split}
\end{eqnarray*}
On the other hand, by \eqref{eq:9.26}--\eqref{eq:9.7}, we have
\begin{eqnarray*}
\theta(B_{\infty})-\theta(B_{*\infty})=O(1)\Big(\arctan(g'_{\infty})-\arctan(g'_{*\infty})\Big)
=O(1)\big(g'_{\infty}-g'_{*\infty}\big)=O(1)\varepsilon.
\end{eqnarray*}
Since
\begin{eqnarray*}
 \Big|\Big(J\big(q,\frac{q^{2}}{2}+\frac{c^{2}}{\gamma-1}\big)+\theta\Big)(B_{\infty})
-\Big(J\big(q,\frac{q^{2}}{2}+\frac{c^{2}}{\gamma-1}\big)+\theta\Big)(B_{*\infty})\Big|=O(1)\varepsilon,
\end{eqnarray*}
and
\begin{eqnarray*}
\begin{split}
&\Big|\Big(\frac{q^{2}}{2}+\frac{c^{2}}{\gamma-1}\Big)(B_{\infty})
   -\Big(\frac{q^{2}}{2}+\frac{c^{2}}{\gamma-1}\Big)(B_{*\infty})\Big|=O(1)\varepsilon,\\[5pt]
&\Big|\frac{p}{{\rho}^{\gamma}}(B_{\infty})-\frac{p}{{\rho}^{\gamma}}(B_{*\infty})\Big|=O(1)\varepsilon,
\end{split}
\end{eqnarray*}
then
\begin{eqnarray*}
\begin{split}
B_{\infty}-B_{*\infty}=O(1)\varepsilon,
\end{split}
\end{eqnarray*}
which leads to
\begin{eqnarray*}
\begin{split}
\lambda_{1}(B_{\infty})-\lambda_{1}(B_{*\infty})=O(1)\varepsilon.
\end{split}
\end{eqnarray*}
Thus, for $\varepsilon$ sufficiently small and $M_{4}$ sufficiently large, we have
\begin{eqnarray*}
g'_{\infty}-\lambda_{1}(B_{\infty})=g'_{*\infty}-\lambda_{1}(B_{*\infty})+O(1)\varepsilon
>M_{4}+O(1)\varepsilon>\frac{M_{4}}{2}.
\end{eqnarray*}
This completes the proof.
\end{proof}

As a direct consequence of Lemma \ref{lem:9.1}, we have
\begin{corollary}\label{coro:9.1}
For any $0<\delta\ll \frac{g'_{\infty}-\lambda_{1}(B_{\infty})}{8}$ sufficiently small,
there exists $x_{\delta}>0$
independent of $(\nu, h)$ such that
\begin{eqnarray}\label{eq:9.10}
\sum_{\tau>x_{\delta}}E_{\nu,h}(\tau)<\delta, \qquad
L_{\nu,h}\big(\Gamma_{h};[x_{\delta},\infty)\big)<\delta,
\end{eqnarray}
where $\delta>0$ depending only on the system.
\end{corollary}

Let $y=\psi^{5}_{g, k_{i}}(x)$ be the approximate $5$-strong rarefaction front generated by the corner point
$A_{k_i}=(x_{k_{i}}, g_h(x_{k_{i}}))$, $k_{i}\geq [\frac{x_{\delta}}{h}]+1$,
which is denoted by $R^{b}_{5, k_i}$.
Let $\Omega_{k_{i}}$
be the region between $R^{b}_{5, k_i}$ and $R^{b}_{5,k_{i}+1}$; see Fig. 9.1.
\vspace{10pt}
\begin{center}\label{fig13}
\begin{tikzpicture}[scale=1.3]
\draw [ultra thick] (-7,0)--(-5.5,0)-- (-3.5,-0.5)--(-2.3,-1);
\draw [thick](-5.5,0) --(-4,1.8);
\draw [thick](-3.5,-0.5) --(-2,1.0);
\node at (-5.5, -0.2) {$A_{k_{i}}$};
\node at (-3.5, -0.7) {$A_{k_{i}+1}$};
\node at (-3.5, 0.8) {$\Omega_{k_{i}}$};
\node at (-1.7, 1.2) {$R^{b}_{5, k_{i}+1}$};
\node at (-3.9, 2) {$R^{b}_{5, k_i}$};
\node [below] at (-4.7, -0.8)
{Fig. 9.1};
\end{tikzpicture}
\end{center}
\hspace{10pt}
Let $L_{\nu, h}(\Omega_{k_{i}})$ be the total amount of the strengths of weak waves entering $\Omega_{k_{i}}$.
Then we have

\begin{lemma}\label{lem:9.6}
For $\varepsilon$ and $\delta$ given as in Theorem {\rm \ref{thm:8.1}} and Corollary {\rm \ref{coro:9.1}},
\begin{eqnarray}\label{eq:9.12}
L_{\nu, h}(\Omega_{k_{i}})=O(1)(\varepsilon+\delta).
\end{eqnarray}
\end{lemma}

\begin{proof}
Using Proposition \ref{prop:6.1} and Lemma \ref{lem:7.7}, we have
\begin{eqnarray*}
\begin{split}
L_{\nu, h}(\Omega_{k_{i}})\leq O(1)\Big(&\sum_{\tau>x_{\delta}}E_{\nu,h}(\tau)
+\|\bar{Z}\|_{\infty}h\sum^{\infty}_{k\geq [\frac{x_{\delta}}{h}]+1}e^{-Lkh}\\
&+L_{\nu, h}\big(\Gamma_{h};[x_{\delta},\infty)\big)+F(U^{v,h};x_{\delta}+)\Big),
\end{split}
\end{eqnarray*}
which leads to the desired result by Theorem \ref{thm:8.1}, Lemma \ref{lem:9.2}, and Corollary \ref{coro:9.1}.
\end{proof}

For any positive constant $M>0$, let
\begin{eqnarray}\label{eq:9.13}
\begin{split}
M_{5}=\inf \Big\{M\,:\,\lambda_{5}(B_{\nu,h}(x))-\lambda_{1}(B_{\nu,h}(x))>M,\  B_{\nu,h}(x)\in D_{\delta_{0}}(U_{\infty})\Big\},
\end{split}
\end{eqnarray}
where $M_{5}>0$ is independent of $(\nu, h)$, and $B_{\nu,h}(x)$ is defined as in \eqref{eq:9.3}.
 Then

\begin{lemma}\label{lem:9.7}
For any points $(x_{j},y_{j})\in \Omega_{k_{i}}$, $j=1, 2$,
\begin{equation}\label{eq:9.14}
\lambda_5(U^{\nu,h}(x_1,y_1))-\lambda_1(U^{\nu,h}(x_2,y_2))>\frac{M_{5}}{2},
\end{equation}
where $M_{5}$ is defined as in \eqref{eq:9.13}.
\end{lemma}

\begin{proof}
First, by Lemma \ref{lem:9.6}, we have
\begin{eqnarray*}
\lambda_5(U^{\nu,h}(x_1,y_1))-\lambda_5(B_{\nu,h}(x_{k_i}))
=O(1) L_{\nu,h}(\Omega_{k_{i}})
=O(1)(\varepsilon+\delta).
\end{eqnarray*}
Similarly,
\begin{eqnarray*}
\lambda_1(U^{\nu,h}(x_2,y_2))-\lambda_1(B_{\nu,h}(x_{k_i}))=O(1)(\varepsilon+\delta).
\end{eqnarray*}
On the other hand, by \eqref{eq:9.13},
\begin{eqnarray*}
\lambda_{5}(B_{\nu,h}(x))-\lambda_{1}(B_{\nu,h}(x))>M_{5}.
\end{eqnarray*}
Then combining above estimates together yields
\begin{eqnarray*}
\begin{split}
\lambda_5(U^{\nu,h}(x_1,y_1))-\lambda_1(U^{\nu,h}(x_2,y_2))&=\lambda_{5}(B_{\nu,h}(x))-\lambda_{1}(B_{\nu,h}(x))
+O(1)(\varepsilon+\delta)\\[3pt]
&=M_{5}+O(1)(\varepsilon+\delta),
\end{split}
\end{eqnarray*}
which leads to the desired result by choosing $\varepsilon$ and $\delta$ sufficiently small.
\end{proof}

Since $TV. \{g'(\cdot)\}<\infty$, then there exists $x_{*}>x_{\delta}$ such that
\begin{eqnarray}\label{eq:9.15}
g'(x+)=g'(\infty)+O(1)\delta \qquad \mbox{for $x>x_{*}$}.
\end{eqnarray}

Let $y=\chi^{5}_{\rm g}(x)$  be the maximal $5$-generalized characteristic
generated by point $(x_{*}, g(x^{*}))$ on the bending wall $y=g(x)$.
Then we have
\begin{lemma}\label{lem:9.8}
For any $x>x_{*}$,
\begin{equation}\label{eq:9.16}
\sup \Big\{\big|U(x,y)-B(x)\big|;\ \big( g(x), \ \chi^{5}_{\rm g}(x)\big) \Big\}=O(1)\big(\varepsilon+\delta \big).
\end{equation}
\end{lemma}

\begin{proof}
According to the construction of the approximate solutions, there exists a subsequence
of approximate maximal $5$-generalized characteristics $\chi^{5}_{{\rm g}, \nu_i, h_i}(x)$
such that $\chi^{5}_{{\rm g}, \nu_i, h_i}(x)\to \chi^{5}_{\rm g}(x)$
uniformly on every bounded interval as $\nu_i\to \infty$ and $h_i\to 0$.
Then
\begin{eqnarray*}
\begin{split}
&\sup \Big\{\big|U^{\nu_i, h_i}(x,y)-B_{\nu_i, h_i}(x)\big|;\ \big(g_{h_i}(x), \chi^{5}_{g, \nu_i, h_i}(x)\big) \Big\}\\
&\leq O(1)\Big(\sum_{\tau>x_{*}}E_{\nu_i, h_i}(\tau)
+\|\bar{Z}\|_{\infty}h_{i}\sum^{\infty}_{k\geq [\frac{x_{*}}{h_i}]+1}e^{-Lkh_i}\\
&\qquad\qquad +L_{\nu_i, h_i}\big(\Gamma_{h_i};[x_{*},\infty)\big)+F(U^{\nu_i, h_i};x_{*}+)\Big)\\[3pt]
&=O(1)\big(\varepsilon+\delta \big).
\end{split}
\end{eqnarray*}
Then, passing to the limits for $(\nu_i, h_i)$ in the above,
we obtain the expected result.
\end{proof}

\begin{lemma}\label{lem:9.9}
For $g(x)\leq y\leq \chi^{5}_{\rm g}(x)$, when $\delta$ is sufficiently small,
\begin{eqnarray}\label{eq:9.17}
\lambda_{1}(U(x,y))-g'_{\infty}<-\frac{M_{4}}{8} \qquad\,\, \mbox{for $x>x_{*}$},
\end{eqnarray}
where $M_{4}$ is given by \eqref{eq:9.8}.
\end{lemma}

\begin{proof}
By Lemma \ref{lem:9.7} and the construction of the approximate solutions, we deduce
\begin{eqnarray*}
&&\lambda_{1}(U^{\nu, h}(x,y))-\lambda_{1}(B_{\nu, h}(x))=O(1)\big(\varepsilon+\delta \big),\\[2mm]
&&\lambda_{1}(B^{\nu, h}(x))-\lambda_{1}(B_{\infty})=O(1)L_{\nu, h}\big(\Gamma_{h};[x_{*},\infty)\big)
   =O(1)\big(\varepsilon+\delta \big).
\end{eqnarray*}
Then, by Lemma \ref{lem:9.5}, choosing $\delta$ sufficiently small, we have
\begin{eqnarray*}
\begin{split}
&\lambda_{1}(U^{\nu, h}(x))-\lambda_{1}(B_{\infty})\\
&=\big(\lambda_{1}(U^{\nu, h}(x))-\lambda_{1}(B^{\nu, h}(x))\big)
   +\big(\lambda_{1}(B^{\nu, h}(x))-\lambda_{1}(B_{\infty})\big)
   +\big(\lambda_{1}(B_{\infty})-g'_{\infty}\big)\\
&=\lambda_{1}(B_{\infty})-g'_{\infty}+O(1)\big(\varepsilon+\delta \big)
  <-\frac{g'_{\infty}-\lambda_{1}(B_{\infty})}{2}+O(1)\varepsilon\\
&<-\frac{M_{4}}{4}+O(1)\varepsilon<-\frac{M_{4}}{8}.
\end{split}
\end{eqnarray*}
Finally, passing to the limits for $U^{\nu,h}$ in the above by choosing a subsequence denoted by itself,
we conclude the proof.
\end{proof}

\begin{lemma}\label{lem:9.10}
$U(x,y)$ and $\chi^{5}_{\rm g}(x)$ satisfy
\begin{eqnarray}
\lim_{x \rightarrow \infty} TV.\Big\{(\frac{v}{u},p)(x,\cdot)\,; \, (g(x),\chi^{5}_{\rm g}(x))\Big\}=0. \label{eq:9.18}
\end{eqnarray}
\end{lemma}

\begin{proof}
Let $U^{\nu_{i},h_{i}}$ be the sequence of approximate solutions stated
in Propositions \ref{prop:8.1}--\ref{prop:8.2},
and let the corresponding
term $E_{\nu_i, h_i}(\tau)$ be the quantity defined in \eqref{eq:6.3}.
As in \cite{gl}, we denote $dE_{\nu_i, h_i}(\tau)$ as the measures assigning
$E_{\nu_i, h_i}(\tau)$
at the interaction line $x=\tau$.
Then, by Lemma \ref{lem:9.1}, we can select a subsequence (still denoted by) $E_{\nu_i,h_i}(\tau)$ such that
\begin{eqnarray*}
dE_{ \nu_i,h_i}(\tau)\rightarrow  dE(\tau) \qquad \mbox{as $\nu_i\rightarrow \infty$ and $h_i\rightarrow 0$}
\end{eqnarray*}
with $E(\tau)<\infty$.
Therefore, for $\hat{\delta}<\delta$ sufficiently small, we can choose $x_{\hat{\delta}}>x_{*}$
independent of $U^{\nu_i, h_i}$ and $(\nu_i, h_i)$ such that
\begin{eqnarray*}
\sum _{\tau>x_{\hat{\delta}}}E_{\nu_i, h_i}(\tau)<\hat{\delta}.
\end{eqnarray*}
Let $X^{1}_{\hat{\delta}}=(x_{\hat{\delta}}, \chi^{5}_{g, \nu_{i}, h_i}(x_{\hat{\delta}}))$
and $X^{5}_{\hat{\delta}}=(x_{\hat{\delta}}, g_h(x_{\hat{\delta}}))$
be the two points lying in the approximate $5$-maximal characteristic
$y=\chi^{5}_{g,\nu_{i}, h_{i}}(x)$ and the approximate boundary $y=g_h(x)$, respectively.
Let $\chi^{j}_{\nu_{i}, h_{i}}$ be the approximate $j$-generalized characteristic
generated  from $X^{j}_{\hat{\delta}}$
for $j=1, 5$, respectively.
According to the construction of the approximate solutions,
there exist constants $\hat{M}_{j}>0, j=1,5$, independent of $U^{\nu_i,h_i}$ and $(\nu_i, h_i)$
such that
\begin{eqnarray*}
\big|\chi^{j}_{ \nu_{i}, h_{i}}(x_1)-\chi^{j}_{\nu_{i}, h_{i}}(x_2)\big|\leq \hat{M}_{j}
\Big(|x_1-x_2|+h_{i}+\frac{1}{2^{\nu_i}}\Big)\qquad \mbox{for $x_1, x_2>x_{\hat{\delta}}$}.
\end{eqnarray*}
Then we can choose a subsequence (still denoted by) $\{\nu_i\}$ and $\{h_i\}$ such that
\begin{eqnarray*}
\chi^{j}_{\nu_{i}, h_{i}}(x)\rightarrow \chi^{j}(x)\qquad \mbox{as $i\rightarrow \infty$}
\end{eqnarray*}
for some $\chi^{j}\in {\rm Lip}$ with $(\chi^{j})'$ bounded.
\vspace{10pt}
\begin{center}\label{fig14}
\begin{tikzpicture}[scale=1.15]
\draw [line width=0.1cm](-3.0, 0)to [out=0,in=145](3.0,-1.3);
\draw [ line width=0.05cm](-1.5,0)--(2.2, 3.2);
\draw [thick][dashed](0.2,-1.2)--(0.2, 3.0);
\draw [blue][ thick](0.2,1.47)to [out=-10, in=120](2.7, -1.1);
\draw [red][ thick](0.2,-0.14)to [out=20, in=-100](1.8, 2.85);
\node at (1.1, 3) {$(t^{5}_{\hat{\delta}}, \chi^{5}(t^{5}_{\hat{\delta}}))$};
\node at (3.7,-1.0) {$(t^{1}_{\hat{\delta}}, \chi^{1}(t^{1}_{\hat{\delta}}))$};
\node at (-0.4, -0.5) {$x=x_{\hat{\delta}}$};
\node at (1.5, -0.8) {$\Gamma$};
\node at (3, 3.4 ) {$y=\chi^{5}_{\rm g}(x)$};
\node at (2.4, 2.2 ) {$y=\chi^{5}(x)$};
\node at (2.8, 0.1 ) {$y=\chi^{1}(x)$};
\node [below] at (0.3, -1.4)
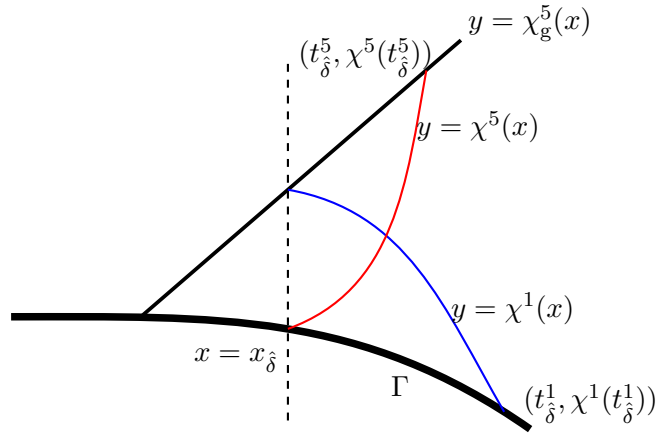
{Fig. 9.2. The characteristics $\chi^{1}$ and $\chi^{5}$
intersect with $\Gamma$ and $y=\chi^{5}_{\rm g}(x)$, respectively};
\end{tikzpicture}
\end{center}

According to Lemma \ref{lem:9.9}, $\chi^{1}$ can intersect with boundary $\Gamma$,
whose intersection point is denoted by $(t^{1}_{\hat{\delta}},\chi^{1}(t^{1}_{\hat{\delta}}))$.
Also, let $\chi^{5}$ intersect with $y=\chi^{5}_{\rm g}(x)$
at point $(t^{5}_{\hat{\delta}}, \chi^{5}(t^{5}_{\hat{\delta}}))$;
see Fig. 9.2.
Since
$\frac{v^{\nu_{i}, h_{i}}}{u^{\nu_{i}, h_{i}}}$
and
$p^{\nu_{i}, h_{i}}$
are invariant across the $k$-contact discontinuities for $k=2,3,4$.
Then we can do as in \cite{gl,l} by applying the approximate conservation laws,
we deduce that, for $x>2\big(t^{1}_{\hat{\delta}}+t^{5}_{\hat{\delta}}\big)$,
\begin{eqnarray*}
TV.\Big\{(\frac{v^{\nu_i, h_i}}{u^{\nu_i, h_i}}, p^{ \nu_i, h_i})(x,\cdot)\,; \ (g(x),\ \chi^{5}_{\rm g}(x))\Big\}
\leq C_{8}\hat{\delta},
\end{eqnarray*}
where the positive constant $C_{8}$ is independent of $(\hat{\delta}, \nu_i, h_i)$ and $U^{\nu_i, h_i}$.
Therefore, by Proposition \ref{prop:8.1}--\ref{prop:8.2} and
taking the limit as $\nu_i\rightarrow \infty$ and $h_i\rightarrow 0$,
it follows that
\begin{eqnarray*}
TV.\Big\{(\frac{v}{u},p)(x,\cdot);\, (g(x),\ \chi^{5}_{\rm g}(x))\Big\}\leq C_{8}\hat{\delta}
\qquad\,\, \mbox{for $x>2(t^{1}_{\hat{\delta}}+t^{5}_{\hat{\delta}})$}.
\end{eqnarray*}
\end{proof}

By similar arguments for $y>\chi^{5}_{\rm g}(x)$, we have
\begin{lemma}\label{lem:9.11}
The following hold{\rm :}
\begin{eqnarray}\label{eq:9.20}
\begin{split}
&\lim_{x \rightarrow \infty} TV.\Big\{\big(J(q, \frac{q^{2}}{2}+\frac{c^{2}}{\gamma-1})+\theta,
 \frac{q^{2}}{2}+\frac{c^{2}}{\gamma-1}\big)(U(x,\cdot));\,
    (\chi^{5}_{\rm g}(x), \infty )\Big\}=0,\\[2mm]
&\lim_{x \rightarrow \infty} TV.\Big\{(\frac{p}{\rho^{\gamma}}, Z)(x,\cdot);\ (\chi^{5}_{\rm g}(x),\ \infty )\Big\}=0.
\end{split}
\end{eqnarray}
\end{lemma}

\begin{theorem}[Asymptotic behavior]\label{thm:9.1}\setcounter{theorem}{0}
Let $\chi^{5}_{\rm g}(x)$ be stated as above.
\begin{enumerate}
\item[\rm (i)] Let $g'_{\infty}=\lim_{x\rightarrow \infty}g'(x+)$. Then
\begin{eqnarray}\label{eq:9.21}
\begin{split}
\lim_{x\rightarrow \infty}\sup \Big\{\big|\frac{v}{u}(x,\cdot)-g'_{\infty}\big|;\ \big(g(x),\ \chi^{5}_{\rm g}(x)\big)\Big \}=0,
\end{split}
\end{eqnarray}
and there exists a constant $p_{+}$ such that
\begin{eqnarray}\label{eq:9.22}
\begin{split}
\lim_{x\rightarrow \infty}\sup \Big\{\big|p(x,\cdot)-p_{+}\big|;\ (g(x),\ \chi^{5}_{\rm g}(x))\Big \}=0.
\end{split}
\end{eqnarray}
\item[\rm (ii)]  There exists a constant state $U^{\infty}=(u^{\infty}, v^{\infty}, p^{\infty}, \rho^{\infty}, 0)$ such that
\begin{eqnarray}\label{eq:9.23}
\begin{split}
&\lim_{x \rightarrow \infty} \sup\Big\{\Big|\big(J(q, \frac{q^{2}}{2}+\frac{c^{2}}{\gamma-1})+\theta\big)(U(x,\cdot))
-\big(J(q, \frac{q^{2}}{2}+\frac{c^{2}}{\gamma-1})+\theta\big)(U^{\infty})\Big|;\  (\chi^{5}_{\rm g}(x),\infty )\Big\}\\
&\qquad =0,\\[1mm]
&\lim_{x \rightarrow \infty} \sup\Big\{\Big|\big(\frac{q^{2}}{2}+\frac{c^{2}}{\gamma-1}\big)(U(x,\cdot))
-\big(\frac{q^{2}}{2}+\frac{c^{2}}{\gamma-1}\big)(U^{\infty})\Big|;\ (\chi^{5}_{\rm g}(x),\ \infty )\Big\}=0,\\[1mm]
&\lim_{x \rightarrow \infty} \sup\Big\{\Big|(\frac{p}{\rho^{\gamma}})(x,\cdot)
-\frac{p^{\infty}}{(\rho^{\infty})^{\gamma}}\Big|
+\big|Z(x,\cdot)\big|;\  (\chi^{5}_{\rm g}(x),\  \infty )\Big\}=0.
\end{split}
\end{eqnarray}
\end{enumerate}

\end{theorem}
\begin{proof}
Let $U^{\nu_{i}, h_{i}}$  be the approximate solutions of problem \eqref{eq:1.1} and \eqref{eq:1.8}--\eqref{eq:1.9},
where $({\nu_{i}}, {h_{i}})$ is the sequence given as
in Propositions \ref{prop:8.1}--\ref{prop:8.2}.
From the construction of the approximate solutions, we have
\begin{eqnarray*}
\left.\frac{v^{\nu_{i}, h_{i}}}{u^{\nu_{i}, h_{i}}}\right|_{\Gamma_{h_i}}=g'_{h_i}(x)
\qquad\,\, \mbox{for some $x\in (kh_i, (k+1)h_i)$}.
\end{eqnarray*}
Then we can choose some $x_{\hat{\delta}}>x_{*}$ such that, for $x>2x_{\hat{\delta}}$,
\begin{eqnarray*}
\big|g'_{h_i}(x)-g'_{\infty}\big|<\frac{\hat{\delta}}{2}.
\end{eqnarray*}
Thus, by Lemma \ref{lem:9.10}, we have
\begin{eqnarray*}
\begin{split}
&\sup\big\{\big|\frac{v^{\nu_i, h_i}}{u^{\nu_i, h_i}}(x,y)
-g'_{\infty}\big|;\ (g_{h_i}(x),\ \chi^{5}_{g, \nu_{i}, h_i}(x))\big \}\\
&\leq TV.\{\frac{v^{ \nu_i, h_i}}{u^{ \nu_i, h_i}}(x,\cdot);\  (g_{h_i}(x),\ \chi^{5}_{g, \nu_{i}, h_i}(x))\}
+\big|g'_{h_i}(x)-g'_{\infty}\big|\leq \hat{\delta}\qquad\,\,\mbox{for $x>2x_{\hat{\delta}}$}.
\end{split}
\end{eqnarray*}
Then, letting $i\rightarrow \infty$
and by  Propositions \ref{prop:8.1}--\ref{prop:8.2},
we obtain \eqref{eq:9.21}.

To prove \eqref{eq:9.22}, let $p_{+}=\lim _{x\rightarrow \infty}p(x,g(x))$. Then
\begin{eqnarray*}
\begin{split}
&\sup \Big\{|p^{\nu_{i},h_i}(x,\cdot)-p_{+}|;\ (g_{h_i}(x),\ \chi^{5}_{g, \nu_{i}, h_i}(x))\Big \}\\
&\leq T.V.\Big\{p^{\nu_{i},h_i}(x,\cdot); (g_{h_i}(x),\chi^{5}_{g, \nu_{i}, h_i}(x) )\Big\}
+\big|p^{\nu_{i},h_i}(x,g_{h_i}(x))-p_{+}\big|\leq \hat{\delta}.
\end{split}
\end{eqnarray*}
Thus, for $x>2x_{\hat{\delta}}$, letting $i\rightarrow \infty$,
and by  Propositions \ref{prop:8.1}--\ref{prop:8.2},
we can obtain the desired result. 

\smallskip
Now we consider {\rm (ii)}.
Let $U^{\infty}=\lim_{x\rightarrow \infty}U(x,\infty)$.
Then
\begin{eqnarray*}
\begin{split}
&\sup\Big|\Big(J\big(q, \frac{q^{2}}{2}+\frac{c^{2}}{\gamma-1}\big)+\theta\Big)(U^{\nu_{i}, h_i}(x,\cdot))
-\Big(J\big(q, \frac{q^{2}}{2}+\frac{c^{2}}{\gamma-1}\big)+\theta\Big)(U^{\infty})\Big| \\
&\leq TV.\Big\{\Big(J\big(q, \frac{q^{2}}{2}+\frac{c^{2}}{\gamma-1}\big)+\theta\Big)(U^{\nu_{i}, h_i}(x,\cdot));\
  (\chi^{5}_{g, \nu_{i}, h_i}(x), \infty )\Big\}\\
&\ \ \  +\Big|\Big(J(q, \frac{q^{2}}{2}+\frac{c^{2}}{\gamma-1})+\theta\Big)(U^{\nu_{i}, h_i}(x,\infty))
-\Big(J\big(q, \frac{q^{2}}{2}+\frac{c^{2}}{\gamma-1}\big)+\theta \Big)(U^{\infty})\Big|\leq \hat{\delta}
\end{split}
\end{eqnarray*}
 for $x>2x_{\hat{\delta}}$.
 Then, letting $i\rightarrow \infty$,
 and by Propositions \ref{prop:8.1}--\ref{prop:8.2}, we obtain the desire result.
The proof of the others in \eqref{eq:9.22} is similar. This completes the proof.
\end{proof}

\section{Appendix}\setcounter{equation}{0}
In this section, we complete the proof of Lemma \ref{lem:2.1} by several lemmas.
 \begin{lemma}\label{lem:10.1}
For $u>c$,
 \begin{eqnarray}\label{eq:10.1}
\cos\big(\theta+(-1)^{\frac{j+3}{4}}\theta_{\rm ma}\big)>0 \qquad \mbox{for $j=1, 5$}.
\end{eqnarray}
\end{lemma}

\begin{proof}
Since $\big(u\sqrt{q^{2}-c^{2}}\big)^{2}-(vc)^{2}=q^{2}(u^{2}-c^{2})>0$.
Then, by direct computation, we have
\begin{equation*}
 q^{2}\cos(\theta+(-1)^{\frac{j+3}{4}}\theta_{\rm ma})=u\sqrt{q^{2}-c^{2}}-(-1)^{\frac{j+3}{4}}vc>0, \qquad j=1, 5.
\end{equation*}
\end{proof}

\begin{lemma}\label{lem:10.2}
For $u>c$,
\begin{eqnarray*}
\begin{split}
&\nabla_{(u,v)}q=(\cos\theta,\sin\theta),\quad
  \nabla_{(u,v)}\theta=\frac{1}{q}(-\sin\theta,\cos\theta),\\
&\nabla_{(u,v)}\theta_{\rm ma}=-\frac{\tan(\theta_{\rm ma})}{q}(\cos\theta,\sin\theta),\quad
\nabla_{(p,\rho)}\theta_{\rm ma}
=\frac{\tan(\theta_{\rm ma})}{2\rho c^2}(\gamma,-c^2).
\end{split}
\end{eqnarray*}
\end{lemma}
\begin{proof}
Since $\theta_{\rm ma}=\arctan(\frac{c}{\sqrt{q^{2}-c^{2}}})$,
then
\begin{eqnarray*}
\frac{\partial\theta_{\rm ma}}{\partial u}&=\frac{q^{2}-c^{2}}{q^{2}}\frac{\partial}{\partial u}\Big(\frac{c}{\sqrt{q^{2}-c^{2}}}\Big)
=-\frac{c}{q\sqrt{q^{2}-c^{2}}}\cos\theta
=-\frac{\cos\theta\tan(\theta_{\rm ma})}{q}.
\end{eqnarray*}
Similarly for $\frac{\partial\theta_{\rm ma}}{\partial v}$. Notice that $c^{2}=\frac{\gamma p}{\rho}$, then
$\frac{\partial c}{\partial p}=\frac{\gamma}{2\rho c}$ and $\frac{\partial c}{\partial\rho}=-\frac{c}{2\rho}$.
Then
\begin{eqnarray*}
&&\frac{\partial\theta_{\rm ma}}{\partial p}=\frac{q^{2}-c^{2}}{q^{2}}
\frac{\partial}{\partial p}\Big(\frac{c}{\sqrt{q^{2}-c^{2}}}\Big)
=\frac{q^{2}-c^{2}}{q^{2}}\frac{q^{2}}{(q^{2}-c^{2})^{\frac{3}{2}}}\frac{\partial c}{\partial p}
=\frac{\gamma}{2\rho c^{2}}\tan(\theta_{\rm ma}),\\[1mm]
&&\frac{\partial\theta_{\rm ma}}{\partial \rho}=\frac{q^{2}-c^{2}}{q^{2}}
\frac{\partial}{\partial \rho}\Big(\frac{c}{\sqrt{q^{2}-c^{2}}}\Big)
=\frac{q^{2}-c^{2}}{q^{2}}\frac{q^{2}}{(q^{2}-c^{2})^{\frac{3}{2}}}\frac{\partial c}{\partial \rho}
=-\frac{1}{2\rho}\tan(\theta_{\rm ma}).
\end{eqnarray*}
The others can be proved in the same way. This completes the proof.
\end{proof}

\begin{lemma}\label{lem:10.3}
For $u>c$,
\begin{eqnarray}
&&\nabla_{(u,v)}\lambda_{j}
  =\frac{\sec^{2}(\theta+(-1)^{\frac{j+3}{4}}\theta_{\rm ma})}{\sqrt{q^{2}-c^{2}}}
   (-\sin(\theta +(-1)^{\frac{j+3}{4}}\theta_{\rm ma}),
    \cos(\theta+(-1)^{\frac{j+3}{4}}\theta_{\rm ma})),\nonumber\\
    \label{eq:10.3}\\
&&\nabla_{(p,\rho)}\lambda_{j}
  =\frac{(-1)^{\frac{j+3}{4}}}{2\rho c^{2}}
   \sec^{2}(\theta+(-1)^{\frac{j+3}{4}}\theta_{\rm ma})\tan(\theta_{\rm ma})(\gamma, -c^2),\label{eq:10.4}\\[0.5mm]
&&\frac{\partial \lambda_{j}}{\partial Z}=0,\label{eq:10.6}
\end{eqnarray}
for $j=1,5$, and
\begin{eqnarray}
\nabla_{(u,v)}\lambda_{i}
=\frac{\sec^2\theta\sin\theta}{q}(-1,1),\quad
\frac{\partial \lambda_{i}}{\partial p}=\frac{\partial \lambda_{i}}{\partial \rho}
=\frac{\partial \lambda_{i}}{\partial Z}=0 \qquad\,\,\, \mbox{for $i=2, 3, 4$}.
\end{eqnarray}
\end{lemma}

\begin{proof}
We only consider the case $j=1$, since $j=2$ and $i=2, 3, 4$ can be carried out in the same way.
By Lemma \ref{lem:10.2}, we have
\begin{eqnarray*}
\begin{split}
\frac{\partial \lambda_{1}}{\partial u}&=\frac{\partial \lambda_{1}}{\partial \theta}\frac{\partial \theta}{\partial u}
+\frac{\partial \lambda_{1}}{\partial\theta_{\rm ma} }\frac{\partial \theta_{\rm ma}}{\partial u}
=\sec^{2}(\theta-\theta_{\rm ma})\big(-\frac{\sin\theta}{q}\big)
+\frac{\cos\theta\tan(\theta_{\rm ma})}{q}\sec^2(\theta-\theta_{\rm ma})\\
&=-\frac{1}{\sqrt{q^{2}-c^{2}}}\sec^2(\theta-\theta_{\rm ma})\sin(\theta -\theta_{\rm ma}),\\[1mm]
\frac{\partial \lambda_{1}}{\partial v}&=\frac{\partial \lambda_{1}}{\partial \theta}
\frac{\partial \theta}{\partial v}+\frac{\partial \lambda_{1}}{\partial \theta_{\rm ma}}\frac{\partial \theta_{\rm ma}}{\partial v}
=\sec^2(\theta-\theta_{\rm ma})\frac{\cos\theta}{q}
+\frac{\sin\theta\tan(\theta_{\rm ma})}{q}\sec^2(\theta-\theta_{\rm ma})\\
&=\frac{1}{\sqrt{q^{2}-c^{2}}}\sec^2(\theta-\theta_{\rm ma})\cos(\theta -\theta_{\rm ma}),\\[1mm]
\frac{\partial \lambda_{1}}{\partial p}&=\frac{\partial \lambda_{1}}{\partial\theta_{\rm ma} }\frac{\partial \theta_{\rm ma}}{\partial p}
=-\frac{\gamma}{2\rho c^{2}}\sec^{2}(\theta-\theta_{\rm ma})\tan(\theta_{\rm ma}),\\[1mm]
\frac{\partial \lambda_{1}}{\partial \rho}&=\frac{\partial \lambda_{1}}{\partial\theta_{\rm ma} }\frac{\partial \theta_{\rm ma}}{\partial \rho}
=\frac{1}{2\rho}\sec^{2}(\theta-\theta_{\rm ma})\tan(\theta_{\rm ma}),
\end{split}
\end{eqnarray*}
and clearly 
$\frac{\partial \lambda_{j}}{\partial Z}=0.$
This completes the proof.
\end{proof}

\smallskip
$\textbf{Proof}\ \textbf{of}\ \textbf{Lemma}\ \ref{lem:2.1}$.
We consider Case $j=1$, since Case $j=5$
can be done in the same way.
By Lemma \ref{lem:10.1} and direct computation, we have
 \begin{eqnarray*}
\begin{split}
\nabla_{U}\lambda_{1}\cdot\tilde{r}_{1}&=-\frac{\partial \lambda_{1}}{\partial u}\tan(\theta-\theta_{\rm ma})
+\frac{\partial \lambda_{1}}{\partial v}
-\frac{\partial \lambda_{1}}{\partial p} \rho q\sec(\theta-\theta_{\rm ma})\sin(\theta_{\rm ma})\\
&\ \ \ -\frac{\partial \lambda_{1}}{\partial \rho}\frac{\rho q}{c^{2}}\sec(\theta-\theta_{\rm ma})\sin(\theta_{\rm ma})\\
&=\frac{1}{\sqrt{q^{2}-c^{2}}}\sec^{2}(\theta-\theta_{\rm ma})\sin(\theta -\theta_{\rm ma})\tan(\theta-\theta_{\rm ma})\\
&\ \ \ +\frac{1}{\sqrt{q^{2}-c^{2}}}\sec^{2}(\theta-\theta_{\rm ma})\cos(\theta -\theta_{\rm ma})\\
&\ \ \ +\frac{\gamma}{2\rho c^{2}}\sec^{2}(\theta-\theta_{\rm ma})\tan(\theta_{\rm ma})\,
   \rho q\sec(\theta-\theta_{\rm ma})\sin(\theta_{\rm ma})\\
&\ \ \ -\frac{1}{2\rho}\sec^2(\theta-\theta_{\rm ma})\tan(\theta_{\rm ma})\,\frac{\rho q}{c^{2}}\sec(\theta-\theta_{\rm ma})\sin(\theta_{\rm ma})\\
&=\frac{\gamma+1}{\sqrt{q^{2}-c^{2}}}\sec^{3}(\theta-\theta_{\rm ma})>0.
\end{split}
\end{eqnarray*}
Similarly, for $j=5$,
$$
\nabla_{U}\lambda_{5}\cdot\tilde{r}_{5}
=\frac{\gamma-1}{\sqrt{q^{2}-c^{2}}}\sec^{3}(\theta+\theta_{\rm ma})>0.
$$
In a similar way, we can prove
\begin{eqnarray*}
\begin{split}
\nabla_{U}\lambda_{j}\cdot\tilde{r}_{j}=0, \qquad j=2, 3, 4.
\end{split}
\end{eqnarray*}
This completes the proof.

\bigskip
$\textbf{Acknowledgements}.$
The research of Gui-Qiang Chen was supported in part by
	the UK EPSRC Award to the EPSRC Centre for Doctoral Training
	in PDEs (EP/L015811/1) and
	the Royal Society--Wolfson Research Merit Award (UK).
Yongqian Zhang was supported in part
by NSFC Project 11421061, NSFC Project 11031001, NSFC Project 11121101,
the 111 Project B08018 (China),
and  by Natural Science Foundation of Shanghai 15ZR1403900.

\bigskip
$\textbf{Conflict of Interest}.$  The authors declare that they have no conflict of interest.

\end{document}